\documentclass[10pt,oneside,a4paper]{article}
\usepackage{titlesec}
	\titleformat{\subsection}[runin]{\normalfont\bfseries}{\thesubsection.}{.5em}{}[. ~]
		\titlespacing{\subsection}{0pt}{1.5ex plus .1ex minus .2ex}{0pt}
\usepackage{graphicx}
\usepackage{subcaption}
\usepackage{wrapfig}
\usepackage{amsfonts,amssymb,amsmath,amsthm}
\usepackage{enumerate}
\usepackage{url}
\usepackage{mathrsfs}
\usepackage[dvipsnames,hyperref]{xcolor}
\usepackage{hyperref}
\usepackage{tikz-cd}
\hypersetup{colorlinks=true,
citecolor=OliveGreen,
linkcolor=Purple,
urlcolor=RedOrange}
\usepackage{mathtools} %needed for declarepaireddelimiter and \coloneq
\usepackage[normalem]{ulem} %strike through text
\DeclareFontEncoding{LS1}{}{}
\DeclareFontSubstitution{LS1}{stix}{m}{n}
\DeclareSymbolFont{symbols2}{LS1}{stixfrak}{m}{n}

\DeclareMathSymbol{\lcurvyangle}{\mathopen}{symbols2}{"E9}
\DeclareMathSymbol{\rcurvyangle}{\mathclose}{symbols2}{"EA}
\DeclarePairedDelimiter{\fgeo}{\wr}{\wr}
\DeclarePairedDelimiter{\geo}{\lcurvyangle}{\rcurvyangle}

\usepackage[mono=false]{libertine}
\usepackage{libertinust1math}
\usepackage[T1]{fontenc}

\newtheoremstyle{thmstyle}
{\topsep}%space above
{\topsep}%space below
{\itshape}%bodyfont
{0pt}%indent
{\bfseries}%headfont
{.}%punctuation
{5pt plus 1pt minus 1pt}%space after head
{#2.\hspace{3pt}#1\thmnote{~\textnormal{(#3)}}}%theoremheadspec

\newtheoremstyle{defistyle}
{\topsep}
{\topsep}
{}
{0pt}
{\bfseries}
{.}
{5pt plus 1pt minus 1pt}
{}

\theoremstyle{thmstyle}
\newtheorem{thm}[subsection]{Theorem}
\newtheorem{lemma}[subsection]{Lemma}
\newtheorem{prop}[subsection]{Proposition}
\newtheorem{cor}[subsection]{Corollary}

\theoremstyle{defistyle}
\newtheorem{defi}[subsection]{Definition}
\newtheorem*{rmk}{Remark}

\newcommand{\HH}{\mathbb{H}}
\newcommand{\CC}{\mathbb{C}}
\newcommand{\RR}{\mathbb{R}}
\newcommand{\KK}{\mathbb{K}}

\usepackage{float}
\newcommand{\trace}{\mathop{\mathrm{tr}}}
\newcommand{\tance}{\mathop{\mathrm{ta}}}

\newcommand{\SU}{\mathop{\mathrm{SU}}}
\newcommand{\PU}{\mathop{\mathrm{PU}}}
\newcommand{\BV}{{\mathrm B}\,V}

\newcommand{\SV}{{\mathrm S}\,V}
\newcommand{\EV}{{\mathrm E}\,V}
\newcommand{\real}{\mathop{\mathrm{Re}}}
\newcommand{\imag}{\mathop{\mathrm{Im}}}
\newcommand{\PV}{\mathbb{P}V}
\newcommand{\dist}{\mathop{\mathrm{dist}}}

\newcommand{\HC}{\mathbb{H}^2_{\mathbb{C}}}
\newcommand{\CHC}{\overline{\mathbb{H}^2_{\mathbb{C}}}}
\newcommand{\Id}{\mathop{\mathrm{Id}}}
\renewcommand{\triangle}{\vartriangle\!\!}

\begin{document}
\title{A non-trivial family of trivial bundles with complex hyperbolic structure}
\author{Hugo C. Bot\'os\footnote{Supported by grant 2023/07381-1, São Paulo Research Foundation (FAPESP)} \and Felipe A. Franco}
\date{}
\maketitle

\begin{abstract}
In $ \PU(2,1)$, the group of holomorphic isometries of the complex hyperbolic plane, we study the space of involutions $R_1, R_2, R_3, R_4, R_5$ satisfying $R_5R_4R_3R_2R_1=1$, where $R_1$ is a reflection in a complex geodesic and the other $R_i$'s are reflections in points of the complex hyperbolic plane. We show that this space modulo $\PU(2,1)$-conjugation is bending-connected and has dimension~$4$. Using this, we construct a $4$-dimensional bending-connected family of
complex hyperbolic structures on a disc orbibundle with vanishing Euler number over the sphere with $5$ cone points of angle $\pi$. Bending-connectedness here means that we can naturally deform the geometric structure, like Dehn twists in Teichmüller theory. Additionally, finding complex hyperbolic disc orbibundles with vanishing Euler numbers is a hard problem, originally conjectured by W. Goldman and Y. Eliashberg and solved by S. Anan'in and N. Gusevskii, and we produce a simpler and more straightforward construction for them.
\end{abstract}

\section{Introduction}
{Determining whether a given manifold has some geometric structure is a fundamental question in uniformization theory, a central field in differential geometry. A manifold $M$ admits a structure modeled over a space $E$ if there exists a discrete subgroup $\Gamma$ of
$\mathrm{Isom}(E)$ such that $M=E/\Gamma$. So, being able to construct
discrete subgroups of the isometry group of a model space is key to understanding the geometry of a manifold.

For geometries modeled by the Poincar\'e disc in classical Teichm\"uller theory, 
the focus is on representations of surface groups in 
$\PU(1,1)\simeq\mathrm{PSL}(2,\RR)$, the isometry group of the 
Poincar\'e disc, modulo conjugations. If~$\Sigma$ is a closed oriented 
surface (or, more generally, a $2$-orbifold) with negative Euler 
characteristic, its {\it $\PU(1,1)$-character variety\/} is  
$\hom(\pi_1(\Sigma),\PU(1,1))/\PU(1,1)$, and the discrete and faithful
representations in this character variety form an open subset that is the 
{\it Teichm\"uller space\/} of~$\Sigma$. Thus, complex structures 
on~$\Sigma$ can be seen as points of this character variety.

A natural generalization is to consider as a model the 
{\it complex hyperbolic $n$-space\/} $\HH_\CC^n$, which can be viewed as a unit open 
ball in $\CC^n$ endowed with its group of biholomorphisms, the group~$\PU(n,1)$. Additionally, $\HH_\CC^n$ has a natural Kähler structure with $\PU(n,1)$ as the group of holomorphic isometries. The space $\HH_\CC^1$ coincides with the Poincaré disc. See Section \ref{subsec:hyperbolic-spaces}.

In this work, we study complex hyperbolic disc {\it orbibundles\/}, that is,
disc orbibundles over a $2$-orbifold~$\Sigma$ for which the total space admits 
a geometric structure modeled over~$\mathbb H_{\mathbb C}^2$. 
In other words, we study discrete
and faithful representations lying in the {\it $\PU(2,1)$-character variety\/} of~$\Sigma$, that is, discrete and faithful representations in $\hom(\pi_1(\Sigma),\PU(2,1))/\PU(2,1)$. Moreover, we are interested 
in examples of  disc orbibundles where the complex hyperbolic structure is 
non-rigid, meaning it can  be continuously deformed, similar to what happens 
to complex structures on surfaces 
in Teichm\"uller theory. 

\smallskip

Let us digress about orbibundles before discussing the core subject of this work.

Consider a cocompact Fuchsian group $\Gamma$. A {\it good disc orbibundle} over the good $2$-orbifold $\HH_\CC^1/\Gamma$ is a map $(\HH_\CC^1 \times \mathbb D)/\Gamma \to \HH_\CC^1/\Gamma$ given by $[x,f] \mapsto [x]$, where $\mathbb D$ is an open two-dimensional disc
and the action of $\Gamma$ on $\HH_\CC^1 \times \mathbb D$ is of the form
$g(x,f) = (gx,a(g,x)f)$, where $a(g,-)\colon\HH_\CC^1 \times \mathbb D \to \mathbb D$ 
is smooth map with $a(h,gx)a(g,x)f=a(hg,x)f$ and $a(1,x)f=f$, for
any $g,h\in\Gamma$, $x\in\HH_\CC^1$, and $f\in\mathbb D$.
Observe that the $2$-orbifold $\Sigma\coloneq\HH_\CC^1/\Gamma$ and the 
$4$-orbifold $L\coloneq(\HH_\CC^1 \times \mathbb D)/\Gamma$, the {\it base space} and
the {\it total space\/} respectively, have the 
same fundamental group $\Gamma$. In this paper, we only consider 
good orbifolds and good orbibundles.

Whenever $\Gamma$ is a surface group, that is, whenever~$\Gamma$ is 
torsion-free, the described quotients are manifolds and we obtain 
common disc bundles. Nevertheless, surface groups are notoriously 
complicated to work with from a numerical/combinatorial viewpoint, 
since they have several generators and relations.
In other words, we use orbifold groups to construct disc orbibundles.
Since, by Selberg's Lemma (See \cite[Section 7.6]{ratcliffe}), it is always possible to find a finite-index torsion-free subgroup
of any co-compact Fuchsian group, if we have a disc orbibundle, we also have a disc bundle over a surface, although sometimes
it is hard to determine exactly which surface.

An oriented disc orbibundle 
$\zeta:(\mathbb H_{\mathbb C}^1\times\mathbb D)/\pi_1(\Sigma)
\to\Sigma$ over a $2$-orbifold~$\Sigma$ has two basic 
invariants: the Euler characteristic~$\chi$ of~$\Sigma$ and the 
Euler number~$e$ of the disc orbibundle~$\zeta$. 
These invariants can be computed via Chern-Weil theory (integral of 
the Euler class) or an adaptation of the Poincar\'e-Hopf 
(see~\cite{orbigoodles}). We follow the second approach (see 
Subsection~\ref{defi:euler number}). If we think of~$\Sigma$ as a 
section of~$\zeta$, then we can define~$\chi$ as the Euler number 
of the tangent orbibundle~$T\Sigma$ and~$e$ as the Euler number of the 
normal orbibundle~$N\Sigma$.

We say that such a disc orbibundle has a {\it complex hyperbolic structure\/}
when the total space $L$ is diffeomorphic to a quotient of 
$\HH_\CC^2$ by a discrete subgroup of $\PU(2,1)$, meaning that there 
exists a discrete and faithful representation $\varrho:\pi_1(\Sigma)\to\PU(2,1)$ 
such that $L\simeq\HH_\CC^2/\varrho(\pi_1(\Sigma))$. Thus, to each complex hyperbolic 
structure on the disc orbibundles~$\zeta$, we have a 
$\PU(2,1)$-representation of~$\Gamma$. Furthermore, equivalent complex hyperbolic 
structures correspond to the same representations $\Gamma \to \PU(2,1)$ 
modulo conjugation by~$\PU(2,1)$. In conclusion, the space of complex 
hyperbolic structures on~$L$ is naturally embedded in  
the $\PU(2,1)$-character variety of~$\Gamma$.
When $L \to \Sigma$ is complex hyperbolic disc orbibundle we have a third discrete invariant called {\it Toledo invariant\/}, defined by 
$\tau\coloneq \frac{4}{2\pi}\int_\Sigma \omega$, where we think of $\Sigma$ as a section of the complex hyperbolic orbibundle and $\omega$ is the symplectic form arising from the complex hyperbolic structure (see Section~\ref{sec:toledo-invariant}).

\smallskip

One of the simplest orbifold groups is, for $n\geq 5$, the {\it hyperelliptic group\/}
$$H_n\coloneq \langle r_1,\ldots,r_n \mid r_i^2=r_n\ldots r_1 = 1\rangle,$$
which is the orbifold fundamental group of the $2$-sphere with~$n$ cone points 
of angle~$\pi$, and is the type of orbifold we will use as basis for our 
orbibundles. These spheres, denoted $\Sigma_{n,2}$, admit hyperbolic structure, 
meaning that $H_n$ can be seen as Fuchsian groups (see Section \ref{sec:hyper}). 
The $\PU(2,1)$-representations of these groups are relatively simple
to work with since given a representation $\varrho:H_n\to\PU(2,1)$
we have that $\varrho(r_i)$ is either the identity or a holomorphic
involution. Moreover, it is quite simple to determine the finite index
torsion-free subgroup of~$H_n$ (see Subsection~\ref{subsec:pu11-repr}).
For a detailed study of $\PU(1,1)$-representations
of~$H_n$ see~\cite{Sasha2012,basic}.

\smallskip

Holomorphic involutions of $\HC$ are either reflections in points or
reflections in complex geodesics (holomorphically embedded Poincar\'e
discs). 
Moreover, each such reflection has a single representative
in~$R\in\SU(2,1)$ satisfying $R^2=1$ (see Subsection~\ref{subsec:isometries}). 
Hence, for each representation $\varrho:H_n\to\PU(2,1)$,
satisfying $\varrho(r_i)\neq 1$ for every~$i$, there
is a unique corresponding relation $R_nR_{n-1}\ldots R_1=\delta$
in~$\SU(2,1)$, where~$\delta$ is a cube root of unity and each $R_i$
is a reflection of one of the described types.
An important aspect that manifests is that, while considering representations
$\varrho:H_n\to\PU(2,1)$, $\varrho(r_i)\neq 1$, or equivalently relations
$R_nR_{n-1}\ldots R_1=\delta$, the choice of the type of the reflection~$R_i$ for each $i$ is a consequential one, affecting the existence of such a relation, 
the discreteness of the representation, 
and the topology of the associated orbibundles. 

The existing examples of 
discrete and faithful representations
$\varrho:H_n\to\PU(2,1)$ are of two kinds: 
\begin{enumerate}[{\bf (1)}]
\item Examples where all 
$\varrho(r_i)$ are reflections in complex lines. Several of these 
representations where
constructed in~\cite{AGG2011}, all having relative
Euler number $e/\chi\in(0,0.5)$;
\item Examples where
all $\varrho(r_i)$ are reflections in points. These appear 
in~\cite{turnover} where, advancing on the techniques of~\cite{AGG2011}, they 
constructed an example of a representation for $n=5$ where $e/\chi=-1$, and 
examples for $n=8,12$ where $e/\chi=0$.
\end{enumerate} 

The relative Euler numbers $0$ and $-1$ mean that these disc orbibundles give rise to complex hyperbolic trivial bundles and cotangent bundles when considering a finite index torsion-free subgroups of $H_n$.

\smallskip

On the works \cite{AGG2011} and \cite{turnover}, the examples were actually discovered as discrete and faithful representations $G_{n_1,n_2,n_3} \to \PU(2,1)$ where
$$G_{n_1,n_2,n_3}\coloneq \langle g_1,g_2,g_3\colon g_1^{n_1}=g_2^{n_2}=g_3^{n_3}=1,\, g_3g_2g_1=1\rangle$$
is the {\it turnover group\/} (also known as {\it von Dyck group\/}),
which is the orbifold fundamental group of the sphere with three cone points with angles $\frac{2\pi}{n_1}$, $\frac{2\pi}{n_2}$ and $\frac{2\pi}{n_3}$. when $-1+\frac1{n_1}+\frac1{n_2}+\frac1{n_3}<0$, this $2$-orbifold admits hyperbolic structure and $G_{n_1,n_2,n_3}$ can be seen as a Fuchsian group (See Section \cite[Section~3.1]{turnover}). For the particular group $G_{n}\coloneq G_{n,2,n}$ with $n\geq 5$, we obtain $H_n$ as a finite index subgroup of $G_n$ by taking $r_i = g_3^{i-1}g_2(g_3^{-1})^{i-1}$. Therefore, we have the orbifold cover $\HH_\CC^1 / H_n \to \HH_\CC^1 / G_n$
and the orbibundle pullback
\begin{equation*}
\begin{tikzcd}
	{\mathbb H_{\mathbb C}^2/H_n} && {\mathbb H_{\mathbb C}^2/G_n} \\
	{\mathbb H_{\mathbb C}^1/H_n} && {\mathbb H_{\mathbb C}^1/G_n}
	\arrow[from=1-1, to=1-3]
	\arrow[from=1-1, to=2-1]
	\arrow[from=2-1, to=2-3]
	\arrow[from=1-3, to=2-3]
\end{tikzcd}    
\end{equation*}

In the preprint \cite{SashaGusevskii2007}, Sasha Anan'in and Nikolay Gusevskii constructed one trivial disc bundle over a genus $2$ closed Riemann surface with complex hyperbolic structure, providing an example of trivial $\mathbb S^1$-bundles with spherical $\mathrm{CR}$-structure (conjectured by Bill Goldman \cite[p.~583]{Goldman1983}) and holomorphically filled (conjectured by Yakov Eliashberg \cite[Open questions 8.1.1']{Eliashberg1992}). See \cite[Section 1.3]{Schwartz2007} for more on $\mathrm{CR}$-structure.

Inspired by the examples mentioned above, we investigate a simple procedure to 
construct families of trivial disc orbibundles with complex hyperbolic 
structures. We do this by mixing techniques from~\cite{Sasha2012,spell}, used to 
in the explicit obtain representations of the hyperelliptic group and to 
parameterize the complex hyperbolic structures via bendings, 
and from~\cite{AGG2011,turnover}, used to build the fundamental domain and 
to compute basic invariants 
(Euler number and Toledo invariant) from the orbifold viewpoint. The idea of 
non-rigid complex hyperbolic structures on disc bundles can be traced back 
to~\cite{Gaye2008}.

More precisely, we investigate disc orbibundles with vanishing Euler numbers over 
$\Sigma_{5,2}$ for which the total space admits complex hyperbolic structure, as 
defined above. Additionally, we study the natural deformation of these complex 
hyperbolic structures via a technique called {\it bending\/} 
(see Subsection~\ref{subsec:bendings}), which transforms the associated
relation $R_nR_{n-1}\ldots R_1=\delta$ by changing a product
$R_{i+1}R_i$ into a new one $R_{i+1}'R_i'$ in a way that 
$R_{i+1}R_i=R_{i+1}'R_i'$. Such a procedure is similar 
to simple earthquakes in Teichm\"uller theory (see \cite[Definition 3.5]{SashaBento2009}). 

The main results of this paper are Corollary~\ref{cor:dim-discrete} and Theorem~\ref{thm:bend-quad}
together with the computational test of Section~\ref{sec:computational}. To summarize, we have 
the following:

\medskip

\noindent{\bf Summary of Main Results.}
{\it There exists an open\/ $4$-dimensional bending-connected subset of\/ 
$$\{\varrho \in\hom(H_5,\PU(2,1))\mid \varrho(r_i) \neq \Id\}/\PU(2,1)$$
consisting of discrete and faithful representations where\/ $\varrho(r_1)$
is a reflection in a complex geodesic and\/ $\varrho(r_i)$, with\/~$i\neq 1$, is
a reflection in a point. Each representation in this subset defines a trivial disc
orbibundle over\/ $\Sigma_{5,2}$ with\/ $e=0$ and\/ $\tau/\chi=2/3$.}

\medskip

The strategy for producing these representations is the following: we first construct representations $H_5 \to \PU(2,1)$ as outlined in Section \ref{sec:decom-iso}. We do this by finding three holomorphic involutions $R_1,R_2,R_3 \in \SU(2,1)$ such that there exists involutions $R_4,R_5 \in  \SU(2,1)$ satisfying $R_5R_4R_3R_2R_1 = \exp(\frac{-2\pi i}{3})$, where $R_1$ is the only reflection in a complex geodesic. Then, following Section \ref{sec:discreteness}, we construct a potential fundamental domain for the group generated by the involutions $R_i$ using a special kind of hypersurface called bisector as its boundary. Under algebraic conditions imposed over these bisectors, we guarantee that the fundamental region tessellates $\HH_\CC^2$. Via bending, we eventually find the involutions $R_i$ for which the tessellation occurs. Once the fundamental domain is established, we make use of the tools developed in~\cite{orbigoodles} and~\cite{turnover} to foliate the fundamental domain by discs, giving rise to disc orbibundles, and computing the invariants~$\chi$,~$e$, and~$\tau$. The disc orbibundles over $\Sigma_{5,2}$ we construct have Euler number $0$ and Toledo invariant $\tau = \frac23 \chi$.

An explicit example is constructed in detail in Section \ref{sec:computational}. There we present several techniques to create examples and compute invariants in practice. In Section \ref{subsec:comp-bend}, we show how bending can be used to navigate the space of complex hyperbolic disc bundles over $\Sigma_{5,2}$. It is worth mentioning that, to the authors' knowledge, these are the first examples of complex hyperbolic disc orbibundles arising from representations where the $R_i$ are holomorphic involutions of mixed types, with $R_1$ being reflection in a complex geodesic and the other $R_i$ being reflections in points. As previously said, for the examples from \cite{AGG2011} and \cite{turnover}, all involutions $R_i$ have the same type.

In Proposition \ref{prop:dim-rep-su21}, we also provide an coordinate-free proof for  \cite[Proposition 40]{Gaye2008}, showing that representations $\varrho: H_n \to \PU(2,1)$ with $\varrho(r_j) \neq \Id$ form a $4n-8$ dimensional space, assuming it is non-empty. If we consider representations modulo $\PU(2,1)$ instead, the dimension is $4n-16$. From this result, we conclude that each family of complex hyperbolic disc orbibundles constructed from $H_5$ is $4$-dimensional.

Finally, $\{\varrho \in\hom(H_n,\PU(2,1))\colon R_j \neq \Id\}/\PU(2,1)$ is divided into three natural components determined by the identity $R_nR_{n-1}\cdots R_1 = \delta$ in $\SU(2,1)$, where $\delta$ can be $1$, $\exp(\frac{2\pi i}{3}),$ or $\exp(-\frac{2\pi i}{3})$. For $n=5$ with $R_1$ as a reflection in a complex geodesic and the other $R_j$'s as reflections in points, we show that the component given by $\delta =1$ is empty and the other two are non-empty (see Corollary \ref{cor:dim-pdelta}). Furthermore, the two non-empty components are bending-connected (see Proposition \ref{prop:bending-connected}). Thus, we know that each complex hyperbolic disc orbibundle over $\Sigma_{5,2}$ found in this work is in a bending connected $4$-dimensional region of the $\PU(2,1)$ character variety of $H_5$  and has an open neighborhood formed of complex hyperbolic disc orbibundles over $\Sigma_{5,2}$ which are pairwise isomorphic as smooth orbibundles but distinct from the holomorphic viewpoint.

\section{Preliminaries}

We produce complex hyperbolic disc orbibundles over $2$-orbifolds via tessellations in the complex hyperbolic plane. We use bisector segments to form the boundary for fundamental domains in the complex hyperbolic plane, similar to how one uses geodesic polygons to produce tessellations in the Poincaré disc. Thus, this section presents the complex hyperbolic plane, bisectors, and holomorphic isometries. 

\subsection{Hyperbolic spaces}
\label{subsec:hyperbolic-spaces}
Consider a field $\KK = \RR$ or $\CC$. Let $V$ be an $(n+1)$-dimensional $\KK$-linear space equipped with a 
Hermitian form $\langle -,- \rangle$ of signature $-+\cdots+$. The pro{\-}to{\-}typi{\-}cal example here is the $\KK$-linear space $\KK^{n+1}$ endowed with the canonical Hermitian form of signature $-+\cdots+$ given by
$$\langle x,y \rangle \coloneq  -x_0\overline{y_0}+\sum_{k=1}^n x_k \overline{y_k},$$
for $x=(x_0,x_1,\ldots,x_n)$ and $y=(y_0,y_1,\ldots, y_n)$ in~$\KK^{n+1}$.

The $n$-dimensional $\KK$-projective space
$\mathbb P_{\KK}V$ is divided in {\it negative\/}, {\it positive\/},
and {\it iso{\-}tro{\-}pic\/} points, respectively:
\begin{equation*}
\begin{gathered}
\BV\coloneq \{p\in\mathbb P_{\mathbb K}V\mid\langle p,p\rangle<0\},\quad 
\EV\coloneq \{p\in\mathbb P_{\mathbb K}V\mid\langle p,p\rangle>0\},\\
\SV\coloneq \{p\in\mathbb P_{\mathbb K}V\mid\langle p,p\rangle=0\}.
\end{gathered}
\end{equation*}
When dealing with signature $-+\cdots+$, the space $\BV$ is always a {\bf B}all (as we see below) and $\SV$ is a {\bf S}phere. The $\EV$ stands for {\bf E}lsewhere.

Indeed, for $V=\KK^{n+1}$ with the canonical Hermitian form, we have 
\begin{equation*}
\begin{gathered}
\BV=\{x \in \mathbb P_{\KK}^n \mid \langle x,x \rangle <0\} = \big\{[1:x_1:\cdots:x_n] \in \mathbb P_{\KK}^n \mid  |x_1|^2+\cdots +|x_n|^2 <1\big\},\\
\SV=\{x \in \mathbb P_{\KK}^n \mid \langle x,x \rangle =0\} = \big\{[1:x_1:\cdots:x_n] \in \mathbb P_{\KK}^n \mid |x_1|^2+\cdots +|x_n|^2 =1\big\}.
\end{gathered}
\end{equation*}
Thus, $\BV$ is a ball and $\SV$ is its boundary, a sphere.

\begin{rmk}
We use the same notation for a point in $\mathbb P_{\KK}V$ and a representative of it in~$V$. 
Also, we write~$\PV$ instead of~$\mathbb P_{\KK} V$ in the absence of ambiguity.
\end{rmk}

Given a nonisotropic point $p\in\PV$ we have the identification 
$T_{p}\PV={\mathrm{Lin}}_{\KK}(\KK p,p^\perp)$, where $p^\perp$ denotes
the space of vectors that are orthogonal to $p$. Such identification provides
a Hermitian metric in both $\BV$ and $\EV$ defined by
$$\langle t_1,t_2\rangle \coloneq  -\frac{\langle t_1(p),t_2(p)\rangle}{\langle p,p\rangle},$$
where $t_1,t_2\in \mathrm{Lin}_{\KK}(\KK p,p^\perp)$ are tangent 
vectors to~$\PV$ 
at~$p$. This Hermitian metric is positive in $\BV$ and of signature $-+\cdots+$ in $\EV$.

The ball $\BV$ equipped with the Hermitian metric $\langle -, - \rangle$ is called {\it $\KK$-hyperbolic $n$-space\/} and it is denoted by $\HH_\KK^n$, while its boundary~$\SV$ is 
denoted~$\partial \HH_\KK^n$. We also denote~$\overline{\HH_\KK^n} \coloneq  \HH_\KK^n \sqcup \partial \HH_\KK^n$.

When $\KK=\RR$, $g\coloneq \langle -, - \rangle$ defines a Riemannian metric on the real 
hyperbolic space $\HH_\RR^n$ that is complete with constant sectional curvature $-1$. 
The space $\EV$ equipped with~$g$ is a complete Lorentz manifold with sectional curvature~$-1$. 

For $\KK=\CC$, the real part of the Hermitian metric provides a Riemannian metric 
$g\coloneq \real\langle-,-\rangle$ on the complex hyperbolic space $\HH_\CC^n$ and the imaginary part, 
$\omega\coloneq \imag\langle-,-\rangle$, a symplectic form. 

The complex hyperbolic line $\HH_\CC^1$ has constant curvature $-4$. Thus, up to scaling the metric, $\HH_\CC^1$ and~$\HH_\RR^2$ are isometric. For $n>1$, the sectional curvature of $\HH_\CC^n$ varies on the interval $[-4,-1]$.

\begin{rmk}
The complex hyperbolic line $\HH_\CC^1$, unlike $\HH_\RR^2$, is complex and biholomorphic to the standard Poincaré disc. On the other hand, the real hyperbolic plane $\HH_\RR^2$ is naturally isometric to the Beltrami-Klein model. In short, we think of $\HH_\CC^1$ as a Poincaré disc and of $\HH_\RR^2$ as a Beltrami-Klein model.
\end{rmk}

We now present the hyperbolic space as a metric space.

For two nonisotropic points $p_1,p_2\in\PV$, we define the {\it tance\/} between
$p_1,p_2$ by
$$\tance(p_1,p_2)\coloneq \frac{\langle p_1,p_2\rangle\langle p_2,p_1\rangle}{
\langle p_1,p_1\rangle\langle p_2,p_2\rangle}.$$
The tance appears in several of the formulas of real and complex hyperbolic geometry, 
attaining different meanings accordingly with the configuration of $p_1,p_2$. 
The distance between points $p_1,p_2\in\mathbb H_{\mathbb K}^n$ is given by
$$\dist(p_1,p_2)=\mathrm{arccosh}\,\sqrt{\tance(p_1,p_2)},$$
thus we may think of the tance as an algebraic version of the distance.

\subsection{The complex hyperbolic plane}
\label{subsec:complex-plane}

From now on, we will restrict ourselves to the {\it complex hyperbolic plane} $\HH_\CC^2$, meaning that $\dim_\CC V = 3$. Nevertheless, most of the subsequent properties of hyperbolic spaces are valid with little to no change for general settings.

As a manifold, $\HH_\CC^2$ is a ball with real dimension $4$ and its boundary $\partial \HH_\CC^2$ is a $3$-sphere. The space of positive points $\EV$ is a pseudo-Riemannian $4$-manifold.

Given a projective line $L$ in $\PV$, there exists a unique point~$c\in\PV$ 
such that $L=\mathbb Pc^\perp$. This point is the {\it polar point\/} of the 
line~$L$.
We say that $L$ is {\it hyperbolic\/}, {\it Euclidean\/}, {\it spherical\/} if
$c\in\BV$, $c\in\SV$, $c\in\EV$, respectively. 
When~$L$ is a hyperbolic line, both $L\cap\HC$ and $L\cap\EV$ are copies of the Poincar\'e
disc~$\mathbb{H}_{\mathbb C}^1$ embedded
as Riemann surfaces.
The intersection of a hyperbolic
projective line with\/ $\HC$ is named {\it complex geodesic}. 
%Notably, a complex geodesic is a one-dimensional complex hyperbolic space $\HH_\CC^1 = \mathrm B(c^\perp)$ and an embedded Riemann surface. 
Additionally, observe that $\EV$ is the space of all complex geodesics. 

\begin{rmk}\label{rmk:tance and line}
Depending on the context, we also refer to $L\cap\CHC$
as a complex geodesic when $L$ is hyperbolic.
\end{rmk}

For two distinct points $p_1,p_2$ of $\EV$, the tance $\tance(p_1,p_2)$ dictates which type of projective line $L\coloneq  
\mathbb P(\mathbb C p_1 + \mathbb C p_2)$ we have in hands:
\begin{itemize}
\item if $\tance(p_1,p_2)<1$, then $L$ is spherical and
$\arccos\sqrt{\tance(p_1,p_2)}$ is the distance between the 
points $p_1,p_2$ of the round sphere $L$;
\item if $\tance(p_1,p_2)>1$, then $L$ is hyperbolic and 
$\mathrm{arccosh}\,\sqrt{\tance(p_1,p_2)}$ is the distance between the 
points $p_1,p_2$ of the Poincar\'e disc $L \cap \EV$;
\item if $\tance(p_1,p_2)=1$, then $L$ is Euclidean.
\end{itemize}

Now we discuss the relative position of two hyperbolic lines. Given two 
hyperbolic lines $L_1$ and $L_2$, we say that they are: {\it ultraparallel\/} 
if they do not intersect in~$\CHC$; {\it asymptotic\/} if they intersect 
in $\partial\HC$; and {\it concurrent\/} if they intersect in $\HC$.

\begin{lemma}
\label{lemma:trianglelines}
{\rm\cite[p.\,100]{Goldman1999}}~Let\/ $L_1,L_2$ be hyperbolic lines and let\/ 
$p_1,p_2\in\EV$ be their polar points,
respectively. Then\/ $L_1$ and $L_2$ are ultraparallel, asymptotic, concurrent
if and only if\/ $\tance(p_1,p_2)>1$, $\tance(p_1,p_2)=1$, $\tance(p_1,p_2)<1$, respectively.
In the case they are concurrent, the angle\/ $\angle(L_1,L_2)$ between\/~$L_1$ 
and\/~$L_2$ is\/ $\arccos\sqrt{\tance(p_1,p_2)}$ and, in the case they are 
ultraparallel, the distance\/ $\mathrm{dist}(L_1,L_2)$ between\/ $L_1$ and\/ $L_2$ 
is $\mathrm{arccosh}\,\sqrt{\tance(p_1,p_2)}$.
\end{lemma}

As we have seen, hyperbolic projective lines give rise to copies of $\HH_\CC^1$ inside $\HH_\CC^2$, which we call complex geodesics. These discs have constant curvature $-4$ and are complex submanifolds.
If $W$ is a real $3$-dimensional subspace of~$V$ such that $\langle -,- \rangle|_{W \times W} $ is real valued and has signature of $-++$, then the surface
$$\mathbb P_{\CC} W  = \{ p \in \mathbb P_\CC V \mid p \in W\}$$
is naturally diffeomorphic to $\mathbb P_\RR W$
and $\mathbb P_{\CC}W \cap \HH_\CC^2$ is isometric to the real hyperbolic plane $\HH_\RR^2$. We say $\mathbb P_{\CC}W \cap \HH_\CC^2$ is a {\it real plane}.

Note that real planes are hyperbolic spaces of curvature $-1$ isometrically embedded in $\HH_\CC^2$ as Lagrangian submanifolds. In general, if we have an 
 orthonormal basis $b_1,b_2,b_3$ for $V$, then $W\coloneq \RR b_1+\RR b_2 + \RR b_3$ define a real plane. 
 
\begin{rmk} Depending on the context, we also refer to $\mathbb P_\CC W \cap\CHC$
as a real plane.
\end{rmk}

\subsection{Geodesics} Let $W$ be a real two-dimensional subspace of $V$ such
that the Hermitian form of $V$ restricted to $W$ is real-valued and 
has signature $-+$. The projection $\mathrm G\,W\coloneq \mathbb P_\CC W$ is 
called {\it extended geodesic\/} of $\HH_\CC^2$ and it is diffeomorphic to $\mathbb P_\RR W = \mathbb P_\RR^1$, a circle. A geodesic of\/~$\mathbb H_{\mathbb C}^2$ is an intersection $\mathrm G\,W\cap\BV$.
Given two distinct point $p_1,p_2\in\CHC$, the
extended geodesic through $p_1,p_2$, namely $\mathbb{P}(\mathbb Rp_1+
\mathbb R\langle p_1,p_2\rangle p_2)$, will be denoted $\mathrm G\fgeo{p_1,p_2}$.
The {\it geodesic\/} through $p_1$ and $p_2$, which we consider with endpoint included, 
will be denoted $\mathrm G\geo{p_1,p_2}$, that is $\mathrm G\geo{p_1,p_2}\coloneq 
\mathrm{G}\fgeo{p_1,p_2}\cap\CHC$. In the
case where $p_1,p_2\in\HC$, the oriented geodesic segment
from $p_1$ to $p_2$ will be denoted $\mathrm G[p_1,p_2]$. Note that every
extended geodesic $\mathrm G\,W$ has its complex line $\mathbb P(W+iW)$, which
is the only complex line containing~$\mathrm G\,W$. The complex
line of $\mathrm G\fgeo{p_1,p_2}$ is denoted $\mathrm L\fgeo{p_1,p_2}$.
Similarly, every (proper) geodesic of\/ $\CHC$ has its complex geodesic. We do not 
introduce a particular notation for this complex geodesic, but sometimes we refer 
to $\mathrm L\fgeo{p_1,p_2}$ as a complex geodesic, 
meaning that we are actually considering $\mathrm L\fgeo{p_1,p_2}\cap\CHC$. 

\begin{rmk}
Given $p \in \HH_\CC^2$ and $t \in T_p\HH_\CC^2$ with $\langle t,t \rangle =1$, the geodesic $\mathbb P(\RR p + \RR t(p))\cap \HH_\CC^2$ is the geodesic passing through $p$ and tangent to $t$. It can be parametrized by $\theta \mapsto \cosh(\theta)p+\sinh(\theta)t(p)$. 
\end{rmk}

\subsection{Bisectors}
\label{subsec:bisectors}
A {\it bisector\/} is a hypersurface defined as the 
equidistant locus of two points in $\mathbb H_{\mathbb C}^2$.

Alternatively, bisectors have the following algebraic description: consider a geodesic
$G\coloneq \mathbb{P}W\cap\CHC$ and let $L\coloneq \mathbb P(W+iW)$ be its complex line. Denote by 
$f$ its polar point, i.e., $W+iW = f^\perp$. 
The {\it bisector\/}~$B$ obtained from the geodesic~$G$ is union of all complex geodesics $\mathbb P(\CC f + \CC x)\cap \CHC$ with $x\in G$:
$$B\coloneq \bigsqcup\limits_{x\in G}\mathbb P(\mathbb C x + \mathbb C f)\cap\CHC.$$
A complex geodesic $\mathbb P(\mathbb C x + \mathbb C f)\cap\CHC$ is a {\it slice\/}, the 
complex geodesic $L\cap\CHC$ is the {\it complex spine\/}, and the geodesic 
$G$ is the {\it real spine\/} of the bisector $B$.
A bisector also has a {\it meridional decomposition\/}: fixing a representative~$f$ for the polar point 
of~$L$, we have
$$B=\bigcup_{\varepsilon\in\mathbb S^1}
\mathbb P(W+\mathbb R\varepsilon f)\cap\CHC.$$
A real plane
$\mathbb P(W+\mathbb R\varepsilon f)\cap\CHC$ is called a {\it meridian\/} of the bisector.
Two distinct meridians intersect exactly over~$G$.

A point $\xi \in B$ which is not a vertex of the real spine $G$ has associated with it a special curve called {\it meridional curve}. 
For $\xi \in G$, the meridional curve determined by $\xi$ is defined as $G$.
For $\xi \in B\setminus G$, there exists a unique meridian plane $R\coloneq \mathbb P(W+\RR\varepsilon f) \cap \CHC$ containing $\xi$ and the meridional curve passing through $\xi$ is the hypercycle parallel to $G$ passing through $\xi$ in the real plane $R$ (for isotropic points~$\xi \in \partial R$ the hypercycle is the segment of~$\partial R$ sharing vertices with~$G$ and containing~$\xi$). See Figure~\ref{fig:meridional-curve}.

\begin{figure}[htb!]
\centering
\includegraphics[scale=1]{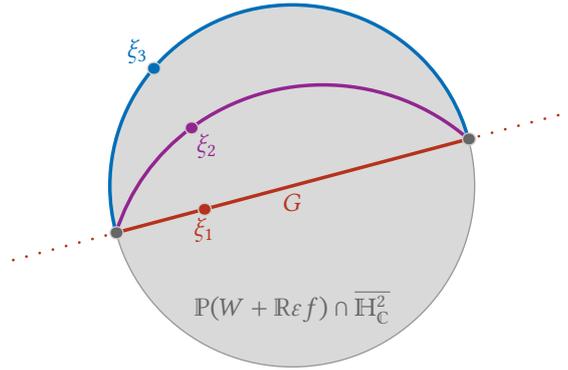}
\caption{A meridian $R\coloneq \mathbb P(W+\RR\varepsilon f) \cap \CHC$ with three meridional curves: the geodesic $G$ in {\color{BrickRed}red} is the meridional curve through $\xi_1$; the hypercycle of $G$ in {\color{Plum}purple} is the meridional curve through $\xi_2$; the segment of $\partial R$ in {\color{RoyalBlue}blue} containing $\xi_3$ is the meridional curve through $\xi_3$.}
\label{fig:meridional-curve}
\end{figure}

Two ultraparallel complex geodesics $C_1$ and $C_2$ determine a single bisector 
containing~$C_1$ and~$C_2$ 
as slices~\cite[Theorem~5.2.4]{Goldman1999}. Explicitly,
since $C_1$ and $C_2$ are ultraparallel, they intersect at a point $f\in\EV$.
If $c_i\in\HC$ is the intersection point of the line~$C_i$ 
and~$\mathbb P f^\perp$, $i=1,2$, and $G\coloneq \mathrm G\geo{c_1,c_2}$ we have
that the bisector $B$ with real spine~$G$ has~$C_1$ and~$C_2$ as slices. We define the
{\it oriented segment of bisector\/} starting with the slice~$C_1$ and ending
with the slice $C_2$ as
$$\mathrm{B}[C_1,C_2]\coloneq 
\bigsqcup_{x\in\mathrm G[c_1,c_2]}\mathbb P(\mathbb C x + \mathbb C f)\cap\CHC.$$
The full oriented bisector containing $\mathrm{B}[C_1,C_2]$ is
denoted~$\mathrm{B}\geo{C_1,C_2}$.
Similarly, the bisector ray $\mathrm B[C_1,C_2\rcurvyangle$ is defined from the geodesic ray $\mathrm G[c_1,c_2\rcurvyangle$, starting at $c_1$ and containing $c_2$.

\begin{rmk}
Another way of seeing the above construction: given two ultraparallel complex geodesics $C_1$ and $C_2$, there are $c_1\in C_1$ and $c_2 \in C_2$ such that $\dist(c_1,c_2) = \dist(C_1,C_2)$ and these points are uniquely determined. The bisector providing the bisector segment $\mathrm B[C_1,C_2]$ is the one with real spine $\mathrm G\geo{c_1,c_2}$.
\end{rmk}

Following~\cite[Section~A.1.8]{AGG2011}, for the geodesic $G$ in $\CHC$ given by $W\subset V$ we have that for two distinct points $p_1,p_2 \in \mathbb P W$ the bisector 
$B$ with real spine $G$ is given by the equation
\begin{equation} 
\label{eq:bisector}
    b(x,p_1,p_2)\coloneq \imag\frac{\langle p_1,x \rangle \langle x,p_2\rangle}{\langle p_1,p_2\rangle}=0, \quad x \in \CHC.
\end{equation}
Thus, as a hypersurface, $B$ separate $\HH_\CC^2$ into two topological balls given by
$$b(x,p_1,p_2)<0 \quad \text{and}\quad b(x,p_1,p_2)>0.$$

There is a natural non-vanishing vector field along $B$ given by
\begin{equation}\label{eq:bisec-normal}
n(x,p_1,p_2) \coloneq  - i\frac{\langle-,x \rangle }{\langle x,x \rangle }\left(\frac{\langle x,p_2\rangle}{\langle p_1,p_2\rangle}p_1-\frac{\langle x,p_1\rangle}{\langle p_2,p_1\rangle}p_2\right),
\end{equation}
as it is presented in \cite[Proposition A.21]{AGG2011}.

If two bisectors share a slice $C$ and intersect transversally along $C$, then the bisectors are transversal and their intersection is~$C$ (see~\cite[Lemma 2.13]{AGG2011}). Bisectors sharing a common slice are called cotranchal (see \cite[Section 9.1.1]{Goldman1999}).

Consider the positive points $p_1,p_2,p_3$ and the complex geodesics $C_j\coloneq \mathbb Pp_j^\perp$. Assume that $C_1,C_2$ and $C_2,C_3$ are ultraparallel, meaning that $\tance(p_1,p_2),\tance(p_2,p_3)>1$. Note that the bisectors $\mathrm B\geo{C_1,C_2}$ and 
$\mathrm B\geo{C_2,C_3}$ share $C_1$. By \cite[Criterion A.27]{AGG2011}, these bisectors are transversal if, and only if,
\begin{equation}
\label{eq:bisec-transv}
\left| \real\left(\frac{\langle p_1,p_3 \rangle \langle p_2,p_2 \rangle}{\langle p_1,p_2\rangle\langle p_2,p_3\rangle}\right) - 1 \right|< \sqrt{1-\frac{1}{\tance(p_1,p_2)}} \,\sqrt{1-\frac{1}{\tance(p_3,p_2)}}.
\end{equation}

\begin{rmk}In the situation where $p_1,p_2$ are positive, the bisector is determined by the slices $C_1,C_2$ with $C_j\coloneq \mathbb Pp_j^{\perp}\cap\HC$. 
The bisector $B\coloneq \mathrm B\geo{C_1,C_2}$ is oriented by following the orientation of the geodesic forming the real spine going from $C_1$ towards $C_2$. Since the slices are complex discs, $B$ is naturally oriented. More precisely, if $c_j$ is the intersection of $C_j$ with the real spine, then we can consider the vectors $t,t'$ tangent to $B$ at $c_1$ with $t$ tangent to the real spine pointing towards $c_2$ and $t' \in T_{c_1}C_1$. Thus, the vectors $t,t',it'$ provide the orientation of $T_{c_1}B$ and, consequentially, the orientation of $B$. Note that the vector $-it$ is normal to $B$ at $c_1$ and provides the orientation of $B$ because the ordered basis $-it,t,t',it'$ has the same orientation as $T_{c_1}\HH_\CC^2$. 

Moreover, $-it$ and $n(c_1,p_1,p_2)$, where $n$ is given by~\eqref{eq:bisec-normal}, are pointing towards the region defined by
$b(x,p_1,p_2)<0$. 
Indeed, up to choosing representatives, we might assume $\langle p_1,p_1 \rangle = \langle p_2,p_2 \rangle = 1$ and $r\coloneq \langle p_1,p_2 \rangle < 0$. Since $\tance(p_1,p_2)=r^2>1$, we must have $r<-1$. Also, $c_1=p_2-r p_1$ and $c_2=p_1-rp_2$. The curve $\gamma(\theta)\coloneq \cosh(\theta)c_1-\sinh(\theta) p_1 $ is pointing towards $c_2$ with velocity $t = -\frac{\langle -,c_1 \rangle}{\langle c_1,c_1 \rangle} p_1$ for $\theta = 0$. On the other hand $$n(c_1,p_1,p_2)= i\frac{\langle -,c_1 \rangle}{\langle c_1,c_1 \rangle}\frac{(1-r^2)}{r} p_1 = \frac{r^2-1}{r}it,$$
meaning that $-it$ and $n(c_1,p_1,p_2)$ are pointing towards the same direction. 
Hence, in the case where $p_1$ and $p_2$ are positive points, $n(x,p_1,p_2)$ provides the orientation of the oriented bisector 
$\mathrm B\geo{C_1,C_2}$. Note as well, that the curve $\theta\mapsto\cosh(\theta)c_1+i\sinh(\theta) p_1$ has velocity $-it$ for $\theta = 0$ and
$$b(\cosh(\theta)c_1+i\sinh(\theta) p_1,p_1,p_2) = \frac{\sinh(2\theta)(-1+r^2)}{r}<0 \quad \text{for}\quad  \theta > 0.$$
Thus the normal vector $n(x,p_1,p_2)$ is pointing towards the region
$b(x,p_1,p_2)<0$.
\end{rmk}

% {\color{Blue} Alt: Given two distinct points $p_1,p_2,\in\BV$
% consider the geodesic $G\coloneq \mathrm G\fgeo{p_1,p_2}$ and let $f\in\EV$ be the polar point
% of the line $L\coloneq \mathrm L\fgeo{p_1,p_2}$. The projective cone $\mathrm B\fgeo{p_1,p_2}$ over $G$ and
% with vertex $f$ is a  ({\it extended\/}) {\it bisector\/}. The intersection $\mathrm B\fgeo{p_1,p_2}\cap
% \mathbb H_{\mathbb C}^2$ is an ``usual'' bisector, that is, the equidistant locus of two points 
% in~$\mathbb H_{\mathbb C}^2$. 
% In this way, $\mathrm B\fgeo{p_1,p_2}$ has a {\it slice decomposition\/} as the union of complex lines 
% connecting the focus $f$ and a point $x \in G$:
% $$\mathrm B\fgeo{p_1,p_2}=\bigsqcup_{x\in G}\mathrm L\fgeo{x,f}.$$
% The lines $\mathrm L\fgeo{x,f}$ are the {\it slices\/}, the geodesic $G$ is the {\it real spine\/},
% and the line $L$ is the {\it complex spine\/} of the bisector $\mathrm B\fgeo{p_1,p_2}$. A bisector also
% has a {\it meridian decomposition\/}: if we fix a representative $f$ for the polar point 
% of $L$ and write $G=\mathbb P W$, we have
% $$\mathrm B\fgeo{p_1,p_2}=\bigcup_{\varepsilon\in\mathbb S^1}
% \mathbb P(W+\mathbb R\varepsilon f).$$
% For each $\varepsilon\in\mathbb S^1$, the real plane $\mathbb P(W+\mathbb R\varepsilon f)$
% is a {\it meridian\/} of the bisector. The intersection of any two distinct meridian is $G\cup\{f\}$.
% }

\subsection{Triangle of bisectors}
\label{subsec:triangle}
% TROQUEI O NOME DOS POLARES PRA q_i POR CONTA DA CONSTRUCAO MAIS PRA FRENTE
Three pairwise ultraparallel complex geodesics $C_1$, $C_2$, $C_3$ with polar points $q_1$, $q_2$, $q_3$, respectively, define a {\it triangle 
of bisectors} $$\triangle(C_1,C_2,C_3)\coloneq \mathrm{B}[C_1,C_2] \cup 
\mathrm{B}[C_2,C_3] \cup \mathrm{B}[C_3,C_1].$$ By Lemma~\ref{lemma:trianglelines}, the algebraic conditions 
guaranteeing that the complex geodesics are ultraparallel are 
$$\tance(q_1,q_2)>1,\quad \tance(q_2,q_3)>1,\quad \tance(q_3,q_1)>1.$$

The triangle $\triangle(C_1,C_2,C_3)$ is {\it transversal\/} if the extended bisectors intersect pairwise 
transversally. Algebraically (see \cite[Criterion~A.31]{AGG2011}), this is characterized by the equations
\begin{equation}
\label{eq:tranversalcond}
\begin{gathered}
\varepsilon_0^2 t_{12}^2 + t_{23}^2+t_{31}^2<1+2t_{12}t_{23}t_{31} \varepsilon_0,
\quad \varepsilon_0^2 t_{31}^2 + t_{12}^2+t_{23}^2<1+2t_{12}t_{23}t_{31} \varepsilon_0,\\
\varepsilon_0^2 t_{23}^2 + t_{31}^2+t_{12}^2<1+2t_{12}t_{23}t_{31} \varepsilon_0,
\end{gathered}
\end{equation}
where $t_{ij} \coloneq  \tance(q_i,q_j)$ and 
$$\varepsilon=\varepsilon_0+i \varepsilon_1\coloneq 
\frac{ \langle q_1,q_2 \rangle\langle q_2,q_3 \rangle\langle q_3,q_1 \rangle}
{\big|\langle q_1,q_2 \rangle\langle q_2,q_3 \rangle\langle q_3,q_1 \rangle\big|}.$$
The determinant of the Gram matrix for $q_1,q_2,q_3$ values 
$$\det [\langle q_i,q_j \rangle] = 1+2t_{12}t_{23}t_{31}\varepsilon_0 - t_{12}^2-t_{23}^2-t_{31}^2$$ and it must be negative, because the signature of $V$ is $-++$. Thus, from the inequalities imposing transversality to the sides of the triangle $\triangle(C_1,C_2,C_3)$, we conclude that $|\varepsilon_0| < 1 $. Therefore, $\varepsilon_1$ cannot be zero, dividing triangles of bisectors in those with $\varepsilon_1<0$ and those with $\varepsilon_1>0$.
Additionally, the parameter $\varepsilon_0$ is always positive for transversal triangles. Indeed, without loss of generality, we may assume $t_{12}\leq t_{23},t_{31}$ and, 
from the first inequality in~\eqref{eq:tranversalcond},
%$$\varepsilon_0^2 t_{12}^2 + t_{23}^2+t_{31}^2<1+2t_{12}t_{23}t_{31} \varepsilon_0$$ 
we conclude
$$\varepsilon_0>\frac{t_{23}^2 + t_{31}^2 - 1}{2t_{12}t_{23}t_{31}-\varepsilon_0 t_{12}^2}>0,$$
since $|\varepsilon|=1$ and $t_{ij}>1$ for $i\neq j$. All in all, we have
$0<\varepsilon_0<1$ and $\varepsilon_1 \in (-1,0)\cup (0,1)$.

\medskip
We now introduce the concept of orientation to a transversal triangle of bisectors.

The simplest types of oriented triangles of bisectors are those over a complex geodesic: if 
$\mathbb D=\mathbb P f^\perp \cap \HC$ is a complex geodesic and $\triangle(c_1,c_2,c_3)$ is a geodesic triangle 
in the Poincar\'e disc $\mathbb D$, where $c_1,c_2,c_3 \in \mathbb D$ are pairwise distinct, then the complex 
geodesics $$C_i \coloneq  \mathbb P(\mathbb C c_i+ \mathbb C f)\cap\HC$$ are ultraparallel and define a transversal
triangle of bisectors. Since $C_i$ is a complex geodesic, it has a polar point $q_i$. If we assume the points 
$c_1,c_2,c_3$ are arranged counterclockwise, the number~$\varepsilon$ is in this particular 
case $\exp(-2 i \,\mathrm{area}(\triangle(c_1,c_2,c_3)))$ and, as consequence,  $\varepsilon_1<0$. Otherwise, 
if the points $c_1,c_2,c_3$ are arranged in a clockwise cycle, then $\varepsilon_1>0$.

Thus, we say that a transversal triangle of bisector is {\it counterclockwise-oriented} when $\varepsilon_1<0$. The 
space of all transversal counterclockwise-oriented triangles of bisectors is path-connected 
(see \cite[Lemma~2.28]{AGG2011}). In this way, we may interpret~$\varepsilon_1<0$ as meaning 
that the triangle of  bisector is similar to the one over a complex line described above, 
constructed from a counterclockwise geodesic 
triangle. For more details, see \cite[Section~2.5]{AGG2011}.
\subsection{Isometries of the complex hyperbolic plane}
\label{subsec:isometries}
The elements of $\mathrm{GL}\,V$ that preserve the Hermitian form of $V$ 
and have determinant~$1$ form the the special unitary group $\SU(2,1)$. The projectivization
$\PU(2,1)\coloneq \SU(2,1)/\{1,\omega,\omega^2\}$, where $\omega\coloneq e^{2\pi i/3}$, is the
group of holomorphic isometries for $\mathbb H_{\mathbb C}^2$. 

\begin{rmk}
In general, $\PU(n,1)$ is the group of holomorphic isometries for $\HH_\CC^n$ and coincides with the group of biholomorphisms of the ball $\HH_\CC^n$.
\end{rmk}

We classify nonidentical isometries of $\mathbb H_{\mathbb C}^2$ as follows: an isometry 
$I\in\PU(2,1)$ is {\it elliptic\/}, {\it parabolic\/}, {\it loxodromic\/} if it,
respectively, fixes a point in $\HC$, fixes exactly one point in~$\partial\HC$, 
fixes exactly two
points in~$\partial\HC$. If $I$ is elliptic and pointwise fixes a projective line, 
we say that $I$ is {\it special elliptic}, otherwise, we say that $I$ is 
{\it regular elliptic\/}.
We use the same classification for nonidentical elements of~$\SU(2,1)$, which we also 
regard as the group of holomorphic isometries of~$\HC$.

Any nonidentical holomorphic involution of $\mathbb H_{\mathbb C}^2$ can be lifted to
$\SU(2,1)$ as an isometry of the form
$$I:x\mapsto 2\frac{\langle x,p\rangle}{\langle p,p\rangle}p-x,$$
for some point $p\in\PV\setminus\SV$. We call $p$ the {\it center\/} of $I$ and denote 
such involution by $R^p$. The involution~$R^p$ is the reflection in~$p$ 
if~$p\in\HC$ and it is the reflection in the complex
geodesic $\mathbb Pp^\perp\cap\HC$ if~$p\in\EV$. Whenever we write $R^p$, it is assumed that $\langle p,p \rangle \neq 0$.

The map $f:\mathbb C\to\mathbb R$ given by
$$f(z)=|z|^4-8\real(z^3)+18|z|^2-27$$
is quite useful to study the conjugacy classes of elements of $\SU(2,1)$. 
We define $$\Delta^\circ\coloneq \{z\in\mathbb C\mid f(z)<0\}, \quad \partial\Delta\coloneq \{z\in\mathbb C\mid f(z)=0\} \quad \text{and} \quad\Delta\coloneq \Delta^\circ\cup\partial\Delta.$$
By~\cite[Theorem~6.2.4]{Goldman1999}
a nonidentical isometry $I\in\SU(2,1)$ is regular elliptic iff $\trace I\in\Delta^\circ$, and it is
loxodromic iff $\trace I\in\mathbb C\setminus\Delta$. The set $\partial\Delta$ is a deltoid
parameterized by $\xi^{-2}+2\xi$ with $\xi\in\mathbb S^1$ and it is known as {\it Goldman's deltoid\/} (see Figure \ref{fig:ell1delt}). Its points correspond to the traces of special elliptic and parabolic isometries.
See~\cite[Section~2.3]{Will2016} and~\cite[Section~2.1]{spell} for details.

\subsection{Holonomy on a triangle of bisectors}
\label{subsec:holonomy}

A transversal triangle of bisector $\triangle(C_1,C_2,C_3)$ bounds a closed ball $T$ 
in~$\CHC$. \begin{rmk}We denote~$\partial_\infty T \coloneq  T \cap \partial \HH_\CC^2$. Observe that $\partial T = \triangle(C_1,C_2,C_3) \cup \partial_\infty T$, which is different of $\partial_\infty T$. Furthermore, $\partial_\infty T$ is a closed solid torus bounded by the torus~$\partial\!\! \triangle(C_1,C_2,C_3)= \triangle(C_1,C_2,C_3)\cap\partial \HH_\CC^2$.\end{rmk}  

\smallskip

Suppose $\triangle(C_1,C_2,C_3)$ is a transversal 
and counterclockwise oriented triangle of bisectors, as defined in Subsection~\ref{subsec:triangle}. Given $\xi_1 \in C_1 \cup \partial C_1$, we consider the meridional curve segment for the bisector $B[C_1,C_2]$ starting at $\xi_1$ and ending at $\xi_2 \in C_2 \cap \partial C_2$. Similarly, we consider the meridional curve segment starting at~$\xi_2$ and ending at $\xi_3 \in C_3 \cup \partial C_3$ on the bisector $B[C_2,C_3]$. One last time, we consider the meridional curve segment starting at~$\xi_3$ and ending at $\xi_4 \in C_1 \cup \partial C_1$ on the bisector $B[C_3,C_1]$.
The obtained map $\xi_1 \mapsto \xi_4$ defines an automorphism of $C_1 \cap \partial C_1$ which we refer to as the {\it holonomy of the triangle $\triangle(C_1,C_2,C_3)$} and denote by~$I$.

Following Subsection~\ref{subsec:isometries}, we make use of reflections to write down $I$. 
If, for $i=1,2,3$, $M_i$ is the middle slice of $B[C_i,C_{i+1}]$ and $m_i$ is its polar point, then
$$\xi_2 = R^{m_1}\xi_1, \quad \xi_3 = R^{m_2}\xi_2, \quad \xi_4 = R^{m_3}\xi_3.$$
Therefore, the holonomy $I$ is the restriction of $R^{m_3}R^{m_2}R^{m_1}$ to $C_1 \cup \partial C_1$.

The important thing about the holonomy is the following: there is always $\xi_1 \in \partial C_1$ such that $I$ moves~$\xi_1$ in the clockwise direction, meaning that when we go from $\xi_1$ to $I^2\xi_1$ in the clockwise direction, we pass through $I\xi_1$. 

So, take a point~$\xi_1\in\partial C_1$ that is moved in the clockwise direction by~$I$. We denote 
 by $[\xi_1,\xi_2] \cup [\xi_2,\xi_3] \cup [\xi_3,\xi_4]$ the concatenation of the three oriented meridional curve segments $[\xi_i,\xi_{i+1}]$ and join it to the circle segment going from $\xi_4$ to $\xi_1$ in the counterclockwise direction, thus giving rise to a closed curve in~$\partial\!\!\triangle(C_1,C_2,C_3)$. This curve is important because it is contractible in $\partial_\infty T$ and, as consequence, it is used when compute Euler numbers. For details, see \cite[Section 2.5.1]{AGG2011}.

\section{Representations of the hyperelliptic group}
\label{sec:hyper}

A {\it hyperelliptic surface} $\Sigma$ is a compact Riemann surface endowed with 
a {\it hyperelliptic involution\/}, that is, a holomorphic involution 
$\iota:\Sigma\to\Sigma$ fixing exactly~$n=2g+2$ points,
where~$g$ is the genus of~$\Sigma$.
By~\cite{Maclachlan1971}, considering the group
\begin{equation}
\label{eq:hypergroup}
H_n\coloneq \langle r_1,\ldots, r_n \mid r_i^2=r_n\ldots r_1=1\rangle,
\end{equation}
we have that the extension of the fundamental group
$\pi_1(\Sigma)$ by $\iota$ is the group~$H_{2g+2}$.
For $n\geq 5$, we call the group $H_n$ defined in~\eqref{eq:hypergroup} the {\it hyperelliptic group\/}.

\subsection{{\rm PU(1,1)}-Representations
of the hyperelliptic group}
\label{subsec:pu11-repr}
Consider a convex polygon in $\HH_\CC^1$ with vertices
$v_1,v_2,\ldots,v_n$. Assume that the sum of the inner angles of 
the polygon is~$2\pi$. Let~$p_i$ denote the middle point of the
segment $\mathrm G[v_i,v_{i+1}]$ and let $R^{p_i}$ denote the
reflection at~$p_i$, indices 
modulo~$n$. Poincar\'e's polygon theorem guarantees that the
given convex polygon with the isometries $R^{p_i}$ tessellates 
$\HH_\CC^1$ and the isometries $R^{p_i}$ satisfy
$$(R^{p_i})^2=R^{p_n}\ldots R^{p_1}=1.$$
It follows that the representation $\varrho:H_n\to\PU(1,1)$ of 
the hyperelliptic group defined by
$\varrho(r_i)=R^{p_i}$ is discrete and faithful and that the 
$2$-orbifold $\HH_\CC^1/H_n$ is topologically a sphere
with~$n$ cone points of angle~$\pi$. We denote this orbifold 
by $\mathbb S^2(2,2,\ldots,2)$, where the number of $2$'s within parenthesis
is~$n$. The Euler characteristic of the orbifold 
$\HH_\CC^1/H_n$ is 
$$\chi(\HH_\CC^1/H_n) = \chi(\mathbb S^2) + \sum_{i=1}^n \left(-1+\frac12\right) = 2-\frac{n}2.$$
See \cite[Section 13.3]{thurston2002} and \cite[Theorem 21]{orbigoodles}.
\smallskip

For $n$ even, the group $G_n$ generated by~$r_1 r_i$, where
$i=2,\ldots,n$, is a surface group and has index~$2$ as subgroup of $H_n$. In fact, if we denote by~$P$ the convex polygon we used as the fundamental domain for $H_n$, then $P \cup r_1 P$ is 
the fundamental domain for $G_n$ and $\HH_\CC^1/G_n$ is a Riemann surface. 
Since $\HH_\CC^1/G_n \to \HH_\CC^1/H_n$ is an orbifold cover of degree two, we obtain $\chi(\HH_\CC^1/G_n) = 2 \chi(\HH_\CC^1/H_n) = 4-n$ and, as consequence, the genus of the mentioned surface is $g=(n-2)/2$. See \cite[Proposition 13.3.4]{thurston2002} and \cite[Theorem~25]{orbigoodles} for details.
Additionally, the isometry $r_1$ is an involution of 
$\HH_\CC^1/G_n$ fixing exactly~$n=2g+2$ points, namely the 
ones corresponding to~$p_1,\ldots,p_n$. 
Therefore, $\HH_\CC^1/G_n$ is a hyperelliptic surface.

\smallskip

Now, for $n$ odd, the group $H_n$ has an explicit surface subgroup 
$T_n$ of index~$4$ generated by
$$r_2r_1r_n,\quad  r_2r_nr_1, \quad r_2r_1r_2r_1, \quad r_1r_i,\quad r_2r_1r_ir_2,$$  
where $i =3,\ldots,n-1$, and with fundamental domain 
$P \cup r_1 P \cup r_2 P \cup r_2r_1 P$. See \cite[Section 2.1.4]{AGG2011}. Thus the map $\HH_\CC^1/T_n \to \HH_\CC^1/H_n$ is an orbifold cover of degree four, $\chi(\HH_\CC^1/T_n) = 4 \chi(\HH_\CC^1/H_n) = 8-2n,$
and the genus of the surface is $g = n-3$.
Furthermore, the surface $\HH_\CC^1/T_n$ is hyperelliptic since the involution~$r_1$ fixes exactly $2(n-3)+2=2g+2$ points, namely the ones corresponding to $p_1$, $r_2p_1$, $p_i$,
$r_2p_i$, for $i=3,\ldots,n-1$. In fact, they are fixed by $r_1$ because the following identities hold in $\HH_\CC^1/T_n$:
\begin{alignat*}{2}
r_1 p_1 \simeq  p_1,& \quad &&r_1 (r_2p_1) \simeq (r_2r_1r_2r_1)^{-1} (r_2p_1) \simeq r_2p_1, \\
r_1 p_i  \simeq (r_1r_i) p_i \simeq p_i,& \quad &&r_1 (r_2p_i)  \simeq (r_2r_1r_2r_1)^{-1}(r_2r_1r_ir_2)(r_2p_i) \simeq r_2p_i.    
\end{alignat*}
Moreover, in this case, the surface has an
extra involution, $r_2$, that fixes exactly~$2$ points: $p_2$ and $r_1p_2$.

\subsection{{\rm PU(2,1)}-Representations of the hyperelliptic group}
\label{subsec:pu21-rep}
Denote by $\Omega\coloneq \{1,\omega,\omega^2\}$ the set of cube roots of unit.
As discussed in Subsection~\ref{subsec:isometries}, a nonidentical orientation-preserving 
involution  of $\HC$ is either a reflection in a point or a reflection in a complex line and 
has a representative in $\SU(2,1)$ of the form $R^p$,
where either $p\in\HC$ if it is a reflection in a point or $p\in\EV$
if it is a reflection in a complex line. Hence, given a representation
$\varrho:H_n\to\PU(2,1)$ with $\varrho(r_i)\neq 1$, there exist points
$p_i\in\PV\setminus\SV$ and $\delta\in\Omega$ such that $\varrho(r_i)=R^{p_i}$,
$i=1,\ldots,n$,
and $R^{p_n}\ldots R^{p_1}=\delta$ in $\SU(2,1)$. On the other hand, starting with a 
relation $R^{p_n}\ldots R^{p_1}=\delta$, we obtain a representation
$\varrho:H_n\to\PU(2,1)$ by defining $\varrho(r_i)\coloneq R^{p_i}$. In both cases we say that
the representation $\varrho$ and the relation $R^{p_n}\ldots R^{p_1}=\delta$
are {\it associated\/}.

We define the~{\it sign\/} of a non-isotropic $p$, denoted $\sigma p$, as $-1$ if 
$\langle p,p \rangle<0$ and as $+1$ if $\langle p,p \rangle>0$.
Given a list of signs $\Sigma\coloneq (\sigma_1,\ldots,\sigma_n)$
and a cube root of unity $\delta\in\Omega$, let 
$\mathcal P_{\Sigma,\delta}H_n$ be the space of representations
$\varrho:H_n\to\PU(2,1)$, with $\varrho(r_i)=R^{p_i}$, satisfying 
\begin{enumerate}[{\bf (P1)}]
\item  $\tance(p_i,p_j)\neq0,1$ if $i\neq j$ (indices considered modulo~$n$);
\item $\sigma p_i=\sigma_i$;
\item $R^{p_n}\ldots R^{p_1}=\delta$ in $\SU(2,1)$.
\end{enumerate}

\begin{rmk}
In {\bf (P1)}, the condition $\tance(p_i,p_j) \neq 0$ means that $p_i,p_j$ are non-orthogonal, and the condition $\tance(p_1,p_2)\neq 1$ means that $p_1,p_2$ are distinct and $\mathrm L\fgeo{p_1,p_2}$ is non-Euclidean.
\end{rmk}

Additionally, the space $\mathcal PH_n$ of representations $H_n \to \PU(2,1)$, $r_i\mapsto R^{p_i}$, satisfying only
{\bf (P1)} is disjoint union of $\mathcal P_{\Sigma,\delta}H_n$ 
varying $\Sigma$ and $\delta$ and the dimension of each nonempty component 
$\mathcal P_{\Sigma,\delta}H_n$ is equal to the one of $\mathcal PH_n$.

\begin{rmk}
The conjugation of a representation $H_n \to \PU(2,1)$, $r_i \mapsto R^{p_i}$, by $I \in \PU(2,1)$ is the representation $r_i \mapsto I R^{p_i}I^{-1}$. Since $I R^{p_i}I^{-1}= R^{I p_i}$, conjugating a representation with centers $p_i$ produces a new representation with centers $I p_i$. Geometrically, $\mathcal P_{\Sigma,\delta}H_n/\PU(2,1)$ can be seen as the space of $n$-tuples of points $(p_1,\ldots,p_n)$ modulo $\PU(2,1)$ satisfying {\bf (P1)}, {\bf (P2)}, and {\bf (P3)}.

The $\PU(2,1)$-character variety for the orbifold group~$H_n$ is 
$\mathop{\mathrm{Hom}}(H_n,\PU(2,1))/\PU(2,1)$. We have that $\mathcal PH_n/\PU(2,1)$ is composed by the representations in character variety 
satisfying our geometric restriction~{\bf (P1)} and is the disjoint union of the spaces $\mathcal P_{\Sigma,\delta}H_n/\PU(2,1)$. 
\end{rmk}

These objects are important for the following reason: consider $(\HH_\CC^1 \times \mathbb D)/H_n$, where $\mathbb D$ is an open disc on $\RR^2$ and the action of $H_n$ on $\HH_\CC^1 \times \mathbb D$ is of the form $g(x,f) = (gx,a(g,x)f)$. We assume that each $a(g,-): \HH_\CC^1 \times \mathbb D \to \mathbb D$ smooth satisfying $a(h,gx)a(g,x)= a(hg,x)$ and $a(1,x)=x$. This action defines the good disc orbibundle $(\HH_\CC^1 \times \mathbb D)/H_n \to \HH_\CC^1/H_n$, given by $[x,f] \mapsto [x]$. If the total space is a complex hyperbolic $4$-orbifold, then there exists $\Gamma \subset \PU(2,1)$ such that $(\HH_\CC^1 \times \mathbb D)/H_n \simeq \HC/\Gamma$, meaning that $H_n \simeq \Gamma \subset \PU(2,1)$. Thus, we have a representation $\varrho: H_n \to \PU(2,1)$. Conjugating $\varrho$ by $\PU(2,1)$ provides the same complex hyperbolic structure on the disc bundle. Thus, we can view disc orbibundles as points of the $\PU(2,1)$-character variety for $H_n$.

\subsection{Bendings and bending-deformations of representations}
\label{subsec:bendings}
When $p_1,p_2$ and $p_1',p_2'$ are pairs of points with $\tance(p_1,p_2)\neq 0,1$, 
$\tance(p_1',p_2')\neq0,1$, and $\sigma p_i=\sigma p_i'$, a relation 
of the form $R^{p_2'}R^{p_1'}=R^{p_2}R^{p_1}$ is called a {\it bending relation\/} (see~\cite{Sasha2012,spell}). 
By~\cite[Theorem~4.7]{spell} (see also \cite[Proposition~2.6]{Sasha2012}), every 
such relation has the following description: 
$p_1'$ and $p_2'$ are obtained by moving $p_1$ and $p_2$ along the geodesic 
$\mathrm G\fgeo{p_1,p_2}$,  maintaining the tance between them, that is, maintaining $\tance(p_1,p_2)=\tance(p_1',p_2')$.
In the case where $\sigma p_1=\sigma p_2$, 
bending moves $p_1$ and $p_2$ over $\mathrm G\fgeo{p_1,p_2}$ preserving the
distance between~$p_1$ and~$p_2$. If $\sigma p_1\neq \sigma p_2$, denoting
by $\tilde p_1$ the point in $\mathrm L\fgeo{p_1,p_2}$ that is orthogonal 
to $p_1$, bending moves $p_1$ and $p_2$ preserving the distance between
$\tilde p_1$ and $p_2$. (For the distance between points in~$\EV$, see Subsection~\ref{subsec:complex-plane}.)

More precisely, given a product of reflections $R^{p_2}R^{p_1}$, where $\tance(p_1,p_2)\neq 0,1$, there is 
a one-parameter subgroup $B:\mathbb R\to\SU(2,1)$ such that:
\begin{itemize}
\item $B(s)$ lies in the centralizer of $R^{p_2}R^{p_1}$ for every $s\in \RR$, meaning that
$B(s)R^{p_2}R^{p_1}B(s)^{-1}=R^{p_2}R^{p_1}$. Since $B(s)R^{p_2}R^{p_1}B(s)^{-1} = R^{B(s)p_2}R^{B(s)p_1}$, we have
$$R^{B(s)p_2}R^{B(s)p_1} = R^{p_2}R^{p_1};$$
\item if $p_1',p_2'\in\PV\setminus\SV$ satisfy $R^{p_2'}R^{p_1'}=R^{p_2}R^{p_1}$ and
$\sigma p_i'=\sigma p_i$, for $i=1,2$,
then $p_1'=B(s)p_1$ and $p_2'=B(s)p_2$ for some $s \in \RR$.
\end{itemize}
Elements $B(s)$ of the one-parameter subgroup of the centralizer of 
$R\coloneq R^{p_2}R^{p_1}$ are called {\it bendings of $R$\/}.

\smallskip

Now, suppose we have a representation $\varrho\in\mathcal P_{\Sigma,\delta}H_n$ and consider
the associated relation 
$$R^{p_n}\ldots R^{p_{i+2}}R^{p_{i+1}}R^{p_i}R^{p_{i-1}}\ldots R^{p_1}=\delta$$
given by~{\bf (P3)}. Since $\varrho$ satisfies {\bf (P1)},
we can consider a nonidentical bending $B(s)$ of $R^{p_{i+1}}R^{p_i}$. Defining
$p_i'\coloneq B(s)p_i$ and $p_{i+1}'\coloneq B(s)p_{i+1}$ we have $\sigma p_i'=\sigma p_i$, 
$\sigma p_{i+1}'=\sigma p_{i+1}$, and
$R^{p_{i+1}'}R^{p_i'}=R^{p_{i+1}}R^{p_i}$. Hence, substituting
$R^{p_{i+1}}R^{p_i}$ by $R^{p_{i+1}'}R^{p_i'}$ in the given relation, we obtain a new 
relation, namely 
$$R^{p_n}\ldots R^{p_{i+2}}R^{p_{i+1}'}R^{p_i'}R^{p_{i-1}}\ldots R^{p_1}=\delta.$$
Note that for small parameters $s$, the representation given by the new relation
will also satisfy {\bf (P1--P3)}, and therefore we have a new representation 
$\widehat\varrho\in\mathcal P_{\Sigma,\delta}H_n$. 
We say the relation $\widehat\varrho$ is a {\it bending-deformation\/} of~$\varrho$.

\subsection{The dimension of the representation space}
\label{subsec:dimension}
From Subsection~\ref{subsec:hyperbolic-spaces}, a tangent vector
$t\in T_p\PV$ at a point~$p\in\PV\setminus\SV$ is a $\CC$-linear transformation
$\mathbb Cp\to p^\perp$. In another words, for an arbitrary representative $p \in V \setminus \{0\}$ we have
$$t = \frac{\langle \Id,p\rangle}{\langle p,p \rangle}t(p)$$
with $t(p) \in p^\perp$. Note that $t$ can be seen as a $\CC$-linear endomorphism of $V$ by defining $t=0$ on $p^\perp$. The {\it adjoint} of $t$ is the unique $\CC$-linear endomorphism of $V$ satisfying $\langle tx,y \rangle = \langle x,t^\ast y\rangle$ for $x,y \in V$. More explicitly, the adjoint of $t$ can be written as
$$t^\ast = \frac{\langle \Id, t(p) \rangle}{\langle p,p \rangle} p.$$

\begin{lemma} 
\label{lemma:diffr}
For\/ $p\in\PV\setminus\SV$ and\/ $t\in T_{p}\PV$,
$$\left.\frac{d}{d\varepsilon}\right|_{\varepsilon=0}R^{p+\varepsilon t(p)}=
2(t+t^\ast).$$
\end{lemma}

\begin{proof}
Note that
\begin{align*}
R^{p+\varepsilon t(p)}&= 
2\frac{\langle\Id,p+\varepsilon t(p)\rangle}
{\langle p+\varepsilon t(p),p+\varepsilon t(p)\rangle}(p+\varepsilon t(p))-\Id\\
&=2\Big[\frac{\langle\Id,p\rangle+\varepsilon\langle \Id,t(p)\rangle}
{\langle p,p \rangle+\varepsilon^2\langle t(p),t(p)\rangle}\Big]p
+2\varepsilon\Big[\frac{\langle\Id,p\rangle+\varepsilon\langle\Id,t(p)\rangle}
{\langle p,p \rangle+\varepsilon^2\langle t(p),t(p)\rangle}\Big]t(p)-\Id\\
&=- \Id+2 \frac{\langle\Id,p\rangle}{\langle p,p \rangle} p +2\varepsilon\left(\frac{\langle \Id,t(p)\rangle}{\langle p,p \rangle} p + \frac{\langle\Id,p\rangle}{\langle p,p \rangle} t(p) \right) + o(\varepsilon),
\end{align*}
where we used that $$\frac{1}{\langle p,p \rangle+\varepsilon^2 \langle t(p),t(p) \rangle } = \frac1{\langle p,p \rangle} + o(\varepsilon).$$
Therefore
$$\left.\frac{d}{d\varepsilon}\right|_{\varepsilon=0}R^{p+\varepsilon t(p)} = 2\langle\Id,v\rangle p + 2\langle\Id,p\rangle v = t+t^\ast$$
and the result follows.
\end{proof}

\begin{lemma}
\label{lemma:proptast}
Given\/ $p\in\PV\setminus\SV$ 
and\/ $t\in T_p\PV$,
$$t+t^\ast=(t-t^\ast)R^p\ \ \text{and\/}\ \ (t-t^\ast)R^p+R^p(t-t^\ast)=0.$$
\end{lemma}

\begin{proof}
For $x\in V$
\begin{align*}
(t-t^\ast)R^p x &= \frac1{\langle p,p \rangle }\bigg\langle2\frac{\langle x,p\rangle}{\langle p,p\rangle}p-x,p\bigg\rangle t(p)
- \frac1{\langle p,p \rangle }\bigg\langle2\frac{\langle x,p\rangle}{\langle p,p\rangle}p-x,t(p)\bigg\rangle p\\
&= \frac{\langle x, p\rangle}{\langle p,p \rangle } t(p)+\frac{\langle x,t(p)\rangle}{\langle p,p \rangle }p\\
&= (t+t^\ast)x.
\end{align*}
Also,
\begin{align*}
R^p(t-t^\ast)x&=R^p\left(\frac{\langle x,p\rangle}{\langle p,p \rangle} t(p)-\frac{\langle x,t(p)\rangle}{\langle p,p \rangle} p \right)\\
&=\frac{\langle x,p\rangle}{\langle p,p \rangle} R^p t(p)-\frac{\langle x,t(p)\rangle}{\langle p,p \rangle} R^p p\\
&=\frac{\langle x,p\rangle}{\langle p,p \rangle} \left(2 \frac{\langle t(p),p \rangle}{\langle p,p \rangle }p-t(p) \right)-\frac{\langle x,t(p)\rangle}{\langle p,p \rangle} \left(2 \frac{\langle p,p \rangle}{\langle p,p \rangle }p - p \right)\\
&= -(t+t^\ast)x\\
&=-(t-t^\ast)R^p x.
\end{align*}
where we used the first part in the last identity.
\end{proof}

Let $n\geq 5$ and consider the function $\Phi:(\PV\setminus\SV)^n\to\SU(2,1)$ given by
$$\Phi(p_1,\ldots, p_n)\coloneq R^{p_n}\ldots R^{p_1}.$$

\begin{prop}
\label{prop:dim-rep-su21}
Let\/ $\delta\in\Omega$. If
there exists\/ $(p_1,\ldots,p_n)\in\Phi^{-1}(\delta)$ such that\/ $\tance(p_1,p_n) \neq 0$ and\/ $\tance(p_1,p_n) \neq 1$,
then\/ $\Phi^{-1}(\delta)$ has\/ {\rm(}real\/{\rm)} dimension\/~$4n-8$.
\end{prop}
\begin{rmk}
The condition $\tance(p_1,p_n) \neq 1$ implies that $p_1 \neq p_n$ and that the projective line $\mathrm L\fgeo{p_1,p_n}$ is either spherical or hyperbolic.
\end{rmk}
\begin{proof}
The proof here is analogous to the one in~\cite{Gaye2008} (but avoiding coordinates). 
We prove that $d\Phi$ has maximal ranking in points of $\Phi^{-1}(\delta)$.
As consequence, the result will follow because~$\dim\SU(2,1)=8$ and $\dim(\PV\setminus SV)^n=4n$.

Fix representatives for the points $p_1,\ldots, p_n$.
Given $t\in T_{p_1}\PV\subset\mathop{\mathrm{Lin}}_{\mathbb R}(V,V)$,
let $\gamma_t$ be the curve in $(\PV\setminus\SV)^n$ given by 
$$\gamma_t(\varepsilon)=
(p_1+\varepsilon t(p_1),p_2\ldots,p_n).$$ 

From the product rule and the previous remark,
$$\left.\frac{d}{d\varepsilon}\right|_{\varepsilon=0}\Phi(\gamma_t(\varepsilon))=
\delta\left.\frac{d}{d\varepsilon}\right|_{\varepsilon=0}R^{p_1}R^{p_1+\varepsilon t(p_1)}
=2\delta(t^\ast-t).$$
In the same way (but without using the remark), for $s\in T_{p_n}\PV$, we consider the curve 
$$\lambda_s(\varepsilon)=
(p_1,\ldots,p_{n-1},p_n+\varepsilon s(p_n)),$$
and
$$\left.\frac{d}{d\varepsilon}\right|_{\varepsilon=0}\Phi(\lambda_s(\varepsilon))=
\delta\left.\frac{d}{d\varepsilon}\right|_{\varepsilon=0}R^{p_n+\varepsilon s(p_n)}R^{p_n}
=2\delta(s-s^\ast).$$

Since, by hypothesis, $\mathrm{L}\fgeo{p_1,p_n}$ is noneuclidean, there exists
a nonisotropic point $v\in \PV$ such that $\mathbb Pv^\perp=\mathrm L\fgeo{p_1,p_n}$. Denote by~$c$ 
the point in $\mathbb P p_1^\perp$ that is orthogonal to $v$, and by~$d$
the point in $\mathbb Pp_n^\perp$ that is orthogonal to~$v$, where we are considering ortogonality with respect to the Hermitian form. Consider representatives for $v$, $c$, and $d$
(also denoted $v$, $c$, and $d$). Note that
$c$ and $d$ are nonisotropic because $p_1, v, c$ and $p_n, v, d$ are orthogonal basis for $V$ and the Hermitian form has signature $-++$. Defining
$$t_1\coloneq \frac{\langle-,p_1\rangle}{\langle p_1,p_1\rangle} v,\ \ t_2\coloneq \frac{\langle-,p_1\rangle}{\langle p_1,p_1\rangle} iv,\ \ 
t_3\coloneq \frac{\langle-,p_1\rangle}{\langle p_1,p_1\rangle} c,\ \ t_4\coloneq \frac{\langle-,p_1\rangle}{\langle p_1,p_1\rangle} ic,$$
and 
$$s_1\coloneq \frac{\langle-,p_n\rangle}{\langle p_n,p_n\rangle} v,\ \ s_2\coloneq \frac{\langle-,p_n\rangle}{\langle p_n,p_n\rangle} iv,\ \ 
s_3\coloneq \frac{\langle-,p_n\rangle}{\langle p_n,p_n\rangle} d,\ \ s_4\coloneq \frac{\langle-,p_n\rangle}{\langle p_n,p_n\rangle} id,$$
we obtain orthogonal bases $\{t_1,t_2,t_3,t_4\}$ and $\{s_1,s_2,s_3,s_4\}$ 
for $T_{p_1}\PV$ and $T_{p_n}\PV$, respectively. Lets prove that the set
\begin{multline*}
\bigg\{\left.\frac{d}{d\varepsilon}\right|_{\varepsilon=0}\Phi(\gamma_{t_1}(\varepsilon)),
\left.\frac{d}{d\varepsilon}\right|_{\varepsilon=0}\Phi(\gamma_{t_2}(\varepsilon)),
\left.\frac{d}{d\varepsilon}\right|_{\varepsilon=0}\Phi(\gamma_{t_3}(\varepsilon)),
\left.\frac{d}{d\varepsilon}\right|_{\varepsilon=0}\Phi(\gamma_{t_4}(\varepsilon)),\\
\left.\frac{d}{d\varepsilon}\right|_{\varepsilon=0}\Phi(\lambda_{s_1}(\varepsilon)),
\left.\frac{d}{d\varepsilon}\right|_{\varepsilon=0}\Phi(\lambda_{s_2}(\varepsilon)),
\left.\frac{d}{d\varepsilon}\right|_{\varepsilon=0}\Phi(\lambda_{s_3}(\varepsilon)),
\left.\frac{d}{d\varepsilon}\right|_{\varepsilon=0}\Phi(\lambda_{s_4}(\varepsilon))
\bigg\},
\end{multline*}
or, equivalently, the set
$$\big\{t_1^\ast-t_1,t_2^\ast-t_2,t_3^\ast-t_3,
t_4^\ast-t_4,s_1-s_1^\ast,s_2-s_2^\ast,s_3-s_3^\ast,
s_4-s_4^\ast\big\}$$
is a linearly independent subset of $\mathfrak{su}(2,1)$.

\medskip

Note that $\langle c,d \rangle\neq 0$, otherwise we would have $c=p_n$ and $d=p_1$ in $\HH_\CC^2$, leading to $\tance(p_1,p_n) = 0$. Similarly, $\langle p_n, c \rangle \neq 0$, otherwise we would have $p_n^\perp = \CC c+\CC v=p_1^\perp$ and, as consequence, $\tance(p_1,p_n) = 1$.

\medskip

Suppose that
\begin{equation}
\label{eq:lincomb}
a_1(t_1^\ast-t_1)+\cdots+a_4(t_4^\ast-t_4)+b_1(s_1-s_1^\ast)+\cdots+b_4(s_4-s_4^\ast)=0,
\end{equation}
for some $a_i,b_i\in\mathbb R$. Applying~\eqref{eq:lincomb} to~$v$, we get
$$\big(a_1-i a_2\big)p_1+(-b_1+ib_2)p_n=0,$$
which implies $a_1=a_2=b_1=b_2=0$ since $v$ in nonisotropic, $p_1$ and $p_n$
are distinct points in~$\PV$, and $a_i,b_j\in\mathbb R$.

Applying to~$c$, we get
$$\big(a_3\langle c,c\rangle -ia_4\langle c,c \rangle\big)p_1=
\big(b_3\langle c,d\rangle -ib_4\langle c,d\rangle\big)p_n-
\big(b_3\langle c,p_n\rangle+ib_4\langle c,p_n\rangle)d$$
which implies, by taking the product $\langle -,v\rangle$, that
$b_3=b_4=0$, since $\langle c,p_n\rangle\neq 0$. Which, in its turn, implies
that $a_3=a_4=0$ from the previous equation.
\end{proof}

\begin{cor}
\label{cor:dim-pdelta}
Given a list of signs\/ $\Sigma$, if\/ $\mathcal P_{\Sigma,\delta}H_n$ is nonempty, then\/ 
$$\dim\mathcal P_{\Sigma,\delta}H_n=4n-8\quad\text{and}\quad \dim\mathcal P_{\Sigma,\delta}H_n/\PU(2,1)=4n-16.$$
\end{cor}

\section{Decomposing isometries as the product of three reflections}
\label{sec:decom-iso}
The goal of this section is to construct a length~$5$ relation between involutions in 
a way that it produces a representation $\varrho\in\mathcal P_{\Sigma,\delta}$, 
for $\Sigma = (+,-,-,-,-)$ and $\delta\in\Omega$, as defined in 
Section~\ref{sec:hyper}.

\subsection{The relative character variety}
\label{subsec:relative}
Let $\pmb\sigma\coloneq (\sigma_1,\sigma_2,\sigma_3)$ be a triple of signs with at most one $+1$ and let
$\tau\in\mathbb C\setminus\Delta$ be a point outside Goldman's deltoid (the trace of a loxodromic isometry).
We will say that a triple of nonisotropic points $(p_1,p_2,p_3)$ is {\it strongly regular\/} 
with respect to $\pmb\sigma,\tau$ or,
for short, that $(p_1,p_2,p_3)$ is a $(\pmb\sigma,\tau)${\it -triple\/} if 
\begin{itemize}
\item  $p_1,p_2,p_3$ are pairwise nonorthogonal and do not lie
in the same complex line,
\item $\sigma p_i=\sigma_i$, $i=1,2,3$,
\item $\trace R^{p_3}R^{p_2}R^{p_1}=\tau$.
\end{itemize}
This is a simplification of~\cite[Definition~5.1]{spell} for the
case of involutions. We will call a triple satisfying only the first item above a {\it regular\/} triple. 

\smallskip

Given $\pmb\sigma$ and $\tau$ as above, and given a triple of nonisotropic points $(p_1,p_2,p_3)$, 
define
\begin{equation*}
\begin{gathered}
\kappa\coloneq \frac{\tau-3}{8},\quad s_1\coloneq \tance(p_1,p_2),\quad s_2\coloneq \tance(p_2,p_3),
\quad s_3\coloneq \tance(p_3,p_1),\\
\eta\coloneq \frac{g_{12}g_{23}g_{31}}{g_{11}g_{22}g_{33}},\quad \text{and}\quad
s\coloneq \real\eta,
\end{gathered}
\end{equation*}
where $[g_{ij}]$ is the Gram matrix of $(p_1,p_2,p_3)$.
By~\cite[Theorem~5.2]{spell}, if $(p_1,p_2,p_3)$ is a $(\pmb\sigma,\tau)$-triple then
\begin{equation}
\label{eq:surface}
s_1^2s_2+s_1s_2^2-2s_1s_2s+s^2+(2\real\kappa)s_1s_2+(\imag\kappa)^2=0,
\end{equation}
and
\begin{equation}
\label{eq:inequalities-surface}
\begin{gathered}
\sigma_1\sigma_2 s_1>0,\quad \sigma_1\sigma_2 s_1>\sigma_1\sigma_2,\quad \sigma_2\sigma_3 s_2>0,\quad
\sigma_2\sigma_3 s_2>\sigma_2\sigma_3,\\ 
\sigma_1\sigma_2\sigma_3(2\real\kappa+1)<0.
\end{gathered}
\end{equation}

\begin{rmk}
Identity~\eqref{eq:surface} follows from $\trace R^{p_3}R^{p_2}R^{p_1}=\tau$ while inequalities~\eqref{eq:inequalities-surface} follow from $\sigma_i\sigma_j\tance(p_i,p_j)>0$ and $\det[g_{ij}]<0$.
\end{rmk}

Conversely, if $(s_1,s_2,s)\in\mathbb R^3$ is a solution of~\eqref{eq:surface} satisfying~\eqref{eq:inequalities-surface},
then any triple of nonisotropic points $(p_1,p_2,p_3)$ with Gram matrix
given by $g_{11}\coloneq \sigma_1$, $g_{22}\coloneq \sigma_2$, $g_{33}\coloneq \sigma_3$, $g_{12}=g_{21}\coloneq \sqrt{\sigma_1\sigma_2 s_1}$,
$g_{23}=g_{32}\coloneq \sqrt{\sigma_2\sigma_3 s_2}$,
$$g_{13}\coloneq \sigma_1\sigma_2\sigma_3\frac{s-i\imag\kappa}{\sqrt{\sigma_1\sigma_2 s_1}\sqrt{\sigma_2\sigma_3 s_2}},$$
and $g_{31}\coloneq \overline{g_{13}}$, is a $(\pmb\sigma,\tau)$-triple 
and, in particular, $\trace R^{p_3}R^{p_2}R^{p_1}=\tau$.

Hence, the surface $\mathcal S_{\pmb\sigma,\tau}$ given by equation~\eqref{eq:surface} and inequalities~\eqref{eq:inequalities-surface}  
is the space of all regular triples $(p_1,p_2,p_3) \in \PV$ modulo $\PU(2,1)$ satisfying
$\trace R^{p_3}R^{p_2}R^{p_1}=\tau$ and $\sigma p_i=\sigma_i$, $i=1,2,3$.

\begin{figure}[htb]
\centering
\includegraphics[scale=.8]{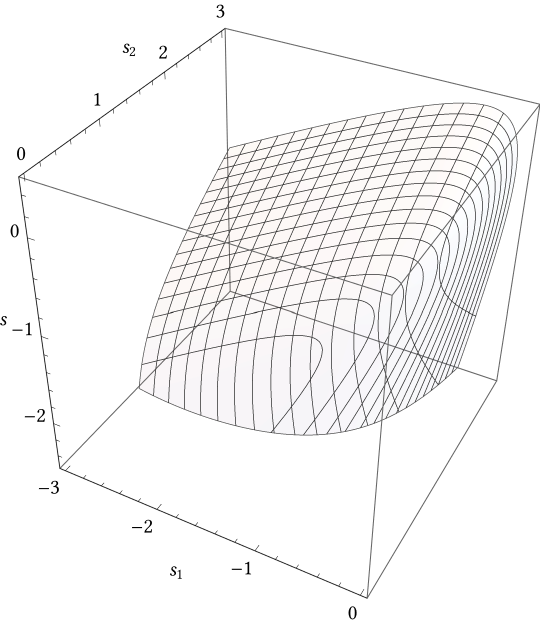}
\caption{Surface $\mathcal S_{\pmb\sigma,\tau}$ in $\RR^3$ obtained by
taking $\pmb\sigma\coloneq (+,-,-)$ and $\tau\coloneq -2.22 - 3.84515\, i$. Vertical (resp. 
horizontal) lines obtained by keeping $s_1$ (resp. $s_2$) constant are marked.}
%The blue lines obtained by keeping $s_1$ constant and the red lines are obtained by keeping $s_2$ constant.} 
\label{fig:surf-sigma-tau}
\end{figure}

%{\color{Purple} By bending $R^{p_2}R^{p_1}$ we are moving along the blue lines and by bending $R^{p_3}R^{p_2}$ we are moving along the red lines. Observe that red lines and blue lines might intersect twice, meaning that we can do two consecutive non-trivial bendings and comeback to the original point.}

\subsection{Bending deformations of regular triples}
\label{subsec:bend-triples}

Fixed a $(\pmb\sigma,\tau)$-triple $(p_1,p_2,p_3)$, 
we also fix the isometry $F\coloneq R^{p_3}R^{p_2}R^{p_1}$ which, by construction, 
satisfies~$\trace F=\tau$. Bendings, as described in Subsection~\ref{subsec:bendings}, act on 
this triple in the following way: if~$B(s)$ is a bending of $R^{p_2}R^{p_1}$ 
(resp. of $R^{p_3}R^{p_2}$) we have that 
$F=R^{p_3}R^{B(s)p_2}R^{B(s)p_1}$ and $(B(s)p_1,B(s)p_2,p_3)$ is a $(\pmb\sigma,\tau)$-triple
(resp. $F=R^{B(s)p_3}R^{B(s)p_2}R^{p_1}$ and $(p_1,B(s)p_2,B(s)p_3)$ is a $(\pmb\sigma,\tau)$-triple). 
Thus, considering that a $(\pmb\sigma,\tau)$-triple gives us a point 
in~$\mathcal S_{\pmb\sigma,\tau}$, bendings also
act on~$\mathcal S_{\pmb\sigma,\tau}$ moving points either vertically (over slices
given by $s_1=\text{const}$) or horizontally (over ones given by 
$s_2=\text{const}$), and this movements are performed by bendings of 
$R^{p_2}R^{p_1}$ and $R^{p_3}R^{p_2}$, respectively.

It follows that all points in the surface $\mathcal S_{\pmb\sigma,\tau}$ correspond to 
a certain bendings deformation of the fixed triple $(p_1,p_2,p_3)$. But, this 
correspondence is not biunivocal. In fact, as described in~\cite[Section~4.2]{Sasha2012}, 
after finitely many bendings of $(p_1,p_2,p_3)$, we can end up with a distinct strongly 
regular triple $(p_1',p_2',p_3')$ that is geometrically equal to the starting one.
Roughly speaking, after walking in a closed path composed of vertical and horizontal segments
on $\mathcal S_{\pmb\sigma,\tau}$, we end up in the same point in the surface, but in a distinct 
corresponding strongly regular triple.
Moreover, in Figure~\ref{fig:surf-sigma-tau} we see that vertical and horizontal lines 
might intersect twice, meaning that in some cases it is possible that after two consecutive 
non-trivial bendings we comeback to the original point in~$\mathcal S_{\pmb\sigma.\tau}$, 
but to a distinct strongly regular triple.
In Subsection~\ref{subsec:comp-bend} we exhibit a concrete example of this 
phenomena. 

\subsection{The construction of length 5 relations between reflections}
When does, for given $\pmb\sigma$ and $\tau$, there is a solution for~\eqref{eq:surface}
subjected to~\eqref{eq:inequalities-surface}? Note that 
$\sigma_1\sigma_2\sigma_3(2\real\kappa+1)=\det[g_{ij}]$. Thus, the last condition 
in~\eqref{eq:inequalities-surface} means that $p_1,p_2,p_3$ do not lie in the same complex 
line. Moreover, when $\sigma_1\sigma_2\sigma_3(2\real\kappa+1)>0$, our system has 
no solution.

\begin{figure}[htb!]
\centering
\includegraphics[scale=.9]{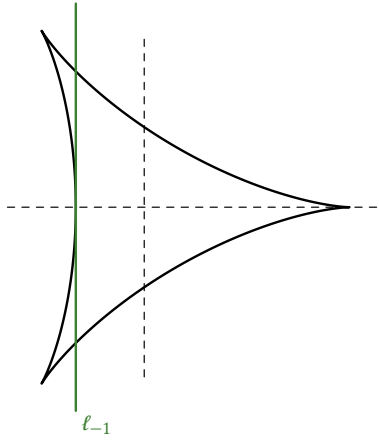}
\caption{Goldman's deltoid $\Delta$ and the line $\ell_{-1}$ tangent to $\Delta$ } 
\label{fig:ell1delt}
\end{figure}

But $\sigma_1\sigma_2\sigma_3(2\real\kappa+1)=0$ if and only if $\real\tau=-1$, that is,
if and only if~$\tau$ lies in the line~$\ell_{-1}$, tangent to~$\partial\Delta$ at~$-1$.
Moreover,

\begin{itemize}
\item If $\sigma_1\sigma_1\sigma_3=+1$, then $\sigma_1\sigma_2\sigma_3(2\real\kappa+1)<0$ if 
and only if $\tau$ lies to the left side of~$\ell_{-1}$;
\item If $\sigma_1\sigma_2\sigma_3=-1$, then $\sigma_1\sigma_2\sigma_3(2\real\kappa+1)<0$ 
if and only if $\tau$ lies to the right side of~$\ell_{-1}$ .
\end{itemize}
When $\pmb\sigma$ and $\tau$ respect the restriction of the items above,
there exists a solution of~\eqref{eq:surface} satisfying~\eqref{eq:inequalities-surface} (see~\cite[Proposition~6.2]{alpha}).
And, of course, given a loxodromic isometry, we can always consider a 
lift of it in $\SU(2,1)$ with trace lying to the left side of~$\ell_{-1}$.

Therefore, if $F\in\SU(2,1)$ is a loxodromic isometry and  
$\pmb\sigma=(\sigma_1,\sigma_2,\sigma_3)$ is a triple of signs with at most one $+1$, 
there is $\delta\in\Omega$ and a regular triple $(p_1,p_2,p_3)$ such that
$$\delta F=R^{p_3}R^{p_2}R^{p_1},$$
and $\sigma p_i=\sigma_i$, $i=1,2,3$. 
In particular,  we obtain a relation length~5 relation between reflection
\begin{equation}
\label{eq:relation}
R^{p_5}R^{p_4}R^{p_3}R^{p_2}R^{p_1}=\delta,
\end{equation}
by taking $F\coloneq R^{p_4}R^{p_5}$, where $p_4,p_5\in\PV$ are distinct nonisotropic points 
generating a hyperbolic line. In fact, as remarked in page \pageref{rmk:tance and line}, if the points $p_4,p_5$ define a hyperbolic line, then $\trace F= 4 \tance(p_4,p_5)-1>3$ and $F$ is loxodromic as consequence.

\begin{prop}
\label{prop:existence}
If\/ $p_4,p_5\in\BV$ are distinct points, then there exist\/ $p_1\in\EV$ and
$p_2,p_3\in\BV$ satisfying the relation\/~\eqref{eq:relation} 
with\/~$\delta=\omega^2$.
\end{prop}

\begin{proof}
Take $\sigma_1=+1$, $\sigma_2=\sigma_3=-1$, and $\pmb\sigma=(\sigma_1,\sigma_2,\sigma_3)$. 
The isometry $R^{p_4}R^{p_5}$ is hyperbolic
(i.e., real loxodromic) with $\trace R^{p_4}R^{p_5}=4\tance(p_4,p_5)-1>3$ 
since $\tance(p_4,p_5)>1$. That is, 
$\trace R^{p_4}R^{p_5}$ lies to the right side of $\ell_{-1}$. So $\tau\coloneq \omega^2\trace R^{p_4}R^{p_5}$
lies to the left side of $\ell_{-1}$ and, since $\sigma_1\sigma_2\sigma_3=+1$, there exists a 
strongly regular triple $(p_1',p_2',p_3')$ with respect to $\pmb\sigma,\tau$. Therefore, 
there exists $I\in\SU(2,1)$ such that $IR^{p_3'}R^{p_2'}R^{p_1'}I^{-1}=\omega^2 R^{p_4}R^{p_5}$.
The result follows by defining $p_i=Ip_i'$, $i=1,2,3$.
\end{proof}

\begin{cor}
Let\/ $\Sigma=(+1,-1,-1,-1,-1)$.
\begin{itemize}
\item If\/ $\delta=1$, then\/ $\mathcal P_{\Sigma,\delta}H_5=\varnothing$.
\item If\/ $\delta=\omega$ or\/ $\delta=\omega^2$, then\/ 
$\mathcal P_{\Sigma,\delta}H_5\neq\varnothing$.
\end{itemize}
\end{cor}

\begin{proof}
The first item follows from the fact that if $R^{p_5}R^{p_4}R^{p_3}R^{p_2}R^{p_1}=1$,
then $\trace R^{p_3}R^{p_2}R^{p_1}=\trace R^{p_4}R^{p_5}$ which is a real positive
number since, by hypothesis, $\sigma p_4=\sigma p_5=-1$ and, in this case,
$\sigma_1\sigma_2\sigma_3(2\real\kappa+1)>0$, so the configuration is impossible.

Now suppose that $\delta=\omega^2$. From Proposition~\ref{prop:existence},
there are $p_1,p_2,p_3,p_4,p_5$ such that $R^{p_5}R^{p_4}R^{p_3}R^{p_2}R^{p_1} = \omega^2$ 
with $p_1$ positive and the other points negative. What remains to be proved is that
these points satisfy condition {\bf (P1--P3)} of Subsection~\ref{subsec:pu21-rep}.
By construction $\tance(p_1,p_2),\tance(p_1,p_3),\tance(p_2,p_3) \neq 0,1$ because 
$(p_1,p_2,p_3)$ is a regular triple. If $p_4=p_5$, then $R^{p_3}R^{p_2}R^{p_1}=\omega^2$
which implies that $p_1,p_2,p_3$ are pairwise orthogonal (see \cite[Section~3]{spell}), a contradiction.
A similar contradiction is obtained if we suppose $p_3=p_4$. If $p_2=p_4$,
the relation assumes the form $R^{p_5}R^{p_2}R^{p_3}R^{p_2}R^{p_1}=\omega^2$
implying $R^{p_5'}R^{p_3}R^{p_1'}=\omega^2$, where $p_5'\coloneq R^{p_2}p_5$ and $p_1'=R^{p_2}p_1$,
and it follows that $p_3$ and $p_5'$ are orthogonal, which is a contradiction since
$\sigma p_3=\sigma p_5'=-1$. The arguments for why we cannot have $p_2=p_5$ or 
$p_3=p_5$ are analogous.

For the nonorthogonality, if $p_4$ is orthogonal to $p_1$, then $R^{p_4}=R^cR^{p_1}$
where $c$ is the point orthogonal to both $p_1$ and $p_4$ (seee \cite[Section~3]{spell}). So, 
$R^{p_5}R^cR^{p_3'}R^{p_2'}=\omega^2$ where $p_3'\coloneq R^{p_1}p_3$ and $p_2'\coloneq R^{p_1}p_2$.
By~\cite[Lemma~4.1]{spell}, the complex lines $\mathrm L(p_2',p_3')$ and
$\mathrm L(c,p_5)$ are orthogonal. Being both hyperbolic lines, it follows
that $R^{p_3'}R^{p_2'}$ fixes their intersection, which is a negative point. This 
is a contradiction, since $\sigma p_2'=\sigma p_3'=-1$ and thus
$R^{p_3'}R^{p_2'}$ is loxodromic.

Finally, if $\delta=\omega$, we just notice that if
$R^{p_5}R^{p_4}R^{p_3}R^{p_2}R^{p_1}=\omega^2$, then 
$R^{p_2}R^{p_3}R^{p_4}R^{p_5}R^{p_1}=\omega$ and this relation
respects the list of signs~$\Sigma$.
\end{proof}

The previous proof shows us that any relation~\eqref{eq:relation} with $\delta=\omega$ or 
$\delta=\omega^2$ and where the signs of points are given by $\Sigma=(+1,-1,-1,-1,-1)$,
produces a representation in~$\mathcal P_{\Sigma,\delta}H_5$. Moreover, by 
Corollary~\ref{cor:dim-pdelta}, if $\Sigma = (+1,-1,-1,-1,-1)$ and 
$\delta\neq 1$, then $\dim\mathcal P_{\Sigma,\delta}H_5/\PU(2,1)=4$.

\begin{prop}[Theorem~7.3~\cite{spell}] 
\label{prop:bending-connected}
Suppose\/ $n=5$. If at most one of the signs in\/ $\Sigma$ is\/ $+1$ and\/ 
$\mathcal P_{\Sigma,\delta}H_5$ is not empty, then it is bending-connected. \
That is, given two representations $\varrho_1,\varrho_2\in \mathcal P_{\Sigma,\delta}H_5$,
we have that $\varrho_2$ is obtained by a finite amount of bending-deformations
of~$\varrho_1$.
\end{prop}

\begin{rmk}
Given a relation~\eqref{eq:relation} associated with a representation
in $\mathcal P_{\Sigma,\delta}H_5$ where $\Sigma$ has at most one positive
sign, we notice that bending the triple $(p_1,p_2,p_3)$ into a geometrically equal
triple $(p_1',p_2',p_3')$ is equivalent,
module $\PU(2,1)$, to bending $R^{p_5}R^{p_4}$. In fact, if $I\in\SU(2,1)$ is the isometry
that sends $(p_1,p_2,p_3)$ to $(p_1',p_2',p_3')$, then since $(p_1',p_2',p_3')$ is
a bending deformation of $(p_1,p_2,p_3)$, it follows that
$$R^{p_5}R^{p_4}R^{p_3}R^{p_2}R^{p_1}=R^{p_5}R^{p_4}R^{p_3'}R^{p_2'}R^{p_1'}
= R^{p_5}R^{p_4}I R^{p_3}R^{p_2}R^{p_1} I^{-1}=\delta.$$
This implies that $I^{-1}R^{p_5}R^{p_4}IR^{p_3}R^{p_2}R^{p_1}=\delta$ and, since
$I$ is in the centralizer of $R^{p_5}R^{p_4}$, the starting relation is equal modulo
$\PU(2,1)$ to the relation $R^{I^{-1}p_5}R^{I^{-1}p_4}R^{p_3}R^{p_2}R^{p_1}=\delta$
which is obtained from the starting one by a bending of $R^{p_5}R^{p_4}$.
\end{rmk}

\section{Discreteness}
\label{sec:discreteness}

In this section, given a length~$5$ relation between orientation-preserving involutions of~$\HC$ 
of the form
\begin{equation}
\label{eq:relation-omega2}
R^{p_5}R^{p_4}R^{p_3}R^{p_2}R^{p_1}=\omega^2,\quad p_1\in\EV,\quad p_2,p_3,p_4,p_5\in\HC,
\end{equation}
whose existence is guaranteed 
by~Proposition~\ref{prop:existence}, we discuss the discreteness of 
$\langle R^{p_2},R^{p_3},R^{p_4},R^{p_5}\rangle$ or, equivalently, the 
discreteness and faithfulness of the representation $\varrho:H_5\to\PU(2,1)$ defined by 
$\varrho(r_i)=R^{p_i}$. We introduce a {\it quadrangle of bisectors\/} as a 
candidate for the fundamental domain of such group and, following \cite{yet} and \cite{AGG2011}, 
describe conditions over such quadrangle so that the group is discrete. Since all these conditions
are open, that is, given by strict inequalities, if we find one such relation that produces a discrete
subgroup of~$\PU(2,1)$, any small bending-deformation (as in Subsection~\ref{subsec:bendings}) of the relation will also produce a
discrete subgroup.

The existence of a relation producing a discrete subgroup is obtained computationally in 
Section~\ref{sec:computational}.

\subsection{Quadrangle of bisectors}
\label{subsec:quadrangle}
For points $p_i\in\PV$ satisfying~\eqref{eq:relation-omega2},
denote $R_i\coloneq R^{p_i}$ and define the  complex geodesics 
$C_1\coloneq \mathbb{P}p_1^\perp\cap\HC$,  $C_2\coloneq R_2C_1$, $C_3\coloneq R_3C_2$, and $C_4\coloneq R_4C_3$. 
We impose four conditions~{\bf (Q1--Q4)} over these points to be defined bellow.

\begin{enumerate}
\item[{\bf (Q1)}] \underline{The complex geodesics $C_i$ are pairwise ultraparallel.}
\end{enumerate}

\begin{figure}[htb!]
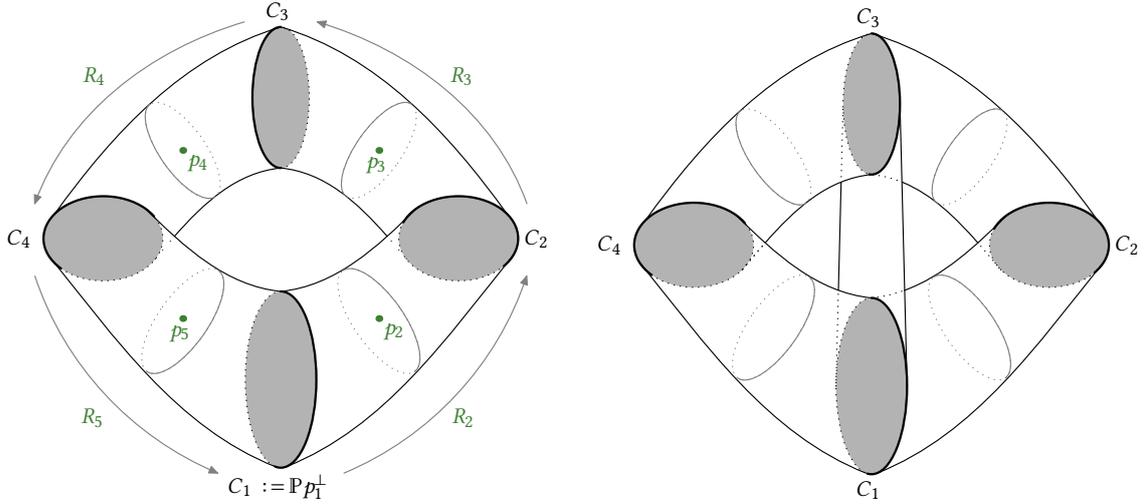

\centering
\includegraphics[scale=.85]{figs/defi-quad-1.mps}
\hspace{.4cm}
\includegraphics[scale=.85]{figs/triang-bisec.mps}
\caption{The construction of the quadrangle of bisectors and its decomposition
as two transversally adjacent triangles}
\label{fig:triang-bisec}
\end{figure}

If the condition {\bf(Q1)} is satisfied, we can consider the adjacent 
triangles of bisectors $\triangle(C_1,C_2,C_3)$ and $\triangle(C_1,C_3,C_4)$ (see Subsection~\ref{subsec:triangle}). 
\begin{enumerate}
\item[{\bf (Q2)}] \underline{The triangles of bisectors $\triangle(C_1,C_2,C_3)$ and
$\triangle(C_1,C_3,C_4)$ are transversal and counterclockwise}\\  
\underline{oriented.}
\end{enumerate}

As seen in Subsection~\ref{subsec:bisectors}, an oriented bisector~$B$ 
divides~$\CHC$ in two half-spaces. Following \cite{AGG2011},
we denote by $K^+$ the half-space lying on the side of the normal vector to~$B$,
while the other is denoted~$K^-$.

Given bisector segments $\mathrm B[S,S_1]$ and $\mathrm B[S,S_2]$ that are transversal along the complex
geodesic~$S$, consider the oriented bisectors $B_1\coloneq \mathrm B\geo{S,S_1}$ and 
$B_2\coloneq \mathrm B\geo{S,S_2}$. For $i=1,2$, denote by $K_i^+$ and $K_i^-$ the half-spaces 
defined by $B_i$, as above. We define the {\it sector\/} from
$\mathrm B[S,S_1]$ to $\mathrm B[S,S_2]$
as $K_1^+\cap K_2^-$ if the angle from $\mathrm B[S,S_1]$ to $\mathrm B[S,S_2]$ is smaller than~$\pi$
at every point~$c\in S$ (or, equivalently, at some point~$c\in S$);
otherwise, we define the sector as $K_1^+\cup K_2^-$.

\smallskip

Now, supposing that conditions {\bf (Q1)} and {\bf (Q2)} are satisfied
we have two transversal and counterclockwise oriented triangles $\triangle(C_1,C_2,C_3)$ and 
$\triangle(C_1,C_3,C_4)$. We say that these triangles are {\it transversally adjacent\/} if
$\mathrm B[C_2,C_3]$ and $\mathrm B[C_3,C_4]$ 
are transversal over~$C_3$, $\mathrm B[C_2,C_1]$ and $\mathrm B[C_1,C_4]$ are transversal
over~$C_1$, and $C_4$ lies in the {\it sector\/} from
$\mathrm B[C_1,C_2]$ to $\mathrm B[C_2,C_3]$. This will imply that the region delimited by 
the quadrangle~$\mathcal Q$ is contained in the sector from $\mathrm B[C_1,C_2]$
to $\mathrm B[C_2,C_3]$ and is also contained in the sector from 
$\mathrm B[C_3,C_4]$ to $\mathrm B[C_4,C_1]$.

\begin{enumerate}
\item[{\bf (Q3)}] \underline{The triangles $\triangle(C_1,C_2,C_3)$ 
and $\triangle(C_1,C_3,C_4)$ are transversally adjacent.}
\end{enumerate}

As in~\cite{turnover}, if the above conditions are satisfied, we can consider the 
{\it quadrangle of bisectors}
$$\mathcal Q\coloneq \mathrm B[C_1,C_2]\cup\mathrm B[C_2,C_3]
\cup\mathrm B[C_3,C_4]\cup \mathrm B[C_4,C_1].$$
In accordance with~\cite{yet}, $\mathcal Q$ is a {\it cornerless polyhedron}:
the faces (codimension~$1$ pieces) are the bisector segments 
$B[C_{i},C_{i+1}]$ and the edges (codimension~$2$ pieces) are the complex
geodesics $C_i$. Each edge is contained in exactly two faces.
The intersection of two faces is either empty or an edge. For~$\mathcal Q$
to be a cornerless polyhedron, {\it strong simplicity\/} (that is,
metric separability of nonintersecting edges and faces) is also required, and
it follows from the simple fact that nonintersecting faces also do not intersect
in $\CHC$ (see also~\cite[Section~2.2.2]{AGG2011}).

The face-pairing is provided by the reflections $R_i$, $i=2,3,4,5$,
and it identifies the sides as in Figure~\ref{fig:defi-quad-1}. There is one
{\it geometric cycle\/} of vertices, the cycle of $C_1$, which has 
length~$8$ and is given by the relation $R_5R_4R_3R_2R_5R_4R_3R_2=\omega^2$.

% If, for $i=1,2,3,4$, $q_i$ is polar points of $C_i$, then
% $q_1=p_1$ and $q_{i+1}=R_{i+1}q_i$. So, condition Q1 is equivalent 
% to $\tance(q_i,q_{i+1})>1$, $i=1,2,3$. If Q1 is satisfied we can
% consider the triangles of bisectors $\Delta(C_1,C_2,C_3)$ 
% and $\Delta(C_1,C_3,C_4)$.

\begin{figure}[htb!]
\centering
\includegraphics[scale=.85]{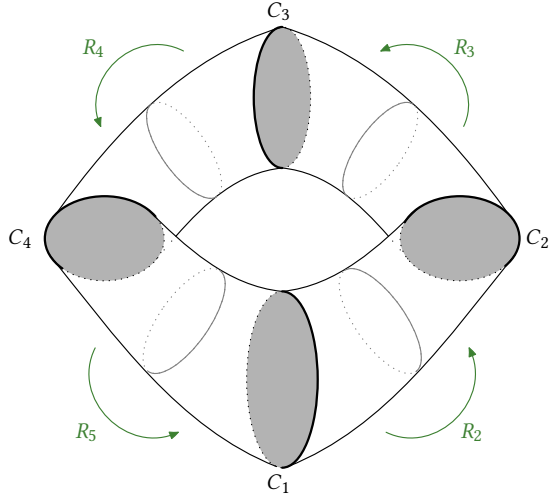}
\caption{The face-pairing of the quadrangle of bisectors}
\label{fig:defi-quad-1}
\end{figure}

Given $x_1\in C_1$, denote $x_2\coloneq R_2x_1$, $x_3\coloneq R_3x_2$, and $x_4\coloneq R_4x_3$.
So, for $i=1,2,3,4$, $x_i\in C_i$. The {\it sum of interior angles of\/~$\mathcal Q$
at\/~$x_1$} is the sum of the angle from
$\mathrm B[C_4,C_1]$ to $\mathrm B[C_1,C_2]$ at~$x_1$, the angle from
$\mathrm B[C_1,C_2]$ to $\mathrm B[C_2,C_3]$ at~$x_2$, the angle from
$\mathrm B[C_2,C_3]$ to $\mathrm B[C_3,C_4]$ at~$x_3$, and the angle from
$\mathrm B[C_3,C_4]$ to $\mathrm B[C_4,C_1]$ at~$x_4$.

\begin{enumerate} 
\item[{\bf(Q4)}] \underline{There exists $x_1\in C_1$ such that the sum of interior
angles of~$\mathcal Q$ at~$x_1$ is~$\pi$.}
\end{enumerate}

\begin{lemma}
\label{lemma:intangles}
If conditions\/~{\bf(Q1--Q3)} are satisfied, then for
any point $x_1\in C_1$, the sum of interior angles of\/~$\mathcal Q$ at~$x_1$
is equal to\/~$\pi\ \mathrm{mod}\ 2\pi$.
%If\/ \eqref{eq:relation-omega2} satisfies the conditions\/ {\bf (Q1--Q3)}, then there exists a points\/ 
%$c_1\in C_1$ such that the total angles of the geometric cycle of the vertex\/
%$C_1$ at\/ $c_1$ is\/ $2\pi$.
\end{lemma}

\begin{proof} 
As we observed in the Subsection~\ref{subsec:bisectors}, if~$x$ is a point in the extended 
oriented bisector $\mathrm B\geo{C_1,C_2}$, then the vector
\begin{equation}
\label{eq:normalbisec}
n(x,q_1,q_2) = -i \frac{\langle  - ,x \rangle}{\langle x,x \rangle}\bigg( \frac{\langle x, q_2 \rangle}{ \langle q_1,q_2 \rangle}q_1 -   \frac{\langle x, q_1 \rangle}{ \langle q_2,q_1 \rangle}q_2\bigg)
\end{equation}
is normal to~$\mathrm B\geo{C_1,C_2}$ at~$x$ (see Equation \eqref{eq:bisec-normal}).

Now, consider $x_1 \in C_1$, with $\langle x_1,x_1 \rangle = -1$,
and define $x_2,x_3,x_4$ by the formula $x_{k+1} \coloneq R_{k+1} x_k$. Note that 
$x_k\in C_k$. By~\eqref{eq:normalbisec}, the vectors 
$$n_1 = \langle -, x_1\rangle i \frac{\langle x_1,q_2 \rangle}{\langle q_1,q_2 \rangle}q_1 \quad \text{and} \quad n_1' = \langle -, x_1\rangle i \frac{\langle x_1,q_4 \rangle}{\langle q_1,q_4 \rangle}q_4$$
are, respectively, normal to the oriented 
bisector~$\mathrm B[C_1,C_2]$ at~$x_1$ and to the oriented 
bisector~$\mathrm B[C_1,C_4]$ at~$x_1$. If~$\theta_1$ is the inner angle at~$x_1$ between the bisectors, then 
\begin{equation}
\label{eq:angletheta1}
\theta_1 = \arg\langle n_1',n_1 \rangle = 
\arg\frac{\langle q_2,x_1 \rangle\langle x_1,q_4 \rangle}{\langle q_2,q_1\rangle\langle q_1,q_4\rangle},
\end{equation}
where $\arg:\CC \setminus \RR_{\leq 0} \to (-\pi,\pi)$ is given by $\arg\left(re^{i\theta}\right)=\theta$ for $r>0$ and $-\pi<\theta<\pi$. The fact that the quadrangle has transversal sides and counterclockwise orientation implies that the angle between two sides lies on $(0,\pi)$.

In the same way, if~$\theta_k$ is the angle at~$x_k$ between the bisector meeting 
at~$C_k$, then we have
\begin{equation}
\label{eq:anglestheta}
\begin{gathered}
\theta_2 = \arg\frac{\langle q_3,x_2 \rangle\langle x_2,q_1 \rangle}{\langle q_3,q_2\rangle\langle q_2,q_1\rangle},\quad 
\theta_3 = \arg\frac{\langle q_4,x_3 \rangle\langle x_3,q_2 \rangle}{\langle q_4,q_3\rangle\langle q_3,q_2\rangle},\\
\theta_4 = \arg\frac{\langle q_1,x_4 \rangle\langle x_4,q_3 \rangle}{\langle q_1,q_4\rangle\langle q_4,q_3\rangle}.
\end{gathered}
\end{equation}

For $k=1,2,3$, we have that $\langle q_k,q_{k+1} \rangle = -1 + 2\tance(q_k,p_{k+1})$ is a negative real number and $\langle x_k,q_{k+1} \rangle = \langle x_{k+1},q_{k} \rangle$. Additionally, 
\begin{align*}
    \langle q_1,q_4 \rangle &= \omega^2 (-1+2\tance(q_1,R_5q_1)),
    \\ \langle q_1,x_4 \rangle &= -\omega^{-2} \langle q_1,R_5x_1 \rangle,
    \\ \langle q_4,x_1 \rangle &= \omega^2 \langle q_1,R_5x_1 \rangle,
\end{align*}
 since $$x_4 = R_4R_3R_2x_1 = -R_4R_3R_2R_1 x_1 = -\omega^2 R_5 x_1$$
 and
 $$q_4 = R_4R_3R_2q_1 = R_4R_3R_2R_1 q_1 = \omega^2 R_5 q_1,$$
where we used $R_1 x_1 = -x_1$ and 
$R_1q_1 = q_1$.
Thus,
\begin{align*}
\theta_1+\theta_2+\theta_3+\theta_4 &= \arg
\frac{\langle x_1,q_2 \rangle\langle q_4,x_1 \rangle\langle x_2,q_3 \rangle
\langle q_1,x_2 \rangle\langle x_3,q_4 \rangle\langle q_2,x_3 \rangle\langle x_4,q_1 \rangle\langle q_3,x_4 \rangle}
{\langle q_1,q_2\rangle\langle q_4,q_1\rangle\langle q_2,q_3\rangle\langle q_1,q_2\rangle\langle q_3,q_4\rangle\langle q_2,q_3\rangle\langle q_4,q_1\rangle\langle q_3,q_4\rangle}\\
&=\arg\big( \langle q_4,x_1 \rangle \langle x_4,q_1 \rangle \big) - \arg \big(\langle q_4,q_1 \rangle^2 \big)\\
&= \arg\big(-\omega^4 |\langle q_1,R_5x_1 \rangle|^2\big)- 
\arg\big(\omega^4 |\langle q_1,R_5x_1 \rangle|^2\big) \\ 
&= \pi \mod 2\pi.
\end{align*}
\end{proof}

\begin{cor}
\label{cor: angles_sum_pi}
The sum of internal angles of\/~$\mathcal Q$
at\/~$x_1\in C_1$ is exactly\/~$\pi$ if
and only if \/~$\imag\langle x_1,q_2 \rangle\langle q_3,x_4\rangle>0$.
That is, assuming that the conditions \/~{\bf (Q1--Q3)} are satisfied, \/~{\bf(Q4)} holds
if, and only if,\/~$\imag\langle x_1,q_2 \rangle\langle q_3,x_4\rangle>0$.
\end{cor}

\begin{proof} 
%Up to choosing a representative for $x_1$, we have $\langle x_1,q_2 \rangle>0$.
Since we are assuming conditions {\bf (Q1--Q3)} are satisfied, 
internal angles $\theta_i$ are such that $0<\theta_i<\pi$. This, together with
Lemma~\ref{lemma:intangles}, implies that either~$\sum\theta_i=\pi$ or $\sum\theta_i=3\pi$.
From~\eqref{eq:anglestheta} in the proof of the previous lemma, 
choosing any  representatives for $x_1$ and for $p_1=q_1$, we 
have 
$$\theta_2+\theta_3=\arg\big(
\langle q_3,x_2\rangle\langle x_2,q_1\rangle \langle q_4,x_3\rangle
\langle x_3,q_2\rangle\big).$$ 
Thus, $0<\theta_2+\theta_3<\pi$
if and only if
$$0<\arg\big(
\langle q_3,x_2\rangle\langle x_2,q_1\rangle \langle q_4,x_3\rangle
\langle x_3,q_2\rangle\big)<\pi.$$
But,  
$\langle x_3,q_2\rangle=\langle x_2,q_3\rangle$
and $\langle x_2,q_1\rangle= \langle x_1,q_2 \rangle$. So, $0<\theta_2+\theta_3<\pi$
if and only if $0<\arg\big(\langle x_1,q_2\rangle
\langle q_3,x_4\rangle\big)<\pi$ or,
equivalently, $\imag\langle x_1,q_2\rangle\langle q_3,x_4\rangle>0$.
\end{proof}

\begin{thm}
Under the conditions \/ {\bf (Q1--Q4)}, the polyhedron delimited by the quadrangle\/ $\mathcal Q$
is a fundamental domain for the group generated by\/ 
$R_2,R_3,R_4,R_5$.
Thus, the representation\/ $\varrho:H_5\to\PU(2,1)$ given by\/
$\varrho(r_i)\coloneq R_i$ is discrete and faithful.
\end{thm}

\begin{proof}
We prove that for points $p_i$ satisfying \eqref{eq:relation-omega2} and {\bf (Q1--Q4)},
the quadrangle $\mathcal Q$ with facepairing isometries $R_i$ satisfies
the three conditions of~\cite[Theorem~3.2]{yet}, and therefore~$\varrho$
is discrete and faithful. 

Let us denote by $\mathbf P$ the polyhedron bounded by the quadrangle 
$\mathcal Q\coloneq s_1 \cup s_2 \cup s_3 \cup s_4$,
where $s_1\coloneq \mathrm B[C_1,C_2]$, $s_2\coloneq \mathrm B[C_2,C_3]$, $s_3\coloneq \mathrm B[C_3,C_4]$, $s_4\coloneq \mathrm B[C_4,C_1]$ are the sides of the polyhedron. The involution $R_{i+1}$ satisfy $R_{i+1}s_i=s_i$ and $\mathbf P \cap R_{i+1}\mathbf P = s_i$ for $i=1,2,3,4$. For a point $x_1$ on the vertex $C_1$ we have the adjacent sides $s_4,s_1$. 

\begin{rmk}
Whenever we write 
$a \diamond x \diamond b \stackrel{I}{\longrightarrow} a' \diamond x' \diamond b'$
we mean that $a,b$ are sides adjacent to the vertex containing $x$, that $a',b'$ are adjacent sides to the vertex containing $x'$, 
and that $Ib=a'$ and $Ix=x'$.
\end{rmk}

The cycle of the vertex $C_1$ for $x$ is given by
\[\begin{tikzcd}
	& {s_2 \diamond R_3R_2x_1 \diamond s_3} \\
	{s_3 \diamond R_4R_3R_2x_1 \diamond s_4} && {s_1 \diamond R_2x_1 \diamond s_2} \\
	& {s_4 \diamond x_1 \diamond s_1}
	\arrow["{R_4}"', from=1-2, to=2-1]
	\arrow["{R_5}"', from=2-1, to=3-2]
	\arrow["{R_3}"', from=2-3, to=1-2]
	\arrow["{R_2}"', from=3-2, to=2-3]
\end{tikzcd}\]
For details, see \cite[Section 3]{yet}.

\smallskip

Denoting $x_2\coloneq  R_2x_1$, $x_3\coloneq  R_3x_2$, $x_4\coloneq  R_4x_3$ we have that the inner angles at $x_1,x_2,x_3,x_4$ are respectively the angles $\theta_1,\theta_2,\theta_3,\theta_4$ 
seem in~\eqref{eq:angletheta1} and~\eqref{eq:anglestheta} of the proof of
Lemma~\ref{lemma:intangles}.

Since $\theta_1+\theta_2+\theta_3+\theta_4=\pi$, the geometric cycle of vertices has length~$8$, meaning that the polyhedrons 
\begin{equation*}
\begin{gathered}\mathbf P, \quad R_2 \mathbf P, \quad R_2R_3 \mathbf P, \quad R_2R_3 R_4 \mathbf P, \quad R_2R_3 R_4 R_5 \mathbf P,\\ \quad R_2R_3 R_4 R_5 R_2 \mathbf P, \quad R_2R_3 R_4 R_5 R_2 R_3 \mathbf P,\quad R_2R_3 R_4 R_5  R_2 R_3 R_4 \mathbf P
\end{gathered}
\end{equation*}
tessellate around the vertex $C_1$. Similarly, we have tessellations around each vertex $C_i$, thus providing a tessellation of the whole complex hyperbolic plane as a consequence of Poincar\'e's polyhedron theorem \cite[Theorem 3.2]{yet}. Therefore, the representation $\varrho$ is discrete.

From Poincar\'e's theorem, we also obtain that the group generated by the isometries $R_1,R_2,R_3,R_4,R_5$ coincides with $\langle R_2,R_3,R_4,R_5 \mid R_2^2=R_3^2=R_4^2=R_5^2=(R_5R_4R_3R_2)^2=1\rangle$. The relations $R_i^2=1$ and $(R_5R_4R_3R_2)^2=1$ follow from $R_i s_i=s_i$ and the tessellation around $C_1$, respectively. More precisely, we mean that $\varrho:H_5 \to \PU(2,1)$ is faithful.
\end{proof}

\subsection{The quadrangle conditions}
\label{subsec:quadrangleconditions}

We now reformulate conditions {\bf(Q1--Q3)} in terms of the Gram's matrix
of the polar points of complex geodesics~$C_i$.
Let $q_1\coloneq p_1$ and $q_i\coloneq R_iq_{i-1}$. Then, for $i=1,2,3,4$, $q_i\in\EV$ is
the polar point of~$C_i$, that is, $C_i=\mathbb Pq_i^\perp\cap\HC$.

Define
\begin{equation}
\begin{gathered}
t_{ij}\coloneq \sqrt{\tance(q_i,q_j)},\quad 
\eta_1\coloneq \frac{\langle q_1,q_2\rangle\langle q_2,q_3\rangle\langle q_3,q_1\rangle}
{\langle q_1,q_1\rangle\langle q_2,q_2\rangle\langle q_3,q_3\rangle},\quad
\eta_2\coloneq \frac{\langle q_1,q_3\rangle\langle q_3,q_4\rangle\langle q_4,q_1\rangle}
{\langle q_1,q_1\rangle\langle q_3,q_3\rangle\langle q_4,q_4\rangle},\quad
\\
\varepsilon\coloneq \frac{\eta_1}{|\eta_1|},\quad\varepsilon_0\coloneq \real\varepsilon,
\quad \varepsilon_1\coloneq \imag\varepsilon,\quad \chi\coloneq \frac{\eta_2}{|\eta_2|},
\quad\chi_0\coloneq \real\chi,\quad\chi_1\coloneq \imag\chi.
\end{gathered}
\end{equation}

 \begin{prop}
 \label{prop:equiv-q1}
 Condition\/ {\bf(Q1)} is equivalent to\/ $t_{ij}>1$ for any indices\/ $i\neq j$. 
 \end{prop}

 This proposition follows directly from Lemma~\ref{lemma:trianglelines}. The following one 
 is just the algebraic conditions for transversality in Subsection~\ref{subsec:triangle}.
 
\begin{prop}
\label{prop:equiv-q2}
Condition\/ {\bf(Q2)} is equivalent to\/
\begin{equation}
\label{eq:trans123}
\begin{gathered}
\varepsilon_1<0, \quad t_{12}^2\varepsilon_0^2+t_{23}^2+t_{31}^2
<1+2t_{12}t_{23}t_{31}\varepsilon_0,\\
t_{23}^2\varepsilon_0^2+t_{31}^2+t_{12}^2<1+2t_{12}t_{23}t_{31}\varepsilon_0,
\quad t_{31}^2\varepsilon_0^2+t_{12}^2+t_{23}^2<1+2t_{12}t_{23}t_{31}\varepsilon_0,
\end{gathered}
\end{equation}
and
\begin{equation}
\label{eq:trans134}
\begin{gathered}
\chi_1<0, \quad t_{13}^2\chi_0^2+t_{34}^2+t_{41}^2<1+2t_{13}t_{34}t_{41}\chi_0,\\
t_{34}^2\chi_0^2+t_{41}^2+t_{13}^2<1+2t_{13}t_{34}t_{41}\chi_0, 
\quad t_{41}^2\chi_0^2+t_{13}^2+t_{34}^2<1+2t_{13}t_{34}t_{41}\chi_0.
\end{gathered}
\end{equation}
\end{prop}

From~\eqref{eq:tranversalcond} in Subsection~\ref{subsec:triangle}, condition~\eqref{eq:trans123} means that 
$\triangle(C_1,C_2,C_3)$ is transversal and counterclockwise oriented and 
condition~\eqref{eq:trans134} means that $\triangle(C_1,C_3,C_4)$ is transversal
and counterclockwise oriented.

\smallskip

The following proposition concerning the condition {\bf(Q3)} is true due to the characterization for the transversality of bisectors following the inequality~\eqref{eq:bisec-transv} and the remark by the end of Subsection~\ref{subsec:bisectors}, stating that the normal vector field \eqref{eq:bisec-normal} points toward the region where the function defined 
in~\eqref{eq:bisector} is negative.

\begin{prop}
\label{prop:equiv-q3}
Condition\/ {\bf(Q3)} holds if and only if
\begin{equation}
\Big|\real\frac{\langle q_4,q_2\rangle\langle q_3,q_3\rangle}
{\langle q_4,q_3\rangle\langle q_3,q_2\rangle}-1\Big|<
\sqrt{1-\frac{1}{\tance(q_3,q_4)}}\sqrt{1-\frac{1}{\tance(q_3,q_2)}},
\end{equation}
%{\rm(}meaning that $\mathrm B\geo{C_3,C_2}$ and\/
%$\mathrm B\geo{C_3,C_4}$ are transversal along $C_3${\rm)},
\begin{equation}
\Big|\real\frac{\langle q_4,q_2\rangle\langle q_1,q_1\rangle}
{\langle q_4,q_1\rangle\langle q_1,q_2\rangle}-1\Big|<
\sqrt{1-\frac{1}{\tance(q_1,q_4)}}\sqrt{1-\frac{1}{\tance(q_1,q_2)}},
\end{equation}
%{\rm(}meaning that $\mathrm B\geo{C_1,C_2}$ and\/
%$\mathrm B\geo{C_1,C_4}$ are transversal along $C_1${\rm)}, 
and there exists $x\in C_1$ such that
\begin{equation}
\label{eq:interior-sector}
\imag\frac{\langle q_2,R_5x\rangle\langle R_5x,q_3\rangle}{\langle q_2,q_3\rangle}>0
\quad\text{and}\quad
\imag\frac{\langle q_1,R_5x\rangle\langle R_5x,q_2\rangle}{\langle q_1,q_2\rangle}>0.
\end{equation}
%{\color{BrickRed}{\rm(}meaning that $R_5x\in C_4$ belongs to the interior sector
%at $C_2$ of the triangle\/~$\triangle(C_2,C_3,C_1)${\rm)}.}
\end{prop}

% \begin{figure}[ht!]
% \centering
% \includegraphics[scale=.6]{figs/bend-quad-1.mps}
% \ \ 
% \includegraphics[scale=.6]{figs/bend-quad-2.mps}
% \end{figure}

\begin{thm}
\label{thm:bend-quad}
Consider a relation\/~\eqref{eq:relation-omega2} satisfying\/ {\bf(Q1--Q4)}.
Let\/ $i$ be an index 
and let $B_i(s)$ be the one-pa{\-}ram{\-}e{\-}ter subgroup given by bendings 
of\/ $R^{p_{i+1}}R^{p_i}$, indices considered modulo~$5$.
There exists\/ $\varepsilon>0$ such that, for every\/ 
$-\varepsilon<s<\varepsilon$,
the quadrangle obtained by substituting\/ $R^{p_{i+1}}R^{p_i}$
by $R^{B_i(s)p_{i+1}}R^{B_i(s)p_i}$ in\/~\eqref{eq:relation-omega2}
also satisfies\/ {\bf(Q1--Q4)} and thus provides a discrete and faithful 
representation\/ $\varrho_s:H_5\to\PU(2,1)$ that is not\/ $\mathbb R$ or\/ 
$\mathbb C$-Fuchsian. Moreover, distinct values of\/ $s$ provide 
distinct representations modulo conjugation.
\end{thm}
\begin{proof} The conditions stated above are given by strict inequalities. Thus, they hold for small bendings. Additionally, the sum of inner angles is $\pi$ modulo $2\pi$, consequently, not changed under bendings.
\end{proof}

\begin{cor}
\label{cor:dim-discrete}
Suppose\/ $\Sigma=(+,-,-,-,-)$ and\/ $\delta=\omega^2$. Then there exists an open\/
{\rm (}$4$-dimensional\/{\rm )} subset of\/ $\mathcal P_{\Sigma,\delta}H_5/\PU(2,1)$
consisting of discrete and faithful representations. Moreover, such open set
is bendings-connected.
\end{cor}

\section{The Toledo invariant}
\label{sec:toledo-invariant}
\begin{defi}  Let $\Gamma$ be a Fuchsian group and $\varrho: \Gamma \to \PU(2,1)$ be a representation. The {\it Toledo invariant\/} of~$\varrho$ is given by the integral
$$\tau \coloneq  \frac{4}{2\pi} \int_{\HH_\CC^1/\Gamma} f^\ast\varpi,$$
where $f: \HH_\CC^1 \to \HH_\CC^2$ is a $\Gamma$-equivariant map and $\varpi$ is 
the symplectic form of $\HH_\CC^2$, the imaginary part of the Hermitian metric.
\end{defi}

\begin{rmk}There is always a $\Gamma$-equivariant map $f:\HH_\CC^1 \to \HH_\CC^2$ and the Toledo invariant does not depend on its choice. See \cite[Lemma 34]{orbigoodles}.
\end{rmk}

Given a tessellation $\HH_\CC^2$ for the disc orbibundle using quadrangle of 
bisectors associated to a discrete and faithful representation $H_5 \to \PU(2,1)$ as 
described in Subsection~\ref{subsec:quadrangle}, 
we can take an arbitrary point 
%$q_1 \in C_1$ and define 
$x_1\in C_1$ and define
%$$q_2 = R_2 q_1, \quad q_3 = R_2 q_2, q_3 = R_3 q_2.$$
$x_2\coloneq  R_2x_1$, $x_3\coloneq R_3x_2$, $x_4\coloneq R_4x_3$.
It follows from~\eqref{eq:relation-omega2} that 
%$R_4 q_3 = q_1$.
$x_4=-\omega^2R_5x_1$. 
%More precisely, if we think 
%of $q_1$ 
%of~$x_1$ as a representative on the vector space~$V$ 
%satisfying $\langle q_1,q_1 \rangle =1$, then
%satisfying~{\color{BrickRed}$\langle x_1,x_1\rangle=-1$}, then
%$$R_4 q_3 = R_4R_3R_2 q_3 =  R_4R_3R_2 R_1 q_1 = \omega^2 q_1.$$
%{\color{RoyalBlue}$$R_5x_4=R_5R_4R_3R_2x_1=-R_5R_4R_3R_2R_1x_1=-\delta x_1.$$}

%\begin{figure}[H]
%\centering
%\includegraphics[scale = 1]{toledo domain.pdf}
%\label{Meridional curves along the quadrangle of bisectors}
%\end{figure}

\begin{figure}[htb]
\centering
\includegraphics[scale=.9]{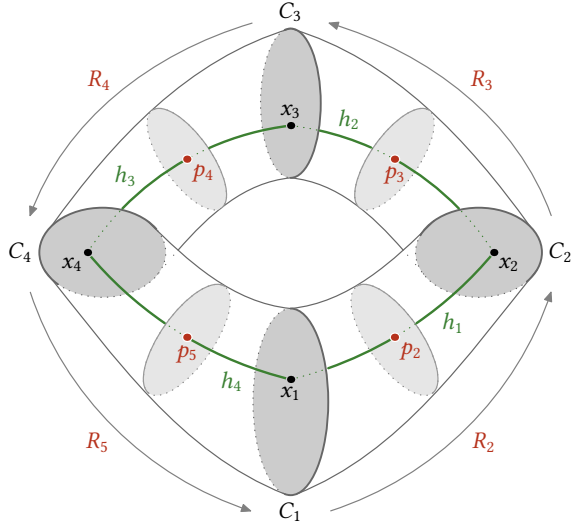}
\caption{Geodesics $h_i$ along the quadrangle of bisectors}
\label{fig:quad-h}
\end{figure}

Let $h_i$ be the geodesic segment connecting 
%$q_i$ to $q_{i+1}$. 
$x_i$ to $x_{i+1}$. Note that $h_i$ is a curve in $\mathrm B[C_{i},C_{i+1}]$.
In fact, the real slice of $\mathrm B[C_{i},C_{i+1}]$ defined by $x_i$ (is totally geodesic and) contains the real spine, which implies that it contains $p_{i+1}$; the geodesic through $x_i$ and $p_{i+1}$ contains $x_{i+1}$.
Note that the curve $h = h_1 \cup h_2 \cup h_3 \cup h_4$ is the boundary for a quadrilateral $F$ embedded on 
the interior of the fundamental domain. Via the tesselation we can extend the quadrilateral to a $G$-equivariant 
embedding $\HH_\CC^1 \to \HH_\CC^2$. Thus
$$\tau =\frac{4}{2\pi} \int_{F} \varpi.$$

The symplectic form is a closed $2$-form and therefore it admits a Kähler potential $P$ such that $\mathrm{d}P = \varpi$. Following Stokes Theorem, we obtain
$$\tau =\frac{4}{2\pi} \int_{h} P,$$
because $\partial F = h$. 
For each $c \in \HH_\CC^2$, the $1$-form $$P_c(t) = -\mathrm{Im}\frac{\langle t(x), c\rangle}{\langle x, c \rangle}, \quad t \in T_x\HH_\CC^2,$$
is a K\"ahler potential (see \cite[Appendix A.5]{AGG2011}).
Additionally, given $c_1,c_2 \in \HH_\CC^2$, we have
$$P_{c_1} - P_{c_2} = \mathrm{d} f_{c_1,c_2}$$
for $$f_{c_1,c_2}(x) = \frac12 \arg\frac{\langle c_1,x \rangle\langle x,c_2 \rangle}{ \langle c_1,c_2 \rangle}.$$

\begin{thm}
\label{thm:toledo}
For the disc orbibundles here constructed
$$\tau =\frac 23 \chi.$$
\end{thm}
\begin{proof}
Consider the 
%potential $P_{q_1}$. 
potential~$P_{x_1}$. Note that
%$$\int_{h_1} P_{q_1}=0 \quad{and}\quad \int_{h_4} P_{q_1}=0$$
$$\int_{h_1} P_{x_1}=0 \quad\text{and}\quad \int_{h_4} P_{x_1}=0,$$
because the curves $h_1,h_4$ are contained in real planes 
%containg $q_1$ and,
containing~$x_1$ and,
as consequence,
%$$\frac{\langle t(x), q_1\rangle}{\langle x, q_1 \rangle}$$
$\langle t(x), x_1\rangle/\langle x, x_1 \rangle$
is real for $x \in h_1 \cup h_4$ and $t$ tangent to the real planes. Therefore,
%$$\tau = \frac{4}{2\pi} \int_{h_2} P_{q_1} + \frac{4}{2\pi} \int_{h_3} P_{q_1}.$$
$$\tau = \frac{4}{2\pi} \int_{h_2} P_{x_1} + \frac{4}{2\pi} \int_{h_3} P_{x_1}.$$
Similarly, since $x_2 \in h_2$ and $x_3 \in h_3$ we have
$$\int_{h_2} P_{x_2} =\int_{h_3} P_{x_3} =0,$$ 
because 
\begin{gather*}\langle t(x), x_2\rangle/\langle x, x_2 \rangle \in \RR \quad \text{for}\quad x \in h_2\quad \text{and}\quad t \in T_xh_2,\\
\langle t(x), x_3\rangle/\langle x, x_3 \rangle \in \RR \quad \text{for}\quad x \in h_3\quad \text{and}\quad t \in T_xh_3.
\end{gather*}
Since 
%$$f_{q_1,q_3}(x) = \frac12 \mathrm{Arg}\frac{\langle q_1,x \rangle\langle x,q_2 \rangle}{ \langle q_1,q_2 \rangle}$$
$$f_{x_1,x_2}(x) = \frac12 \arg\frac{\langle x_1,x \rangle\langle x,x_2 \rangle}{ \langle x_1,x_2 \rangle},$$
satisfy $df_{x_1,x_2}=P_{x_1}-P_{x_2}$
we obtain
%\begin{align*}
% \int_{h_2} P_{q_1} = \int_{h_2} P_{q_1} - P_{q_2} &= \int_{h_2} \mathrm{d}f_{q_1,q_2}= \frac12 \mathrm{Arg}\frac{\langle q_1,q_3 \rangle\langle q_3,q_2 \rangle}{ \langle q_1,q_2 \rangle} - \frac12 \mathrm{Arg}\frac{\langle q_1,q_2 \rangle\langle q_2,q_2 \rangle}{ \langle q_1,q_2 \rangle} 
% \\&=  \frac12 \mathrm{Arg}\frac{\langle q_1,q_3 \rangle\langle q_3,q_2 \rangle}{ \langle q_1,q_2 \rangle} - \frac{\pi}2 \quad \mod\pi.   
%\end{align*}
\begin{multline*}
\int_{h_2} P_{x_1} = \int_{h_2} P_{x_1} - P_{x_2} = 
\int_{h_2} \mathrm{d}f_{x_1,x_2}= 
\frac12 \arg\frac{\langle x_1,x_3 \rangle\langle x_3,x_2 \rangle}{ \langle x_1,x_2 \rangle} - \frac12 \arg\frac{\langle x_1,x_2 \rangle\langle x_2,x_2 \rangle}{ \langle x_1,x_2 \rangle}=\\
= \frac12\arg\frac{\langle x_1,x_3 \rangle\langle x_3,x_2 \rangle}{ \langle x_1,x_2 \rangle} - \frac{\pi}2\mod\pi.   
\end{multline*}
Similarly,
%$$\int_{h_3} P_{q_1} = \int_{h_3} P_{q_1} - P_{q_3} = \frac12 \mathrm{Arg}\frac{\langle q_1,q_4 \rangle\langle q_4,q_3 \rangle}{ \langle q_1,q_3 \rangle} - \frac{\pi}2 \quad \mod\pi.$$
$$\int_{h_3} P_{x_1} = 
\int_{h_3} P_{x_1} - P_{x_3} = 
\frac12 \arg\frac{\langle x_1,x_4 \rangle\langle x_4,x_3 \rangle}{ \langle x_1,x_3 \rangle} - \frac{\pi}2\mod\pi$$
and, as consequence,
%$$\tau =\frac1{\pi} \mathrm{Arg}\frac{\langle q_1,q_3 \rangle\langle q_3,q_2 \rangle}{ \langle q_1,q_2 \rangle} +  \frac1{\pi} \mathrm{Arg}\frac{\langle q_1,q_4 \rangle\langle q_4,q_3 \rangle}{ \langle q_1,q_3 \rangle}  \quad \mod 2,$$
\begin{align*}
\tau &=\frac1{\pi} \arg\frac{\langle x_1,x_3 \rangle\langle x_3,x_2 \rangle}{ \langle x_1,x_2 \rangle} +  \frac1{\pi} \arg\frac{\langle x_1,x_4 \rangle\langle x_4,x_3 \rangle}{ \langle x_1,x_3 \rangle}\mod 2\\
&=\frac1{\pi} \arg\langle x_3,x_2 \rangle\langle x_2,x_1 \rangle +  \frac1{\pi} \arg\langle x_1,x_4 \rangle\langle x_4,x_3 \rangle\mod 2.
\end{align*}
%$$\tau =\frac1{\pi} \mathrm{Arg}\langle q_3,q_2 \rangle\langle q_2,q_1 \rangle +  \frac1{\pi} \mathrm{Arg}\langle q_1,q_4 \rangle\langle q_4,q_3 \rangle \quad \mod 2$$

Note that, choosing a representative for $x_1$ such that
$\langle x_1,x_1\rangle=-1$, we have $\langle x_i,x_i\rangle=-1$ for any~$i$ and
%\begin{align*}
%\langle q_2,q_1 \rangle &= \langle R_2 q_1,q_1 \rangle 
%\\ &= \Big\langle -q_1 + 2\frac{\langle q_1, p_2 \rangle }{\langle p_2, p_2 \rangle} p_2,q_1 \Big\rangle
%\\ &=1 - \mathrm{ta}(p_2,q_1) < 0.
%\end{align*}
\begin{align*}
\langle x_{i+1},x_i \rangle &= \langle R_{i+1} x_i,x_i \rangle 
\\ &= \Big\langle -x_i + 2\frac{\langle x_i, p_{i+1} \rangle }
{\langle p_{i+1}, p_{i+1} \rangle} p_{i+1},x_i \Big\rangle
\\ &=1 - 2\tance(p_{i+1},x_i) < 0,
\end{align*}
for $i=1,2,3$.
%Similarly, $\langle q_4,q_3 \rangle <0$ and $\langle q_3,q_2 \rangle$. 
This implies
%Thus,
%$$\tau =  \frac1{\pi} \mathrm{Arg}\langle q_1,q_4 \rangle  + \frac 12 \quad \mod 2.$$
$$\tau =  \frac1{\pi} \arg\langle x_1,x_4 \rangle  + 1\mod 2.$$
Additionally, from $x_4=-\omega^2R_5x_1$ we obtain
%$\langle q_1,q_4 \rangle = \omega \langle R_4 q_4,q_4 \rangle$, $\omega$ multiplied by a negative number.
$\langle x_1,x_4\rangle=-\overline{\omega}^2(1-2\tance(x_1,p_5))$ and therefore
$$\tau = \frac{2}{3} + 1\mod 2 = -\frac{1}{3} \mod 2.$$

On the other hand, $\HH_\CC^1/{H_5}$ is the $2$-orbifold $\Sigma_{5,2}$, a sphere with $5$ cone points of angle $\pi$. Thus, the Euler characteristics of $\HH_\CC^1/{H_5}$ is
$$\chi = \chi(\mathbb S^2) + \sum_{k=1}^5 \Big(-1+\frac12\Big) = 2 - \frac 52 = -\frac 12.$$
The Toledo inequality states that $|\tau| \leq |\chi|$. See \cite[Prop.~1.4]{Toledo1989} for surface groups and \cite[Thm.~44]{orbigoodles} for $2$-orbifold groups. Thus, since $$\tau \in -\frac13+ 2\mathbb Z =\Big\{ \cdots,-\frac{13}{3},-\frac 73,-\frac13, \frac53,\frac{11}3,\cdots\Big\} $$ and $|\tau|\leq \frac12$ we conclude that
$\tau = -\frac13.$
\end{proof}

\section{Computational results}
\label{sec:computational}
Recall that once we have a Hermitian form $\langle -,- \rangle$ with signature $-++$ on $\mathbb C^3$, the set of points $p \in \mathbb P_{\mathbb C}^2$ defined by $\langle p,p \rangle <0$ is denoted by $\BV$ (or $\HH_{\mathbb C}^2$), the set defined by $\langle p,p \rangle >0$ is denoted by $\EV$, and the remaining points form $\SV$ (see Subsection~\ref{subsec:hyperbolic-spaces}). We construct an explicit example of points $p_i\in\mathbb P_{\mathbb C}^2$ satisfying~\eqref{eq:relation-omega2} and~{\bf (Q1--Q4)}.

\subsection{The relation}
\label{subsec:the-relation}
Take $\pmb{\sigma}=(\sigma_1,\sigma_2,\sigma_3)\coloneq (+1,-1,-1)$ and 
$\tau\coloneq -2.22 - 3.845152793\, i$. Note that $\tau$ is the trace of an isometry of the form
$\omega^2R^{p_4'}R^{p_5'}$ where $p_4',p_5'\in\BV$ are points with $\tance(p_4',p_5')=1.36$.

The point $(s_1,s_2,s)$ given by $s_1=-0.615$, $s_2=1.36$ and $s=-0.823663831$ is in 
$\mathcal S_{\pmb\sigma,\tau}$, that is, it satisfies~\eqref{eq:surface} and~\eqref{eq:inequalities-surface}. 
Note that $s_2$ is the starting tance and this choice is inspired by the construction in~\cite{SashaGusevskii2007}.

The Gram matrix associated with~$s_1,s_2,s$ (via the construction in Subsection~\ref{subsec:relative}) is
$$
G = \left[\begin{smallmatrix}
  1 & 0.7842193570679061 & -0.9006228957613458 + 0.5255531007357581\,i \\
  0.7842193570679061 & -1 &1.1661903789690602 \\
  -0.9006228957613458 - 0.5255531007357581\,i \quad & 1.1661903789690602 \quad& -1
\end{smallmatrix}
\right],
$$

Now we consider $\mathbb C^3$ equipped with the Hermitian form given by~$G$, i.e.,
the Hermitian form $\langle z,w \rangle = z^T G\overline{w}$, where we see vectors of $\CC^3$ as columns. This implies that
\begin{equation*}
p_1 = \left[\begin{smallmatrix}
  1 \\
  0 \\
  0
\end{smallmatrix}
\right], \quad 
p_2 = \left[\begin{smallmatrix}
  0 \\
  1 \\
  0
\end{smallmatrix}
\right],\quad
p_3 = \left[\begin{smallmatrix}
  0 \\
  0 \\
  1
\end{smallmatrix}
\right]
\end{equation*}
are such that the Gram matrix of 
$(p_1,p_2,p_3)$ is $G$ and $\trace R^{p_3}R^{p_2}R^{p_1}=\tau$. Taking
\begin{equation*}
p_4=\left[\begin{smallmatrix}
  -1.418265301931986 + 0.47532199460124075\, i \\
  -0.4357001043898596 + 2.288308794951551\, i \\
 -1.113898813019385+2.06429519903981\, i
\end{smallmatrix}\right],\quad 
p_5=\left[\begin{smallmatrix}
 -0.5282839230176636+0.922498012587838\, i \\
 1.110665387479294+0.873137721694037\, i \\
  0.6529626910515587 + 0.961593899934676\, i
\end{smallmatrix}\right],
\end{equation*}
which satisfy $\tance(p_4,p_5)=1.36$, we verify by direct computation that 
$R^{p_5}R^{p_4}R^{p_3}R^{p_2}R^{p_1}=\exp(-2\pi i/3)\Id$, that is, we obtain points satisfying~\eqref{eq:relation-omega2}. More precisely, with respect to the max norm for matrices, we have $$\Vert R^{p_5}R^{p_4}R^{p_3}R^{p_2}R^{p_1}  - \exp(-2\pi i/3)\Id\Vert_{\infty}<7\times 10^{-14}.$$

\subsection{Verifying transversality conditions}
Denote, as in Subsection~\ref{subsec:quadrangle}, $R_i\coloneq R^{p_i}$, $q_1\coloneq p_1$, $q_2\coloneq R_2q_1$, $q_3\coloneq R_3q_2$, and $q_4\coloneq R_4q_3$.
Note that the points $q_i$ are all positive.

$\bullet$ Condition {\bf (Q1)}: the complex geodesics $C_i \coloneq  \mathbb{P}(q_i^{\perp}) \cap {\mathbb H}_{\mathbb C}^2$ are ultraparallel because $\tance(q_i,q_j) > 1$ for~$i\neq j$:
\begin{equation*}
\begin{gathered}
\tance(q_1,q_2) \simeq 4.97, \quad \tance(q_2,q_3) \simeq 10.74, \quad \tance(q_3,q_4) \simeq 10.56, \\ 
\tance(q_4,q_1) \simeq 4.97, \quad \tance(q_1,q_3) \simeq 4.93, \quad \tance(q_4, q_2)\simeq 48.21,
\end{gathered}
\end{equation*}
(see Proposition \ref{prop:equiv-q1}).

Now we verify the condition {\bf (Q2)}, following Proposition \ref{prop:equiv-q2}.

$\bullet$ Condition {\bf (Q2)}: Following the notation in Subsection~\ref{subsec:quadrangleconditions},
the triangles of bisectors $\triangle(C_1,C_2,C_3)$ is counterclockwise oriented and transversal due to Proposition~\ref{prop:equiv-q2}:
\begin{align*}
&\varepsilon_1 \simeq -0.87, \\
&1+2t_{12}t_{23}t_{31}\varepsilon_0 - t_{12}^2\varepsilon_0^2 - t_{23}^2 - t_{31}^2 \simeq 0.31,\\
&1+2t_{12}t_{23}t_{31}\varepsilon_0 - t_{12}^2 - t_{23}^2\varepsilon_0^2 - t_{31}^2 \simeq 4.64,\\
&1+2t_{12}t_{23}t_{31}\varepsilon_0 - t_{12}^2 - t_{23}^2 - t_{31}^2\varepsilon_0^2 \simeq 0.28.
\end{align*}
The same goes for the triangle of bisectors $\triangle(C_1,C_3,C_4)$:
\begin{align*}
&\chi_1 \simeq -0.87, \\
&1+2t_{13}t_{34}t_{41}\chi_0 - t_{13}^2\chi_0^2 -t_{34}^2 - t_{41}^2 \simeq 0.32,\\
&1+2t_{13}t_{34}t_{41}\chi_0 - t_{13}^2 -t_{34}^2\chi_0^2 - t_{41}^2 \simeq 4.55,\\
&1+2t_{13}t_{34}t_{41}\chi_0 - t_{13}^2 -t_{34}^2 - t_{41}^2\chi_0^2 \simeq 0.35.
\end{align*}

$\bullet$ Condition {\bf (Q3)}: Here we follow Proposition~\ref{prop:equiv-q3}. First we must verify that the triangles of 
bisectors described above are transversally adjacent, meaning that the 
bisectors $\mathrm B[C_2,C_1]$ and $\mathrm B[C_4,C_1]$ are transversal, $\mathrm B[C_2,C_3]$ and
$\mathrm B[C_4,C_3]$ are transversal, and the sector defined by the bisector rays $\mathrm B[C_2,C_1\rcurvyangle$ and $\mathrm B[C_2,C_4\rcurvyangle$ contains the vertex $C_4$.
The transversality at $C_1$ occurs because
$$
\Big|\real\frac{\langle q_4,q_2\rangle\langle q_3,q_3\rangle}
{\langle q_4,q_3\rangle\langle q_3,q_2\rangle}-1\Big|-
\sqrt{1-\frac{1}{\tance(q_3,q_4)}}\sqrt{1-\frac{1}{\tance(q_3,q_2)}} \simeq -0.05
$$
and the transversality at $C_4$ happens because
$$\Big|\real\frac{\langle q_4,q_2\rangle\langle q_1,q_1\rangle}
{\langle q_4,q_1\rangle\langle q_1,q_2\rangle}-1\Big|-
\sqrt{1-\frac{1}{\tance(q_1,q_4)}}\sqrt{1-\frac{1}{\tance(q_1,q_2)}} \simeq -0.78.
$$
We now verify that $C_4$ is in the sector defined by $C_2$ and its adjacent sides:
\begin{equation*}
\imag\frac{\langle q_2,R_5x\rangle\langle R_5x,q_3\rangle}{\langle q_2,q_3\rangle}\simeq 11.69 \quad\quad\text{and}\quad\quad
\imag\frac{\langle q_1,R_5x\rangle\langle R_5x,q_2\rangle}{\langle q_1,q_2\rangle} \simeq 8.01,
\end{equation*}
where $x$ is the point in $C_1$ given by 
\begin{equation}
\label{eq:pointx}
x = \left[\begin{smallmatrix}
 0.6233725425638523 + 0.3637653164269322\,i\\
  0  \\
  0.6921571120361996 
\end{smallmatrix}\right].
\end{equation}

Note that $R_5x$ is just a arbitrary point of $C_4$ and that we are verifying if it belongs to the described sector. The presence of this point in the sector implies that the entire triangle $\triangle(C_1,C_2,C_3)$ is in the sector (see \cite[Section~2.3.2]{AGG2011}).

\smallskip

$\bullet$ Condition {\bf (Q4)}: For $x \in C_1$ given by~\eqref{eq:pointx}, the angles between transversal bisectors at $x$, $R_2x$, $R_3R_2x$, and $R_4R_3R_2x = R_5x$ sum up to $\pi$.  
By Corollary~\ref{cor: angles_sum_pi}, the sum is indeed~$\pi$ because
$$\imag \langle x,q_2 \rangle \langle q_3, R_5x \rangle \simeq 6.36.$$

Alternatively, note that by Proposition \ref{lemma:intangles}, we know the sum of angles is $\pi\ \mathrm{mod}\ 2\pi$. Thus, by computing the numerical value of the sum of the inner angles at the described points and showing it to be equal to $\pi$ up to the precision of the machine, we conclude the sum of angles must be precisely $\pi$.

The Toledo invariant $\tau = \frac23 \chi$ can be computed directly as well. Instead of using Theorem \ref{thm:toledo}, we can use the formulas present in its proof and compute the Toledo invariant directly.

The quadrangle bounds a fundamental domain for the action of $H_5$ and we can extend the natural fibration by open discs of the quadrangle to the fundamental domain, where the disc fibers of the quadrangle constitute in slices of bisectors. From this configuration we obtain a disc orbibundle $\HH_\mathbb{C}^2/H_5 \to \HH_\mathbb{C}^1/H_5$. The Euler number of the disc orbibundle can be computed using the formula $3\tau = 2e+2\chi$, valid due to~\cite[Section~4]{interplay}. Thus the Euler number is zero. Alternatively, we compute the Euler number directly in 
Subsection~\ref{subsec:eulerorbi}.

\subsection{Disc orbibundle structure}\label{subsec:discorbibundle}
Consider the example outlined in Subsection~\ref{subsec:the-relation}. 
The respective quadrangle with vertices $C_1, C_2, C_3, C_4$ satisfies the tesselation conditions described in Section~\ref{sec:discreteness}, thus giving rise to a complex hyperbolic disc orbibundle over $\Sigma_{5,2}$, the sphere with $5$ cone points of angle~$\pi$. More precisely, the quadrangle of bisectors bounds a closed $4$-ball $K$ inside $\CHC$ and this closed $4$-ball is foliated by closed discs, the discs on the quadrangle being the slices of the segments of bisectors defining its sides. The region $K \cap \HC$ tesselates $\HC$, with the identification relations provided by the isometries $R_i$, which glue slice to slice on the quadrangle. Thus, the region $K \cap \HC$ forms a fundamental domain for the action of $H_5$ on $\HC$ and $\HC/H_5$ defines a complex hyperbolic disc orbibundle. 

We now explain why the fibration exists. Our fundamental domain $K$, bounded by the quadrangle of bisectors for $C_1, C_2, C_3, C_4$, can be broken into two regions bounded by the transversal counter-clockwise oriented triangles of bisectors $\triangle(C_1, C_2, C_3)$ and $\triangle(C_1, C_3, C_4)$. Additionally, these triangles are transversally adjacent. The bisector $\mathrm{B}\geo{C_1, C_3}$ separates the ball $\HC$ into two components, having each triangle in a different component. It is possible to deform each triangle $\triangle(C_1,C_2,C_3)$ and $\triangle(C_1,C_3,C_4)$ separately to new triangles of bisectors $\triangle(C_1',C_2',C_3')$ and $\triangle(C_1',C_3',C_4')$ arising from geodesic triangles over the complex spine of $\mathrm{B}\geo{C_1,C_3}$. 

More explicitly, if $L=\mathbb Pq^\perp\cap \HC$ is the complex spine of $\mathrm{B}\geo{C_1,C_3}$, then there are $c_1',c_2',c_3',c_4' \in L$ forming a convex geodesic polygon such that $C_i' = \mathbb P (\mathbb C c_i'+\mathbb C q) \cap \HC$, $C_1'=C_1$, and $C_3'=C_3$.
The deformations of the triangles of bisectors are done keeping the vertices $C_1, C_3$ fixed and moving $C_2, C_4$ until we obtain $C_2', C_4'$, thus defining triangles of bisectors at each stage of the deformation. Lemma 2.28 in \cite{AGG2011} guarantees that the space of transversal counterclockwise oriented triangles of bisectors is path-connected, meaning we can deform each triangle of bisectors freely. In \cite[7.2. Deformation Lemma]{turnover}, it is discussed in detail how to do it with both triangles as described above, keeping fixed a segment of bisector separating the triangles. 

These triangles of bisectors arising from the complex spine of $\mathrm{B}\geo{C_1, C_3}$ are naturally fibered by complex geodesics because for each $x$ on the convex geodesic polygon defined by $c_1',c_2',c_3',c_4'$ over $L$ we have the complex geodesic $\mathbb P(\mathbb C x+ \mathbb C q) \cap \HC$  through $x$ orthogonal to $L$. Thus, the quadrangle of bisector arising over $L$ is naturally fibered and its fibration can be deformed back to our original quadrangle.

\subsection{Brute force computation of the Euler number}
\label{subsec:eulerorbi} We start with two brief remarks that will be needed ahead: first we define the Euler number for our setting, then we discuss the relative position of triplets of points in the boundary of a complex geodesic.

\bigskip

\begin{rmk}[Euler number]\label{defi:euler number}
\noindent Consider an oriented disc orbibundle $L\to \Sigma$. Like an oriented disc bundle over a closed and oriented surface, an oriented disc orbibundle $L~\to~\Sigma$ over a closed and oriented $2$-orbibundle~$\Sigma$ can be seen as a rank $2$ oriented vector orbibundle, by taking a section and considering its normal orbibundle inside the total space. Associated to this vector orbibundle, we have an oriented circle orbibundle $\mathbb{S}^1(L) \to \Sigma$. We now outline the Euler number of such disc bundle: let $x_1,\ldots,x_k$ be the singular points of $\Sigma$. By removing small open discs $D_i$ centered at each $x_i$, we obtain the surface with boundary $\Sigma' = \Sigma \setminus \sqcup_i D_i$. If $\sigma$ is a section for the circle bundle restricted to $\Sigma'$ (the section $\sigma$ exists because $\Sigma'$ is homotopic to a graph and oriented circle bundles over graphs are trivial), we can consider the element $\sigma|_{\partial \Sigma'}$ in the homology group $H_1(\mathbb{S}^1(L),\mathbb Q)$. Let $s$ be an oriented fiber over a regular point of $\Sigma$ for $\mathbb{S}^1(L) \to \Sigma$. There exists a rational number $e$ such that
$$\sigma|_{\partial \Sigma'} = - e s$$
in $H_1(\mathbb{S}^1(L),\mathbb Q)$. This number $e$ is the Euler number of the orbibundle (see \cite{orbigoodles} for details).
\end{rmk}

\begin{rmk}[Cyclic order in the boundary of a complex geodesic]
%\noindent {\it Cyclic order in the boundary of a complex geodesic:}
We now outline how to distinguish if three points over the boundary of a complex geodesic are in cyclic order following the orientation of the boundary: consider three distinct points $\xi_1,\xi_2,\xi_3$ on the boundary $\partial C$ of a complex geodesic $C \subset \HH_\CC^2$. For any representatives of these points, the number 
$$\langle \xi_1,\xi_2 \rangle\langle \xi_2,\xi_3 \rangle\langle \xi_3,\xi_1 \rangle$$
has real part zero and the imaginary part has a well-defined sign. If the imaginary part is negative, the points $\xi_1,\xi_2,\xi_3$ are in the cyclic order, meaning that they are set in this order when going around the circle~$\partial C$ following the counterclockwise orientation. Otherwise, the imaginary part is positive and we say that they are not in cyclic order (see~\cite[Section~7.1]{Goldman1999} for more details).
\end{rmk}

Back to our example, we identify the fibers of $\mathbb{S}^1(L) \to S$ with the circles arising from extending the fibers of $K$ to the absolute $\partial\HC$. We will denote the complex hyperbolic disc orbibundle by $L \to \Sigma_{5,2}$ and its associated circle orbibundle by $\mathbb S^1(L) \to \Sigma_{5,2}$. The singular points of $\Sigma_{5,2}$, which we denote by $x_1,x_2,x_3,x_4,x_5$, have as fibers 
$\partial C_1$, $\partial M_1/R_2$, $\partial M_2/R_3$, $\partial M_3/R_4$ and $\partial M_4/R_5$, where~$M_i$ is the middle slice of the segment of bisector $\mathrm B[C_i,C_{i+1}]$.

Let $m$ be the polar for the middle slice of the bisector $\mathrm B[C_1,C_3]$:
$$m=\left[\begin{smallmatrix}
    0.5910828046793254+0.3412618163949675\,i \\
    0.309025718039677+0.5352482444901744\,i \\
    0.007168088802214577+0.840803229440756\,i \\
\end{smallmatrix}\right].$$

The isometries $R^m R_3R_2$ and $R_5 R_4 R^m$ when restricted to $C_1$ are hyperbolic. Indeed, for 
$$
z=\left[\begin{smallmatrix}
    0.956848389978639+1.131906097990544\,i \\
    2.616046110576799\\
    3.038930164212775-0.5165459094148785\,i \\
\end{smallmatrix}\right] 
$$
in $\partial C_1$ we have
\begin{align*}
\big\langle z,R^m R_3R_2 z \big\rangle\big\langle R^mR_3R_2 z, (R^mR_3R_2)^2 z\big\rangle \big\langle (R^m R_3R_2)^2 z,z\big\rangle &\simeq 0.56\,i,\\
\big\langle z,R_5 R_4 R^m z \big\rangle\big\langle R_5R_4 R^m z, (R_5 R_4 R^m)^2 z\big\rangle \big\langle (R_5 R_4 R^m)^2 z,z \big\rangle &\simeq-0.56\,i,
\end{align*}
meaning that $z$ is moving in the clockwise direction under action of $R^mR_3R_2$ and in the counterclockwise direction under action of $R_5R_4R^m$.

On the other hand, for
$$z=\left[\begin{smallmatrix}
    1.391513324127465+0.03028946901214607\,i \\
    -2.616046110576799\,i \\
    0.175611202621321-2.346773052176575\,i \\
\end{smallmatrix}\right]$$
in $\partial C_1$, we have
\begin{align*}
\big\langle z,R^m R_3R_2 z \big\rangle \big\langle R^m R_3R_2 z, (R^m R_3 R_2)^2 z\big\rangle \big\langle (R^{m} R_3 R_2)^2 z,z \big\rangle &\simeq -0.89\,i,\\
\big\langle z,R_5 R_4 R^m z \big\rangle \big\langle R_5 R_4 R^m z, (R_5 R_4 R^m)^2 z \big\rangle \big\langle (R_5 R_4 R^m)^2 z,z \big\rangle &\simeq 0.89 \,i,
\end{align*}
meaning that $z$ is moving in the counterclockwise direction under the action of $R^m R_3 R_2$ and in the clockwise direction under the action of $R_5 R_4R^m$.

Thus $R^m R_3 R_2$ and $R_5 R_4 R^m$ restricted to $C_1$ are hyperbolic. Since $R_5 R_4 R_3R_2|_{C_1} = \Id_{C_1}$, we have that $R^m R_3R_2$ and $R_5R_4R^m$ share their fixed points on $C_1$. Consider the fixed point
\begin{equation}
\label{eq:fixedpoint}
z_1 = \left[\begin{smallmatrix}
0.84066419047005-0.4449563543970761\,i \\
-2.04715035667857-1.628764154141047\,i \\
-1.465881589653703-1.056897533805411\,i \\
\end{smallmatrix}\right]
\end{equation}
for $R^m R_3R_2|_{C_1}$. More precisely, the vector $z_1$, found by diagonalyzing $R^m R_3R_2$, satisfy 
\begin{gather*}
|\langle z_1,q_1 \rangle|<1.2 \times 10^{-16}, \quad \langle z_1,z_1 \rangle <1.0 \times 10^{-15}\\
\Vert R^m R_3R_2z_1 -(-0.4689270359547878-0.2707351504387908\, i) z_1\Vert_\infty<2.6 \times 10^{-15},\\
\Vert R_5 R_4R^mz_1 -(-0.4689270359547878 +0.2707351504387908\, i)z_1\Vert_\infty<4 \times 10^{-14}.
\end{gather*}

We will use this point to construct a section over the regular points of the circle orbibundle $\mathbb S^1(L) \to \Sigma_{5,2}$. This is done by constructing a curve over the boundary of the quadrangle $\mathcal Q = \mathrm B[C_1,C_2]\cup \mathrm B[C_2,C_3]\cup \mathrm B[C_3,C_4]\cup \mathrm B[C_4,C_1]$ that can be glued under relations imposed by the reflections $R_2,R_3,R_4,R_5$ and that cross each boundary of complex geodesic fibering $\mathcal Q$ only once, except for the discs over singular points. 

Let us construct the section. Consider the fixed point~$z_1$ defined in~\eqref{eq:fixedpoint} (see Figure \ref{fig:section_strange_curve}). Recall that~$M_i$ is the middle slice of the bisector segment $\mathrm B[C_i,C_{i+1}]$. Connect $\partial C_1$ to $\partial M_1$ via a meridional curve starting at $z_1$. We denote its endpoint by $z_2$. Define $z_3\coloneq  R_2 z_2$ and let $z_4$ be the endpoint of the meridional curve starting at $z_3$ connecting $\partial M_1$ to $\partial C_2$. Note that $z_4 = R_2 z_1$. Similarly, we connect $z_4$ to $z_5$ through a meridional curve connecting $\partial C_2$ to $\partial M_3$ and define  $z_6 \coloneq  R_3 z_5$. The point $z_7$ is the end point of the meridional curve starting at $z_6$, connecting $\partial M_3$ to $\partial C_3$. In a similar procedure we construct $z_8$, $z_9$, $z_{10}$, $z_{11}$ and $z_{12}$. Note that $z_{10}=R_4 z_7$ and $z_1=R_5 z_{10}$. 

\begin{figure}[htb]
\centering
\includegraphics[scale=.9]{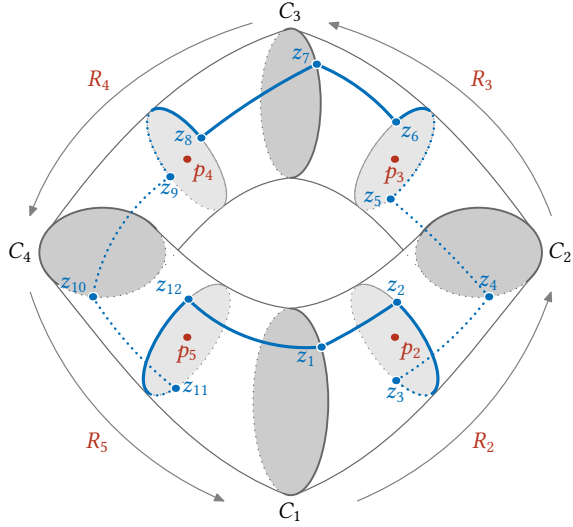}
\caption{The curve~$\gamma$ in {\color{RoyalBlue} blue} provides us with a section.}
\label{fig:section_strange_curve}
\end{figure}

We join $z_2$ to $z_3$ by an oriented circle segment in the clockwise direction. We also join $z_8$ to $z_9$ by an oriented circle segment in the clockwise direction. The oriented circle segments connecting $z_5$ to $z_6$ and $z_{11}$ to $z_{12}$ are taken in the counterclockwise orientation. We denote the constructed curve by $\gamma$, oriented in such a way it passes through the points $z_i$'s in the order of their indexes.

%\begin{wrapfigure}[24]{r}{0.45\textwidth}
\begin{figure}[htb]
\centering
\includegraphics[scale=.9]{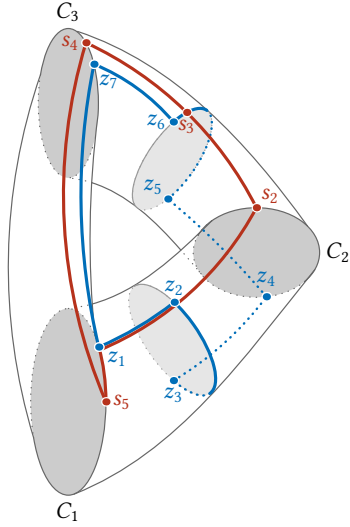}
%\ \ 
%\includegraphics[scale=.6]{section_strange_curve_right.pdf}
\caption{The curve in {\color{RoyalBlue}blue} is the part of $\gamma$ on the right triangle with the meridional curve $[z_7,z_1]$. The curve in {\color{BrickRed}red} is contractible in the solid torus bounded by $\partial\!\!\triangle(C_1,C_2,C_3)$ in $3$-sphere~$\partial \HH_\CC^2$.}
\label{fig:section_strange_curve_right}
\end{figure}
%\end{wrapfigure}

Note that the meridional curves are glued by the isometries $R_2,R_3,R_4,R_5$. The curve $[z_1,z_2]$ is glued to $[z_3,z_4]$ via $R_2$. The curve $[z_4,z_5]$ is glued to $[z_6,z_7]$ by $R_3$ and so on. The quadrangle $\mathcal Q$ bounds a fundamental domain $K$ and its fibration by complex geodesics is extended to a disc fibration to all $K$. Additionally, such fiber of~$K$ can be extended to~$\partial \HH_\CC^2$. 

In this way, $\partial_\infty K \coloneq K \cap \partial \HH_\CC^2$, the absolute of~$K$, is fibered by circles and forms a solid $2$-torus. We will show that~$\gamma$ is contractible in~$\partial_\infty K$, thus showing that there is a section over the regular points of the circle orbibundle associated with the disc orbibundle. Another way of phrasing this is the following: the $2$-torus $\partial \mathcal Q$ on the $3$-sphere $\partial\HH_\CC^2$ bounds a solid torus $\partial_\infty K$ fibered by circles arising from the discs fibering $K$. We will show that the curve $\gamma\subset \partial \mathcal Q$ is contractible on the solid torus $\partial_\infty K$.

By our choice of $z_1$, we have that $R^m z_7 = z_1$ and as consequence, the curve in {\color{RoyalBlue} blue} in Figure~\ref{fig:section_strange_curve_right} is closed. We now construct the curve in {\color{BrickRed}red}.

The curves $[z_1,s_2]$, $[s_2,s_3]$, $[s_3,s_4]$ and $[s_4,s_5]$ are obtained through meridional curves.

The polar points for $M_1,M_2,M_3,M_4$ are $m_1,m_2,m_3,m_4$:
\begin{equation*}
\begin{gathered}
m_1=\left[
\begin{smallmatrix}
    0.7868894753646337 \\
    0.617093958453955 \\
    0\\
\end{smallmatrix}
\right],\quad
m_2=\left[
\begin{smallmatrix}
    -0.6838620300713663\\
    -1.07259568309144\\
    -0.634948964320878-0.3594058103794568\, i \\
\end{smallmatrix}
\right],\\
m_3=\left[
\begin{smallmatrix}
    0.6844691229169399-1.088359820112815\, i \\
    -0.4031872153910329-0.82761658877558\, i \\
    0.1055037213393959-0.5230976120695976\, i \\
\end{smallmatrix}
\right],\quad
m_4=\left[
\begin{smallmatrix}
    1.278851231138038+0.963501401905471\, i \\
    1.362030724410985-0.4602994305090045\, i \\
    1.182655368126949-0.02112151207509668 \, i \\
\end{smallmatrix}
\right].
\end{gathered}
\end{equation*}

% \begin{alignat*}{2}
% &m_1=\left[
% \begin{smallmatrix}
%     0.7868892297 \\
%     0.6170942717 \\
%     0\\
% \end{smallmatrix}
% \right],\quad&&
% m_2=\left[
% \begin{smallmatrix}
%     -0.6838614226 \\
%     -1.07259561 \\
%     -0.6349509996 - 0.3594033706i \\
% \end{smallmatrix}
% \right]\\
% &m_3=\left[
% \begin{smallmatrix}
%     -0.6844794864 + 1.088386519i \\
%     0.4032154433 + 0.8276536802i \\
%     -0.1054881842 + 0.5231436697i \\
% \end{smallmatrix}
% \right],\quad && m_4=\left[
% \begin{smallmatrix}
%     1.278866677 + 0.9635422752i \\
%     1.362099703 - 0.460325163i \\
%     1.182713336 - 0.02112868687i \\
% \end{smallmatrix}
% \right].
% \end{alignat*}

Note that $s_2 = R^{m_1}z_1$, $s_4=R^{m_2}s_2$ and $s_5=R^{m}s_4$. The point $s_3$ is the intersection of the meridional curve $[s_2,s_3]$ with $\partial M_2$ of $\mathrm B[C_2,C_3]$. 

The isometry $R^{m}R^{m_2}R^{m_1}$ moves $z_1$ in the clockwise direction. Indeed, 
$$\big\langle z_1,R^{m}R^{m_2}R^{m_1}z_1 \big\rangle 
\big\langle R^{m}R^{m_2}R^{m_1}z_1, (R^{m}R^{m_2}R^{m_1})^2z_1 \big\rangle \big\langle (R^{m}R^{m_2}R^{m_1})^2z_1, z_1 \big\rangle \simeq 0.24\, i.$$

Connecting $s_5$ to $z_1$ with a circle segment in the counterclockwise direction, we obtain a contractible curve (see the curve in {\color{BrickRed}red} in Figure~\ref{fig:section_strange_curve_right}). This is true for counterclockwise-oriented transversal triangles of bisectors (see \cite[Theorem 2.24]{AGG2011}). 

\begin{rmk} A brief explanation of why the curve in {\color{BrickRed}red} is contractible in the absolute of the region bounded by the triangle of bisectors: consider an arbitrary coun{\-}ter{\-}clock{\-}wise-oriented and transversal triangle of bisectors with vertices $S_1,S_2,S_3$ and let $w_1,w_2,w_3$ be the polar of the middle slices of $\mathrm B[S_1,S_2]$, $\mathrm B[S_2,S_3]$, $\mathrm B[S_3,S_1]$, respectively. The isometry $I\coloneq R^{w_3}R^{w_2}R^{w_1}|_{S_1}$ is called holonomy at $S_1$ of the triangle $\triangle(S_1,S_2,S_3)$. Because the triangle $\triangle(S_1,S_2,S_3)$ is counterclockwise-oriented and transversal, the holonomy restricted to~$\partial S_1$ always has a non-empty component where its points move in the clockwise direction, meaning that any $s \in \partial S_1$ on this component satisfy
$$\imag \langle s, I s\rangle\langle Is, I^2s\rangle\langle I^2s,s \rangle >0.$$

Given a point $s_1$ on this component, we can construct the meridional curves $[s_1,s_2]$, $[s_2,s_3]$, $[s_1,s_4]$, where $s_{i+1} = R^{w_i}s_i$. The point $s_4$ is a point at $\partial S_1$ and, connecting~$s_4$ to~$s_1$ via a segment of
the circle $\partial S_1$ oriented in the 
counterclockwise direction, we obtain a closed curve. Additionally, the space of coun{\-}ter{\-}clock{\-}wise-oriented transversal triangles of bisectors is path-connected. Consequently, we can deform the triangle $\triangle(S_1,S_2,S_3)$ into a $\CC$-plane configuration, where the three vertices $S_1,S_2,S_3$ of the triangle of bisectors are perpendicular to a common complex geodesic. In this configuration, $R^{w_3}R^{w_2}R^{w_1}|_{\partial S_1}$ becomes a rotation in the clockwise direction with explicit angle, where it is easy to prove that the described curve is contractible because we can make the holonomy arbitrarily close to the identity by moving the vertices $S_1,S_2,S_3$ on this $\CC$-plane configuration (see \cite[Section~2.5.1]{AGG2011}).
\end{rmk}

Back to the Figure \ref{fig:section_strange_curve_right}, the point $s_3$ is over the circle segment connecting $z_5$ to $z_6$ in the counterclockwise direction because $$\langle z_5, s_3 \rangle \langle s_3, z_6 \rangle\langle z_6, z_5 \rangle \simeq -0.25\,i.$$  Thus, to go from $s_3$ to $z_6$ we must move in the counterclockwise direction. In the same way, to move from $s_5$ to $z_1$ we must move in the counterclockwise direction, thus the curves in {\color{RoyalBlue}blue} and {\color{BrickRed}red} in the Figure~\ref{fig:section_strange_curve_right} are homotopic in $\partial_\infty\!\! \triangle(C_1,C_2,C_3)$. Therefore, the curve in {\color{BrickRed}red} is contractible in the solid torus bounded by~$\partial_\infty \mathcal Q$.

%\begin{wrapfigure}[24]{l}{0.45\textwidth}
\begin{figure}[H]
\centering
\includegraphics[scale=.9]{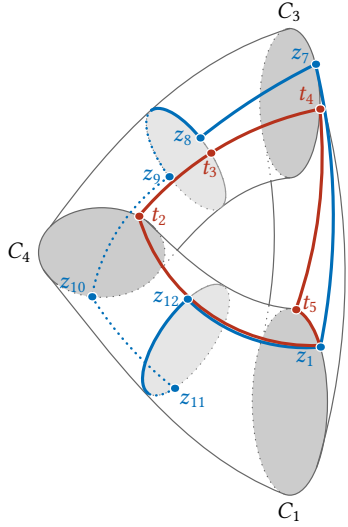}
%\ \ 
%\includegraphics[scale=.6]{section_strange_curve_left.pdf}
\caption{The {\color{RoyalBlue}blue} curve is the part of $\gamma$ on the left triangle joined with the meridional curve $[z_7,z_1]$. 
The {\color{BrickRed}red} curve is contractible in the solid torus bounded by $ \partial\!\!\triangle(C_1,C_3,C_4)$ in $\partial \HH_\CC^2$.}
\label{fig:section_strange_curve_left}
\end{figure}
%\end{wrapfigure}

Now we do the same procedure in the triangle $\triangle(C_1,C_3,C_4)$. Following the Figure~\ref{fig:section_strange_curve_left}, we construct the meridional curve $[z_1,z_7]$ and joining it with the part of the curve~$\gamma$ on the triangle $\triangle(C_1,C_3,C_4)$ we obtain the curve in {\color{RoyalBlue}blue}. By constructing the meridional curves $[z_1,t_2]$, $[t_2,t_4]$, $[t_4,t_5]$ and then connecting~$t_5$ to~$z_1$ with a circle segment in the counterclockwise direction we obtain the curve in {\color{BrickRed}red}. Following the order of points $t_5,t_4,t_2,z_1$, produces a closed curve, because $R^{m_4}R^{m_3}R^{m}$ maps $t_5$ to $z_1$ and the dynamics of $t_5$ under action of $R^{m_4}R^{m_3}R^{m}$ moves the point in the clockwise direction, because
$$\big\langle t_5, R^{m_4}R^{m_3}R^{m}t_5 \big\rangle \big\langle R^{m_4}R^{m_3}R^{m}t_5, (R^{m_4}R^{m_3}R^{m})^2t_5 \big\rangle \big\langle (R^{m_4}R^{m_3}R^{m})^2 t_5, t_5 \big\rangle \simeq  0.24\, i$$
has positive imaginary part. Therefore, the curve in red, following Figure \ref{fig:section_strange_curve_left}, is contractible in the solid torus bounded by $\partial\!\! \triangle(C_1,C_3,C_4)$ in the $3$-sphere $\partial \HH_\CC^2$.

The point $t_3$ appears as the intersection of the meridional curve $[t_2,t_4]$ with the middle slice $M_3$. Since
$$\langle t_3, z_8 \rangle  \langle z_8,z_9 \rangle\langle z_9, t_3\rangle \simeq 0.25\, i,$$
we have that $t_3$ is not over the circle segment going from $z_8$ to $z_9$ in the clockwise direction, because $t_3,z_8,z_9$ are not in cyclic order. Note that since to go from $t_3$ to $z_8$ we must move in the clockwise direction and, in the same way, to move from $t_5$ to $z_1$, we must move in the clockwise direction over the {\color{BrickRed}red} circle segment, the curves in {\color{RoyalBlue}blue} and {\color{BrickRed}red} are homotopic.

Therefore, we conclude that $\gamma$ is contractible in the solid torus $\partial_\infty K$.

By shrinking this curve inside the solid torus $\partial_\infty K$, we obtain a section $\sigma$ defined over the regular points for the circle bundle  $\mathbb{S}^1(L) \to \Sigma_{5,2}$. Recall that the $2$-orbifold $\Sigma_{5,2}$ is a sphere with $5$ conic points, $x_1,x_2,x_3,x_4,x_5$. These points arise from the discs $C_1,M_1,M_2,M_3,M_4$. The fibers over $x_1,x_2,x_3,x_4,x_5$ are $\partial C_1$, $\partial M_1/R_2$, $\partial M_2/R_3$, $\partial M_3/R_4$, $\partial M_4/R_5$. Remove small open discs $D_i$'s centered at the points $x_i$'s from the sphere $\Sigma_{5,2}$. The remaining surface $\Sigma$ is a surface with boundary. We use the section $\sigma|_{\partial \Sigma}$ to computes the Euler number. Let $s$ be the fiber of the circle orbibundle over a regular point, oriented in the counterclockwise orientation. 
In the homology group $H_1(\mathbb S^1(L),\mathbb Q)$ we have the identity
$$\sigma|_{\partial \Sigma} = -\partial M_1/R_2 + \partial M_2/R_3 - \partial M_3/R_4+\partial M_4/R_5 = -\frac12 s +\frac12 s-\frac12 s+\frac12 s=0.$$

Indeed, following the {\color{RoyalBlue}blue} curve in Figure~\ref{fig:section_strange_curve}, the circle segment $z_2z_3$, which follows the clockwise direction, becomes $ -\partial M_1/R_2$ in homology when viewed in the circle orbibundle. Similarly, the circle segment $z_5z_6$ follows the counterclockwise direction, thus in the circle orbibundle this curve becomes $\partial M_2/R_3$. In the same way, the circle segments $z_8z_9$ and $z_{10}z_{11}$ become $- \partial M_3/R_4$ and $\partial M_4/R_5$. Additionally, observe that $\partial M_i/R_{i+1} = \frac12s$ for $i=1,2,3,4$.

Thus, the Euler number $e=0$.

\subsection{Computational tests on bending the relation}
\label{subsec:comp-bend}

By Corollary~\ref{cor:dim-discrete}, there is a $4$-dimensional open neighborhood of
the representation constructed in Section~\ref{sec:computational} composed by 
discrete and faithful representations.
And, as stated in Theorem~\ref{thm:bend-quad}, this neighborhood is given by 
small bending-deformations of the representation that preserves discreteness and faithfulness.

To illustrate this fact, here we investigate what happens with the discreteness
and faithfulness of the representation constructed explicitly 
in Subsection~\ref{subsec:the-relation} under the action of bendings of 
$R^{p_2}R^{p_1}$ and $R^{p_3}R^{p_2}$.

\begin{rmk}
Let $p_1,p_2\in\PV$ be nonisotropic points, not both positive. Then $R^{p_2}R^{p_1}$
is loxodromic and it fixes two isotropic points $v_1,v_2\in\SV$ and a positive 
point~$c\in\EV$. A bending $B(\theta)$, for $\theta\in\mathbb R$, of $R^{p_2}R^{p_1}$ 
is written as
\begin{equation}
\label{eq:bend-matrix}
B(\theta)=\left[
\begin{matrix}
1 & 0 & 0 \\
0 & e^\theta & 0 \\ 
0 & 0 & e^{-\theta}
\end{matrix}
\right].
\end{equation}
in the basis $c,v_1,v_2$.
\end{rmk}

To do this, we look at the strongly regular triple $(p_1,p_2,p_3)$ and the corresponding point
$(s_1,s_2,s)$ in the surface~$\mathcal S_{\pmb\sigma,\tau}$ of Subsection~\ref{subsec:the-relation}.
We have that, as described in Subsection~\ref{subsec:relative}, 
bendings of $R^{p_2}R^{p_1}$ move the point $(s_1,s_2,s)$ over~$\mathcal S_{\pmb\sigma,\tau}$
preserving~$s_1$, while bendings of $R^{p_3}R^{p_2}$ move $(s_1,s_2,s)$ preserving~$s_2$.

Using the previous remark, for values of $\theta$ starting with $\theta=0$, and using
increments of $\Delta\theta=0.02$ (for ``positive'' bendings) or decrements of $\Delta\theta=-0.02$
(for ``negative'' bendings), we apply the bending~$B(\theta)$ of $R^{p_2}R^{p_1}$ as 
in~\eqref{eq:bend-matrix} (in the appropriate basis) to our 
representation and, in each step of these small bending deformations, we check the quadrangle conditions of 
Subsection~\ref{subsec:quadrangle}, marking the obtained
point in the surface in {\color{RoyalBlue} blue} if the conditions are satisfied (in which
case we have a discrete and faithful representation) or in {\color{Red} red} if the they are
not (in which case we can say nothing about the discreteness/faithfulness of the representation).
This gives us Figure~\ref{fig:surf-bend-1}.
\begin{figure}[H]
\centering
\includegraphics[scale=.65, trim=.5cm .5cm 2cm 2.5cm, clip]{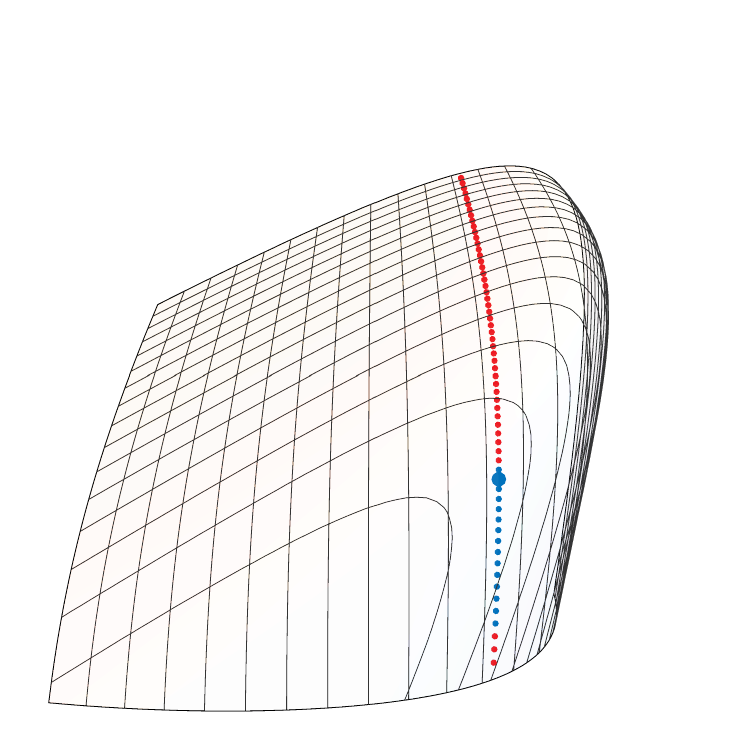}\hspace{1cm}
\includegraphics[scale=.57]{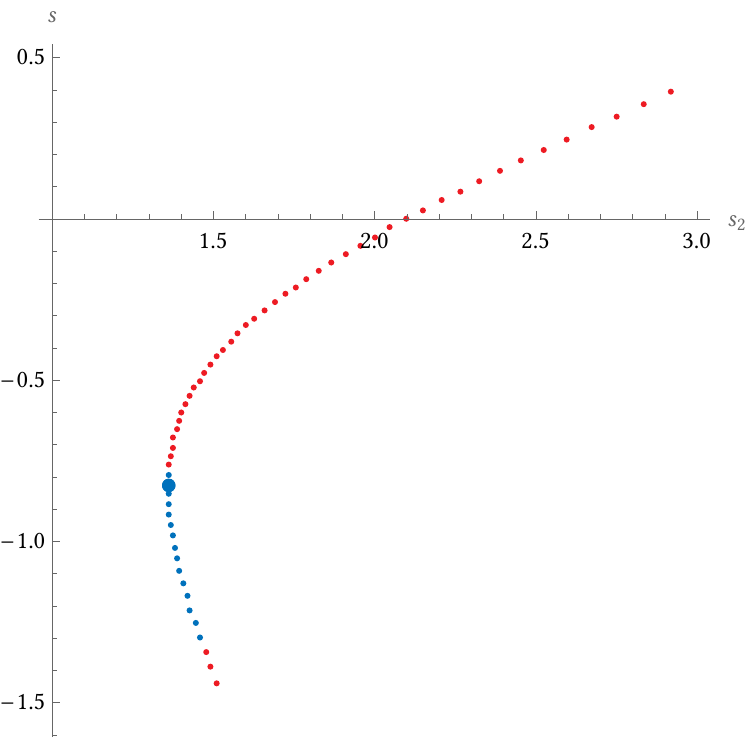}
\caption{Point $(s_1,s_2,s)$ in $\mathcal S_{\pmb{\sigma},\tau}$ and the interval 
corresponding to bendings of $R^{p_2}R^{p_1}$ producing discrete representations}
\label{fig:surf-bend-1}
\end{figure} 

Thereafter, we do the same thing, but now considering bendings of $R^{p_3}R^{p_2}$, 
and obtain Figure~\ref{fig:surf-bend-2}.

\begin{figure}[H]
\centering
\includegraphics[scale=.65, trim=.5cm .5cm 2cm 2.5cm, clip]{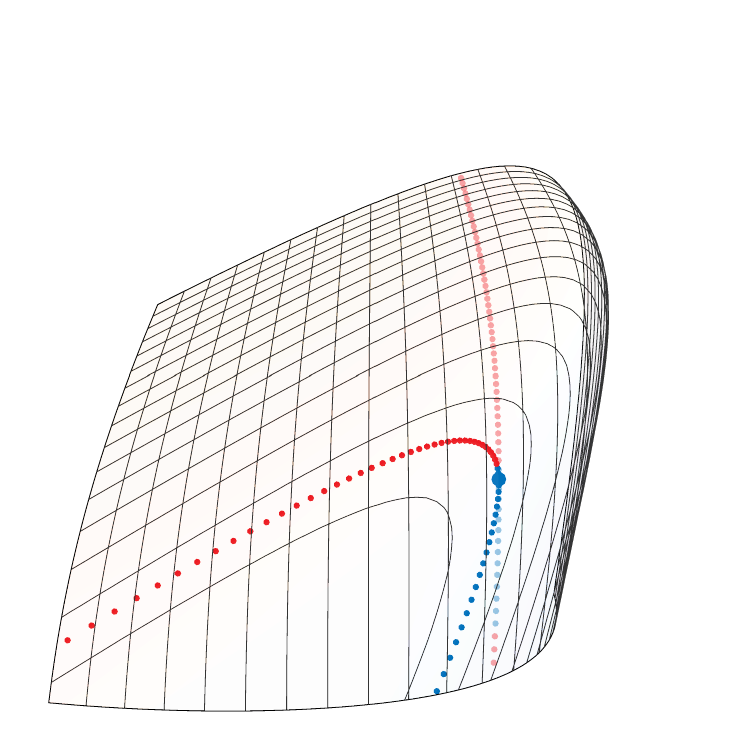}\hspace{1cm}
\includegraphics[scale=.5]{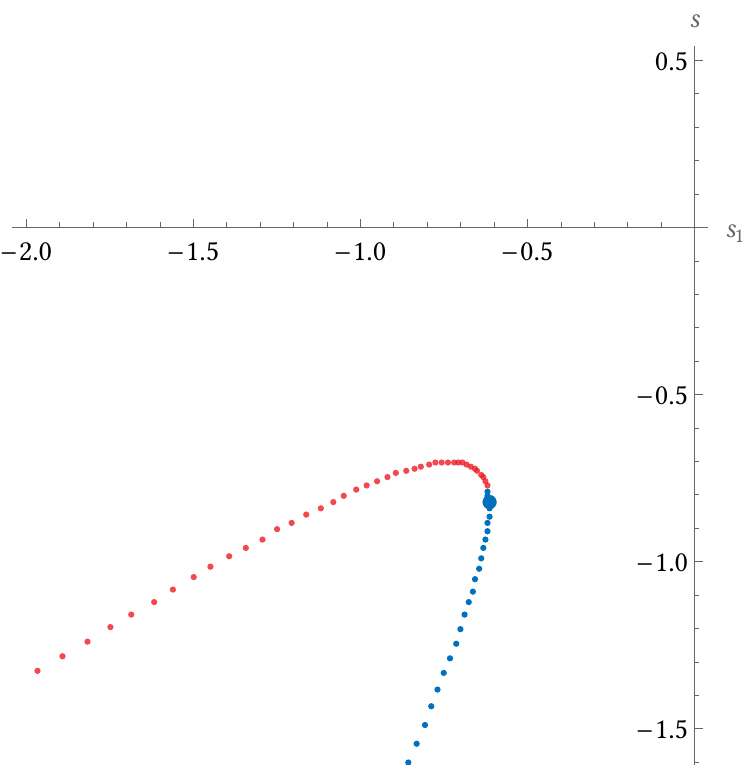}
\caption{Point $(s_1,s_2,s)$ in $\mathcal S_{\pmb{\sigma},\tau}$ and the interval 
corresponding to bendings of $R^{p_3}R^{p_2}$ producing discrete representations}
\label{fig:surf-bend-2}
\end{figure} 

\smallskip

Note that the vertical curve on the left of Figure~\ref{fig:surf-bend-1} and the horizontal curve 
on the left of Figure~\ref{fig:surf-bend-2} are almost tangent. In Figure~\ref{fig:curve-c},
we plot the curve $C$ in $\mathcal S_{\pmb\sigma,\tau}$ given by the points where
the vertical and horizontal lines are tangent, together with the point
$(s_1,s_2,s)$ in Subsection~\ref{subsec:the-relation}, that where
obtained via the construction in~\cite{SashaGusevskii2007}.

\begin{figure}[H]
\centering
\includegraphics[scale=.65, trim=.5cm .5cm 2cm 2.5cm, clip]{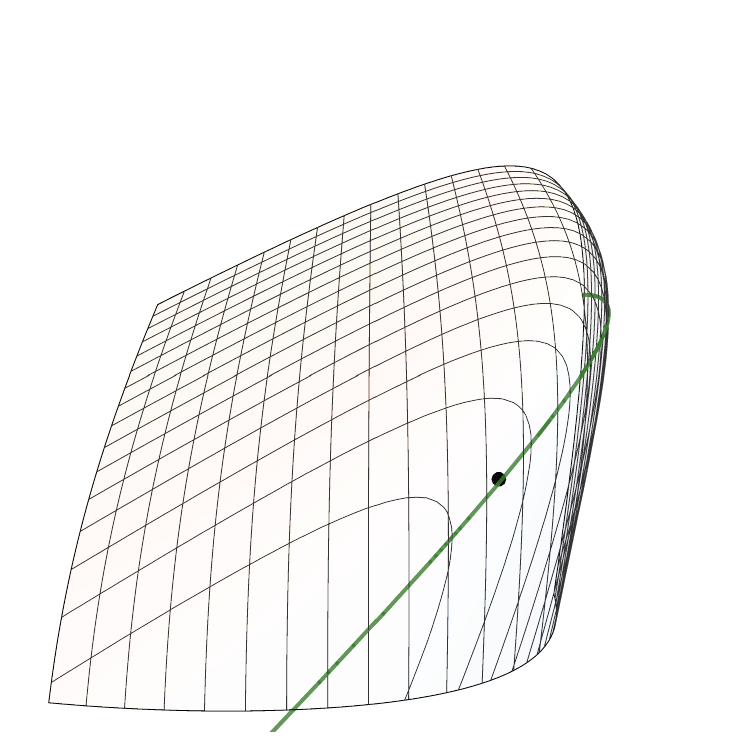}
\caption{The curve $C$ in {\color{OliveGreen} green} and the point $(s_1,s_2,s)$ given 
in Subsection~\ref{subsec:the-relation}}
\label{fig:curve-c}
\end{figure}

\smallskip

Lastly, consider the points $p_1,p_2,p_3,p_4,p_5$ constructed in
Subsection~\ref{subsec:the-relation} which determines a representation
$\varrho\in\mathcal P_{\Sigma,\delta}$, where $\Sigma\coloneq (+,-,-,-,-)$ and $\delta\coloneq \omega^2$.
As briefly mentioned in Subsection~\ref{subsec:relative}, 
it is possible that after the composition of two consecutive bendings of the 
strongly regular triple~$(p_1,p_2,p_3)$ we end up with a distinct triple of points
$(p_1',p_2',p_3')$ that is geometrically equal to $(p_1,p_2,p_3)$. In summary,
$(p_1,p_2,p_3)$ and its bending deformation $(p_1',p_2',p_3')$ are geometrically equal and so
they correspond to point in the surface $\mathcal S_{\pmb\sigma,\tau}$, where
$\pmb\sigma=(+,-,-)$ and $\tau\coloneq -2.22 - 3.845152793\, i$, but are distinct triple of points.
This implies that, the lists of points $p_1',p_2',p_3',p_4,p_5$ and $p_1,p_2,p_3,p_4,p_5$ are geometrically distinct, and associated with 
the former we have a representation
$\varrho'\in\mathcal P_{\Sigma,\delta}$ which is a bending deformation of
$\varrho$ and is distinct from $\varrho$ modulo~$\PU(2,1)$.
In Figure~\ref{fig:hol}, the big {\color{RoyalBlue} blue} dot corresponds to the
starting strongly regular triple $(p_1,p_2,p_3)$ and the points in the 
vertical line through it corresponds to the bendings of $R^{p_2}R^{p_1}$
obtained in~Figure~\ref{fig:surf-bend-1}. Over
this vertical line, we take a big {\color{Red} red} dot 
(so we are at a representation that fails the quadrangle conditions) and
mark points over the horizontal lines through this point, following
the process of increments/decrements of $\Delta\theta=\pm 0.02$ as above.
Finally, we notice that we obtain points that are simultaneously {\color{RoyalBlue} blue}
and {\color{Red} red}.

\begin{figure}[H]
\centering
\includegraphics[scale=.65, trim=.5cm .5cm 2cm 2.5cm, clip]{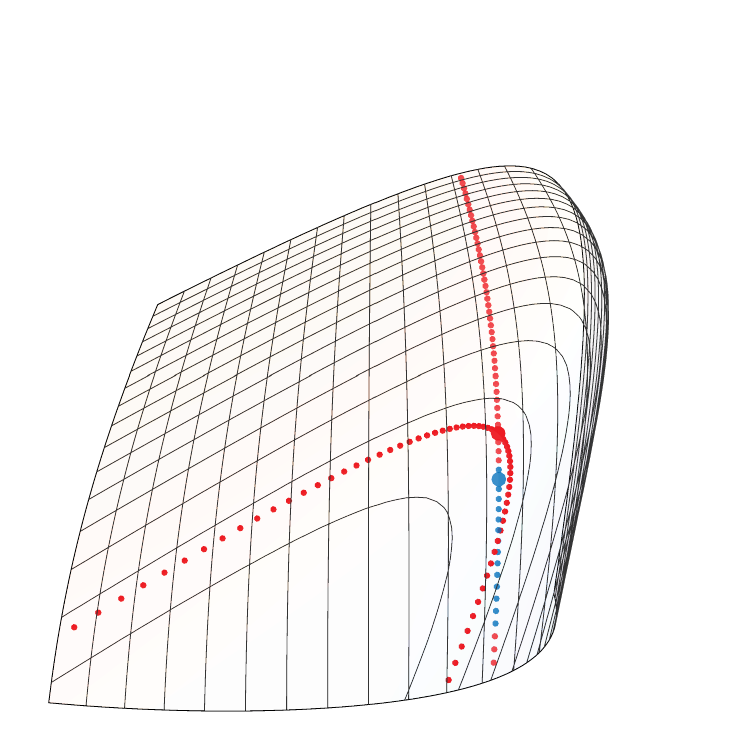}
\caption{A closed path in $\mathcal S_{\pmb\sigma,\tau}$ obtained via bendings where the tessellation conditions break along the way.}
\label{fig:hol}
\end{figure}

\bibliographystyle{alpha}
\bibliography{references-quad-refl}

\bigskip

\noindent
{\sc Hugo C. Bot\'os}

\noindent
{\sc Instituto de Matem\'atica e Estat\'istica, IME, Universidade de S\~ao Paulo, S\~ao Paulo/SP, Brasil}

\noindent
\url{hugocbotos@usp.br}

\bigskip

\noindent
{\sc Felipe A.~Franco}

\noindent
{\sc Centro de Matem\'atica, Computa\c c\~ao e Cogni\c c\~ao, CMCC, Universidade Federal do ABC, Santo Andr\'e/SP, Brasil}

\noindent
\url{f.franco@ufabc.edu.br}

\end{document}